\documentclass[11pt,twoside]{article}
\usepackage{times}
\usepackage{amsmath,amssymb,amsthm}
\usepackage{enumerate}
\usepackage{cite}
\usepackage{graphicx}
\usepackage{epstopdf}
\usepackage{color}

\allowdisplaybreaks

\def\beq{\begin{equation}}
\def\eeq{\end{equation}}
\def\ben{\begin{eqnarray}}
\def\een{\end{eqnarray}}
\def\bec{\begin{cases}}
\def\eec{\end{cases}}
\def\ss{\sum\limits}

\newcommand{\no}{\nonumber}

\newcommand{\R}{\mathbb R}

\newcommand{\p}{\partial}
\newcommand{\ve}{\varepsilon}
\newcommand{\f}{\frac}
\newcommand{\na}{\nabla}
\newcommand{\la}{\lambda}

\newcommand{\al}{\alpha}
\renewcommand{\t}{\tilde}

\newcommand{\vp}{\varphi}
\renewcommand{\O}{\Omega}
\renewcommand{\th}{\theta}
\newcommand{\g}{\gamma}

\newcommand{\dl}{\delta}

\newcommand{\ds}{\displaystyle}
\newcommand{\opdiv}{\operatorname{div}}
\newcommand{\RN}[1]{\textup{\uppercase\expandafter{\romannumeral#1}}}
\newcommand{\be}{\begin{equation}}
\newcommand{\ee}{\end{equation}}
\newcommand{\bea}{\begin{eqnarray}}
\newcommand{\eea}{\end{eqnarray}}
\newcommand{\md}{\mathrm{d}}

\newcommand{\ra}{[1,\infty)\times\mathbb{R}^2}

\allowdisplaybreaks

\newcommand{\crit}{\textup{crit}}
\newcommand{\conf}{\textup{conf}}

\theoremstyle{plain}
\newtheorem{theorem}{Theorem}[section]

\newtheorem{lemma}[theorem]{Lemma}

\theoremstyle{definition}

\theoremstyle{remark}
\newtheorem{remark}{Remark}[section]

\numberwithin{equation}{section}

\title{Global small data weak solutions of 2-D semilinear wave equations
 with  scale-invariant  damping, I}

\author{Li Qianqian$^1$\quad  Wang Dinghuai$^2$\quad Yin Huicheng$^{1,}$\footnote{Li Qianqian (\texttt{214597007@qq.com}) and Yin Huicheng
    (\texttt{huicheng@} \texttt{nju.edu.cn}, \texttt{05407@njnu.edu.cn}) are supported by the NSFC
    (No.~12331007) and by the National key research and development program of China (No. 2020YFA0713803).
    Wang Dinghuai (\texttt{Wangdh1990@126.com}) is supported by NSFC (No. 12101010), Natural Science Foundation of Anhui Province (No. 2108085QA19) and Key Scientific Project of Higher Education Institutions in Anhui Province
    (No.2023AH050145).}\vspace{0.5cm}\\
\small
1.  School of Mathematical Sciences, Nanjing Normal University, Nanjing 210023, China.\\
\small
2. School of Mathematics and Statistics, Anhui Normal University, Wuhu 241002, China.\\
}
\vspace{0.5cm}

\begin{document}

\date{}

\maketitle
\thispagestyle{empty}

\begin{abstract}
There is an interesting open question:  for the $n$-D ($n\ge 1$) semilinear wave equation with  scale-invariant  damping
$\p_t^2u-\Delta u+\f{\mu}{t}\p_tu=|u|^p$, where $t\ge 1$,  $p>1$ and $\mu>0$, the global small data weak solution $u$ will exist
when $p>p_{crit}(n,\mu)=\max\{p_s(n+\mu), p_f(n)\}$ with $p_{s}(n+\mu)=\f{n+\mu+1+\sqrt{(n+\mu)^2+10(n+\mu)-7}}{2(n+\mu-1)}$
and $p_f(n)=1+\f{2}{n}$. It is noticed that the weak solution $u$ can blow up in finite time when $1<p\le p_{crit}(n,\mu)$.
In addition, for $n=1$, this open question has been solved
recently. We now systematically solve this open problem for $n=2$.
As the first part, in the present paper, the global small solution $u$ is established for $p_{s}(2+\mu)<p<p_{conf}(2,\mu)=\f{\mu+5}{\mu+1}$
and $\mu\in(0,1)\cup(1,2)$. Our main ingredients are to find the suitable conformal power $p_{conf}(2,\mu)$
and derive some new kinds of
spacetime-weighted $L^{q}_tL^{q}_x([1, \infty)\times \R^2)$ or $L^q_tL^\nu_rL^2_{\theta}([1, \infty)\times [0, \infty)\times [0, 2\pi])$
Strichartz  estimates for the solutions of linear generalized Tricomi equation $\p_t^2v-t^m\Delta v=F(t,x)$
($m>0$). In forthcoming papers, we shall show the global existence of small solution $u$ for
the remaining cases of $p>1$ and $\mu>0$.
\end{abstract}

\noindent
\textbf{Keywords.} Critical exponent,  conformal exponent, scale-invariant damping, global existence,

\qquad \quad weighted Strichartz estimate,  Fourier integral operator
\vskip 0.1 true cm

\noindent
\textbf{2010 Mathematical Subject Classification} 35L70, 35L65, 35L67

\vskip 0.2 true cm

\addtocontents{toc}{\protect\thispagestyle{empty}}
\tableofcontents

\section{Introduction}

\subsection{Open questions and main results}

Consider the following semilinear wave equation with scale-invariant damping
\begin{equation}\label{00-1}
\left\{ \enspace
\begin{aligned}
&\partial_t^2 u-\Delta u +\f{\mu}{t}\,\p_tu=|u|^p, \\
&u(1,x)=u_0(x), \quad \partial_{t} u(1,x)=u_1(x),
\end{aligned}
\right.
\end{equation}
where $t\ge 1$, $x=(x_1,...,x_n)\in\Bbb R^n$ ($n\ge 1$),
$\mu>0$,  $p>1$, $\p_i=\p_{x_i}$ ($1\le i\le n$), $\Delta=\p_1^2+\cdot\cdot\cdot+\p_n^2$,
$(u_0(x), u_1(x))\in C_0^{\infty}(\Bbb R^n)$, and $\text{supp $u_0$, supp $u_1$}\subset B(0, 1)$.
Denote by $p_{crit}(n,\mu)=\max\{p_s(n+\mu), p_f(n)\}$,
where the Strauss index $p_s(z)=\f{z+1+\sqrt{z^2+10z-7}}{2(z-1)}$ ($z>1$)
is a positive root of the quadratic algebraic equation
$(z-1)p^2-(z+1)p-2=0$ (see \cite{Strauss} for the Strauss index $p_s(n)$ of the semilinear wave equation
$\partial_t^2 u-\Delta u=|u|^p$), and $p_f(n)=1+\f{2}{n}$ is the Fujita index
(see \cite{Fuj} for the Fujita index $p_f(n)$ of the semilinear parabolic equation
$\partial_t u-\Delta u=|u|^p$).
In terms of \cite{Rei1} and \cite{Imai}, there is an interesting open question as follows:

{\bf Open question (A).} {\it For $\mu>0$, when $p>p_{crit}(n,\mu)$, the small data weak solution $u$
of \eqref{00-1} exists globally; otherwise, the solution $u$ may blow up in finite time when $1<p\le p_{crit}(n,\mu)$.
}

For $n=1$, Open question (A) has been solved in \cite{DA-0}.
In addition, it follows from direct computation that $p_{crit}(n,\mu)=p_s(n+\mu)$
for $0<\mu\le\bar \mu(n)$ and
$p_{crit}(n,\mu)=p_f(n)$ for $\mu\ge\bar\mu(n)$ with $\bar\mu(n)=\frac{n^2+n+2}{n+2}$.
Note that the blowup result in Open question (A)  has been shown for \(1<p\le p_{crit}(n,\mu)\)
(see \cite{IS}, \cite{LTW}, \cite{PR}, \cite{TL1}, \cite{TL2} and \cite{W1}).
However, there are few results on the global existence part in Open question (A) for $n\ge 2$.
For examples, for \(n=2\) when \(\mu\geq3\) or \(n\geq3\) when \(\mu\geq n+2\) and $p<\f{n}{n-2}$, the global
small solution $u$ of \eqref{00-1} with $p>p_f(n)$ is established in \cite{DA};
for $n=3$ when $\mu=2$, it is proved in \cite{Rei1} that \eqref{00-1} with $p>p_s(n+\mu)$ has a global small symmetric
solution; for odd $n\ge 5$ when $p>p_{s}(n+\mu)$ and $\mu=2$, the authors in \cite{Rei2}
obtain the global small solution with radial symmetry;
for \(n=3\) and the radial symmetrical case,
the small solution $u$ exists globally for \(\mu\in[\frac{3}{2},2)\) and \(p_s(3+\mu)<p\leq2\)
(see \cite{LZ}).
Therefore, in terms of Open question (A), so far one has such an explicit open question:

{\bf Open question (B).} {\it For $n\ge 2$, \eqref{00-1} has a global small data weak solution $u$
when $p>p_s(n+\mu)$ and  $0<\mu<\bar\mu(n)$ or when $p>p_f(n)$ and  $\mu\ge\bar\mu(n)$.}

Especially, for $n=2$ and by $\bar\mu(2)=2$, Open question (B) can be stated as

{\bf Open question (C).} {\it For $0<\mu<2$ and $p>p_s(2+\mu)$ or $\mu\ge 2$ and $p>2$, there exists a
global small data weak solution $u$ to the 2-D problem
\begin{equation}\label{00-1-2}
\left\{ \enspace
\begin{aligned}
&\partial_t^2 u-\Delta u +\f{\mu}{t}\,\p_tu=|u|^p, \qquad\qquad \qquad  t\ge 1, \\
&u(1,x)=u_0(x), \quad \partial_{t} u(1,x)=u_1(x),\quad x\in\Bbb R^2.
\end{aligned}
\right.
\end{equation}
}

We start to solve Open problem (C) systematically.
In the present paper, the global small solution $u$ will be established for $p_{s}(2+\mu)<p<p_{conf}(2,\mu)=\f{\mu+5}{\mu+1}$
and $\mu\in (0, 1)\cup(1, 2)$. In forthcoming papers, we will show the global existence of small solution $u$ for
$p\ge p_{conf}(2,\mu)$ and $\mu\in (0, 1)\cup(1, 2)$ in  \cite{HLWY}, for $p\ge 2$ and $\mu\ge 2$ in \cite{LY},
and for $p>p_{s}(3)=1+\sqrt 2$ and $\mu=1$ in  \cite{HLWY-1}. Our main results are
\begin{theorem}\label{TH-1}
Let $\mu\in(0,1)\cup(1,2)$ in \eqref{00-1-2}. When either one of the following conditions is satisfied\\
(i) $p_{s}(2+\mu)<p<p_{conf}(2,\mu)$ for $\mu\in (0,1)\cup(1, 2\sqrt{2}-1)$;\\
(ii) $p_{s}(2+\mu)<p<\f{2}{\mu-1}$ for $\mu\in [2\sqrt{2}-1, 2)$,\\
there exists a small  constant $\varepsilon_0>0$ such that as long as $\|u_0\|_{W^{\f{2+|\mu-1|}{2}+\delta,1}(\mathbb{R}^2)}+ \|u_1\|_{W^{\f{2-|\mu-1|}{2}+\delta,1}(\mathbb{R}^2)}$ $\le\ve_0$,  problem \eqref{00-1-2} admits a global weak solution $u$ with
\begin{equation}\label{con2}
t^{\theta(\mu, p)}\left( 1+\left|{\psi_{\mu}^2(t)}-|x|^2 \right|\right)^{\gamma}u\in L^{p+1}([1,+\infty) \times \mathbb{R}^2),
\end{equation}
where $\psi_{\mu}(t)=|\mu-1|t^{\f{1}{|\mu-1|}}$,
$\theta(\mu, p)=\f{2\mu}{(1-\mu)(p+1)}$ for $0<\mu<1$ and $\theta(\mu, p)=\f{3-\mu-p(\mu-1)}{(\mu-1)(p+1)}$ for $1<\mu<2$,
$0<\delta<\f{2-|\mu-1|}{2}-\f{1}{p+1}-\gamma$, and
the constant $\gamma$ fulfills
\begin{equation}\label{con1}
\f{1}{p(p+1)}<\gamma<\f{(\mu+1)p-(\mu+3)}{2(p+1)}.
\end{equation}
\end{theorem}

\begin{remark}\label{W-1}
{\it It follows from $p>p_{s}(2+\mu)$ and direct computation that
$\f{1}{p(p+1)}<\f{(\mu+1)p-(\mu+3)}{2(p+1)}$ holds.
This implies that the condition \eqref{con1} makes sense.
In addition, by conditions (i)-(ii) in Theorem \ref{TH-1} and \eqref{con1}, one can verify  $\f{2-|\mu-1|}{2}-\f{1}{p+1}-\gamma>0$
and then the choice of $\delta>0$ is reasonable.}
\end{remark}

\begin{theorem}\label{TH-2}
For $\mu\in [2\sqrt{2}-1, 2)$ and $\f{2}{\mu-1}\leq p<p_{conf}(2,\mu)$, when the smooth initial data $(u_0, u_1)$ satisfies
\begin{equation}\label{q-1}
  \sum_{|\beta| \leq 1}(\|Z^{\beta} u_0\|_{\dot{H}^{s}\left(\mathbb{R}^{2}\right)}+\|Z^{\beta} u_1\|_{\dot{H}^{s-(\mu-1)}\left(\mathbb{R}^{2}\right)})\leq\varepsilon_0,
\end{equation}
where $\varepsilon_0>0$ is a small constant, $s=1-\frac{\mu+1-p(\mu-1)}{p-1}$ and $Z=\left\{\partial_{1}, \partial_{2}, x_{1} \partial_{2}-x_{2} \partial_{1}\right\}$, problem \eqref{00-1-2} admits a global solution $u$ such that for $|\beta| \leq 1$,
$$
t^{\f{3-\mu-p(\mu-1)}{(\mu-1)(p+1)}}Z^{\beta} u \in L_{t}^{q} L_{r}^{\nu} L_{\theta}^{2}\left([1,+\infty)\times [0,\infty)\times [0,2\pi]\right) \cap L_{t}^{\infty} \dot{H}^{s}([1,+\infty)\times\mathbb{R}^2),
$$
where $(x_1,x_2)=(r cos\theta, r sin\theta)$ with $r=\sqrt{x_1^2+x_2^2}$ and $\th\in [0, 2\pi]$,  $\frac{1}{q}+\frac{2}{\nu(\mu-1)}=\frac{\mu+1-p(\mu-1)}{(p-1)(\mu-1)}-\frac{3-\mu-p(\mu-1)}{(p+1)(\mu-1)}$,
and the constant $\nu$ is chosen as: if $2\sqrt{2}-1\leq\mu<\mu_1=1.977$, then $\nu=p+\f{1}{3}$;
if $\mu_1\leq\mu<2$, then $\nu=p$
for $\f{2}{\mu-1}\leq p<p^*(\mu)$ and $\nu=p+\f{1}{3}$ for $p^*(\mu)\leq p<p_{conf}(2,\mu)$ with
$p^*(\mu)$ being the positive root of the quadratic algebraic equation
$3(\mu+1)p^2-2(\mu+4)p-(\mu+11)=0.$
\end{theorem}

\begin{remark}\label{JY-4-1}
{\it Collecting the results in Theorems \ref{TH-1}-\ref{TH-2},
one knows that the Open problem (C) holds
for $p_s(2+\mu)<p<p_{conf}(2,\mu)$ when $\mu\in (0,1)\cup(1,2)$. }
\end{remark}

\begin{remark}\label{JY-3}
{\it When $\mu=1$, by the Liouville transformation $v=t^{\f12}u$, the equation in \eqref{00-1-2}
becomes $\p_t^2v-\Delta v+\f{1}{4t^2}v=t^{\f{1-p}{2}}|v|^p$, which is a semilinear Klein-Gordon equation
with time-dependent coefficient. So far there are no results on the global existence of $v$ and further on
the solution $u$ in \eqref{00-1-2} with $\mu=1$. In \cite{HLWY-1}, we will solve this problem for $\mu=1$.}
\end{remark}

\begin{remark}\label{JY-8-1}
{\it For $n\ge 3$, $\bar\mu(n)>2$ holds.
By the method in the paper, we can prove Open question (B) for $n\ge 3$
and $\mu\in(0,1)\cup(1,2)$.}
\end{remark}

\begin{remark}\label{JY-8-2}
{\it In the case of $1<\mu<2$, the resulting semilinear generalized Tricomi equation from \eqref{00-1-2}
will have different critical power properties for \(p_s(2+\mu)<p<\f{2}{\mu-1}\)
and \(\frac{2}{\mu-1}\leq p<p_{\conf}(2,\mu)\) when \(2\sqrt{2}-1\leq\mu<2\)
(see the explanations in Remark \ref{JY-8-3} below). This leads to that we require to
establish different kinds of Strichartz estimates for the corresponding linear Tricomi equations.
Therefore, the solution of Open question (C) for $p_{s}(2+\mu)<p<p_{conf}(2,\mu)$ and $\mu\in (0,1)\cup(1,2)$
is divided into Theorem \ref{TH-1} and Theorem \ref{TH-2}, respectively.}
\end{remark}

\subsection{Physical backgrounds on some wave operators with  scale-invariant  damping}

At first, we point out that the linear wave operator $\partial_t^2-\Delta  +\f{\mu}{t}\p_t$ in \eqref{00-1}
admits an interesting physical background
which comes from the compressible Euler equations with time-dependent damping. In fact,
consider the 3-D compressible isentropic Euler equations
\begin{equation}\label{1.1-L}
\left\{ \enspace
\begin{aligned}
&\p_t\rho+\opdiv(\rho u)=0,\\
&\p_t(\rho u)+\opdiv(\rho u\otimes
  u+p\,\RN{1}_{3})=-\f{\mu}{(1+t)^{\la}}\rho u,\\
&\rho(0,x)=\bar \rho+\ve\rho_0(x),\quad u(0,x)=\ve u_0(x),
\end{aligned}
\right.
\end{equation}
where $x=(x_1, x_2, x_3)$, $\rho$, $u = (u_1, u_2, u_3)$, and
$p$ stand for the density, velocity, and pressure, respectively,
$\RN{1}_{3}$ is the $3\times 3$ identity matrix, $\mu>0$, $\la\ge 0$ and
$\bar\rho>0$ are constants, $u_0=(u_{1,0},u_{2,0}, u_{3,0})$, $(\rho_0, u_0)\in C_0^{\infty}(\R^3)$, $(\rho_0,
u_0)\not\equiv 0$, $\rho(0, x)>0$, and $\ve>0$ is sufficiently
small. In addition, the state equation of gases is
$p(\rho)=A\rho^{\g}$ with $A>0$ and $\g>1$ being constants,
and the gases are assumed to be irrotational, namely, $\operatorname{curl} u_0=
\left(\p_2u_{3,0}-\p_3u_{2,0}, \p_3u_{1,0}-\p_1u_{3,0},
\p_1u_{2,0}-\p_2u_{1,0}\right)\equiv 0$. In this case,
one can know that $\operatorname{curl} u(t,x)\equiv 0$ holds for any $t\ge 0$
as long as the smooth solution $(\rho, u)(t,x)$ exists.
By introducing a scalar potential function $\chi=\chi(t,x)$ with
$u=\na_x\chi$, we substitute $u=\na_x\chi$
into the second equation of \eqref{1.1-L} to obtain
\begin{equation}\label{1.2-L}
  \p_t\chi+\f12|\na_x\chi|^2+h(\rho)+\f{\mu}{(1+t)^{\la}}\,\chi=0,
\end{equation}
where $\ds h'(\rho)=c^2(\rho)/\rho$ with $c(\rho)=\sqrt{p'(\rho)}$ and
$h(\bar\rho)=0$. Without loss of generality, $c(\bar\rho)=1$ is supposed.
By $h'(\rho)>0$ for $\rho>0$ and the implicit function theorem, one has from \eqref{1.2-L} that
\begin{equation}\label{1.3-L}
  \rho=h^{-1}\big(-\big(\p_t\chi+\f12\,|\na_x\chi|^2
  +\f{\mu}{(1+t)^{\la}}\,\chi\big)\big),
\end{equation}
where $\bar\rho=h^{-1}(0)$ and $h^{-1}$ is the inverse function of
$h=h(\rho)$.
Inserting \eqref{1.3-L} into the first equation of
\eqref{1.1-L} derives
\begin{equation}\label{1.4-L}
\begin{split}
&\p_t^2\chi-c^2(\rho)\Delta\chi+2\sum_{k=1}^3\p_k\chi
  \p_{tk}^2\chi
  +\sum_{i,k=1}^3\p_i\chi\p_k\chi\p_{ik}^2\chi
  +\f{\mu}{(1+t)^{\la}}\,|\na_x\chi|^2+\p_t\big(\f{\mu}{(1+t)^{\la}}\,\chi\big)=0.
\end{split}
\end{equation}
Let $\ds \psi=
\f{\chi}{(1+t)^\la}$. It follows from \eqref{1.4-L} and $c(\bar\rho)=1$ that
\begin{equation}\label{2.6-L}
  \square\psi+\f{\mu}{(1+t)^\la}\p_t\psi+\f{2\la}{1+t}\p_t\psi
  -\f{\la(1-\la)}{(1+t)^2}\psi=Q(\psi,\p\psi, \p^2\psi),
\end{equation}
where
\begin{align*}\label{2.7-L}
Q(\psi,\p\psi, \p^2\psi)=&
(c^2(\rho)-1)\Delta\psi-2(1+t)^\la\p_t\na_x\psi\cdot\na_x\psi-2\la(1+t)^{\la-1}|\na_x\psi|^2
\\ \no
&-\mu|\na_x\psi|^2-(1+t)^{2\la}\ds\sum_{1\le i,j\le
  3}\p_i\psi\p_j\psi\p_{ij}^2\psi.
\end{align*}
Especially, for $\lambda=1$, \eqref{2.6-L} becomes
\begin{equation}\label{2.8-L}
  \square\psi+\f{\mu+2}{1+t}\p_t\psi=Q(\psi,\p\psi, \p^2\psi).
\end{equation}
Thus, the linear wave operator $\square+\f{\mu+2}{1+t}\p_t$ with  scale-invariant  damping
naturally  appears in \eqref{2.8-L}.
Moreover, the detailed analyses on $\square+\f{\mu+2}{1+t}\p_t$ and  the resulting time-weighted energy estimates
play key roles in proving the global existence of the smooth solution $(\rho, u)$ in \eqref{1.1-L} (see \cite{Hou-1}-\cite{Hou-2}
and \cite{Chen-Mei}). On the other hand, analogous scale-invariant  damping
can appear in the irrotational bipolar Euler-Poisson system
with time-dependent damping (see \cite{Mei-2} and so on).

Secondly, we discuss the global existence and stability of the smooth supersonic polytropic gases
in a 3-D infinitely long  divergent nozzle (see Figure 1 and Figure 2 below). Assume that the 3-D infinitely long
divergent straight nozzle is given by  $\O_0=\{x\in \mathbb{R}^3: x_1^2+x_2^2\le \tan^2\vp_0 x_3^2,
x_1^2+x_2^2+x_3^2\ge1, x_3\ge cos\vp_0\}$,
where $\vp_0\in (0, \f{\pi}{2})$ is a fixed constant (see Figure 2). The flow in $\Omega_0$
is described by the following 3-D steady  compressible isentropic Euler equations
\beq
\bec
\ss_{j=1}^3\p_j(\rho v_j)=0,\\
\ss_{j=1}^3\p_j(\rho v_iv_j)+\p_ip=0,~~i=1,2,3,
\eec\label{1-1-L}\eeq
where $\rho, v=(v_1, v_2, v_3)$ and $p$ represent the density, velocity and pressure
of polytropic gases, respectively. In addition,  the state equation is $p=\rho^{\g}$ for $\g>1$,
and the local sound speed is $c(\rho)=\sqrt{p'(\rho)}=\sqrt{\g\rho^{\g-1}}$.
Moreover, for simplicity, one can assume that the following
Bernoulli's law holds
\begin{align}\label{1-2-L}
\f12|v|^2+\f{c^2(\rho)}{\g-1}=B,
\end{align}
where $|v|^2=v_1^2+v_2^2+v_3^2>c^2(\rho)$, and $B>0$ is the Bernoulli's  constant.
Note that for $\rho(x)=\hat\rho(r)$ and $v(x)=\ds\f{x}{r}\hat V(r)$ with $r=\sqrt{x_1^2+x_2^2+x_3^2}\ge 1$, \eqref{1-1-L} with \eqref{1-2-L}
can have such a supersonic symmetric solution $(\hat \rho(r), \hat V(r))$ (see \cite{Xu-Yin-1} and \cite{Xu-Yin-2})
\begin{equation*}\label{E-L}
\left\{
\begin{aligned}
&(r^2\hat \rho\hat  V)'(r)=0,\\
&\ds\f{1}{2}{\hat V}^2(r)+\f{\gamma}{\gamma-1}{\hat\rho}^{\gamma-1}(r)=1,\\
&\hat \rho(1)=\rho_0,\quad \hat V(1)=q_0,\quad q_0>c(\rho_0).
\end{aligned}
\right.
\end{equation*}

\begin{center}
\includegraphics[width=18.5cm,height=4.5cm]{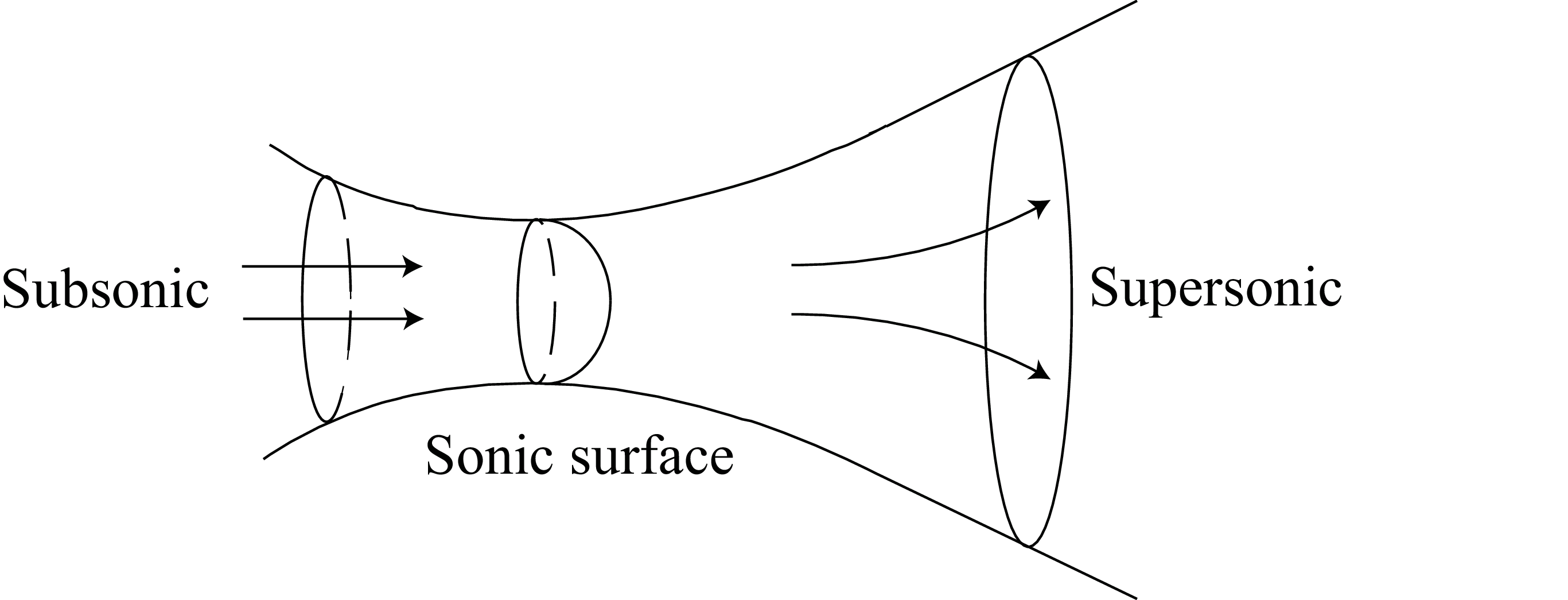}
\end{center}

\centerline{\bf Figure 1. Supersonic flow in a 3-D infinite long de Laval nozzle}

\vskip 0.4 true cm
\begin{center}
\includegraphics[width=11cm,height=4cm]{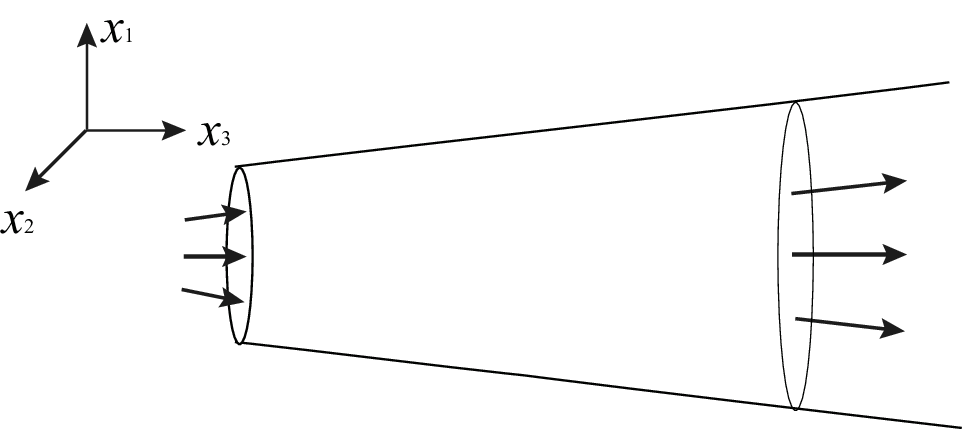}
\end{center}

\centerline{\bf Figure 2. Supersonic flow in a 3-D  infinitely long  divergent straight nozzle}
\vskip 0.4 true cm

If the supersonic flow in $\Omega_0$ is irrotational, then one can introduce a potential function $\Phi(x)$
with $v=\nabla_x\Phi$ such that \eqref{1-1-L} becomes the following second order quasilinear
wave equation
\begin{align}\label{1.12-L}
\ds\sum_{i=1}^3((\p_i\Phi)^2-c^2(\rho))\p_i^2\Phi+2 \ds\sum_{1\le i<j\le
3}\p_i\Phi\p_j\Phi\p_{ij}^2\Phi=0, \qquad x_3\ge 0.
\end{align}
By the linearization of \eqref{1.12-L} around the background solution $(\hat \rho(r), \hat V(r))$,
it follows from (3.1) of \cite{Xu-Yin-1} that
\begin{align}\label{3.1-L}
\mathcal {L}\dot{\Phi}=\dot
f\quad\text{in $\O_0$},
\end{align}
where $\dot{\Phi}=\Phi-\hat{\Phi}$ with $\p_r\hat{\Phi}=\hat V(r)$,  the linear wave operator $\mathcal {L}= \p_r^2-\f{1}{(1+r)^{2(\g-1)}}(\p_1^2+\p_2^2)+\f{2(\g-1)}{1+r}\p_r$, and $\dot f$ is of the second order error term.
Obviously, $\mathcal {L}\dot{\Phi}$ has the scale-invariant damping $\f{2(\g-1)}{1+r}\p_r\dot{\Phi}$.
Based on the delicate and involved weighted energy estimates on $\mathcal {L}\dot{\Phi}$,
when $1<\gamma<\f{11}{7}$, the authors in \cite{Xu-Yin-1}-\cite{Xu-Yin-2}
establish the global existence and stability of the smooth supersonic solution $(\rho, v)$ of \eqref{1-1-L} in $\Omega_0$.
On the other hand, the long time behavior
of solutions to the $n-$D ($n\ge 2$) linear equation $\partial_t^2 v-\Delta v+\f{\mu}{t^{\beta}}\p_tv=0$ for $\beta>1$
and $0<\beta<1$ are different, which corresponds to the time-decay rate $t^{-\f{n-1}{2}}$
of linear free wave equations and the time-decay rate $t^{-\f{n}{2}}$ of linear free parabolic equations, respectively
(see \cite{Wirth-1}-\cite{Wirth-2}). This means that the wave operator $\partial_t^2-\Delta+\f{\mu}{t}\p_t$
is the critical case for the time-dependent damping.

\subsection{Reformulation and more general global existence results}\label{E-3}

For $0<\mu<1$ in \eqref{00-1-2}, by $\mu=\f{k}{k+1}$ with $k\in (0,\infty)$
and $T=t^{k+1}/(k+1)$ as in \cite{Rei1}, the equation in \eqref{00-1-2} is essentially equivalent to
$\partial_T^2 u-T^{2k}\Delta u=T^{2k} |u|^p$. For $1<\mu<2$, by $v=t^{\mu-1}u$ and $2-\mu=\f{\t k}{\t k+1}$ with $\t k\in (0,\infty)$
and $T=t^{\t k+1}/(\t k+1)$, the equation in \eqref{00-1-2} can become
$\partial_T^2 u-T^{2\t k}\Delta u=T^{\al_0} |u|^p$
with $\al_0=2\t k+1-p$. Therefore, in order to solve Open question (C) with $\mu\in(0,1)\cup(1,2)$,
it is required to study the global solution of the following 2-D problem
\begin{equation}\label{YH-4}
\begin{cases}\partial_t^2 u-t^{m} \Delta u=t^\alpha|u|^p, \\
u(1, x)=u_0(x), \p_t u(1, x)=u_1(x),
\end{cases}
\end{equation}
where $m>0$, $\alpha\in\Bbb R$, $p>1$, $u_0(x), u_1(x)\in C_0^{\infty}(\Bbb R^2)$ and
supp $u_0$, supp $u_1\in B(0, 1)$.
When $\al=0$, it has been shown that there exists a critical power \(p_{\crit}(2,m)>1\) such that
when $p>p_{\crit}(2,m)$, the small data solution $u$ of \eqref{YH-4} exists globally;
when $1<p\le p_{\crit}(2,m)$, the solution $u$ may blow up in finite time,
where \(p_{\crit}(2,m)\) is the positive root of
\begin{equation*}\label{equ:p2}
(m+1)p^2-3p-(m+2)=0,
\end{equation*}
the reader may see \cite{HWY1,HWY2,HWY3,HWY4}.
In addition, it has been shown in Theorem 1.1 of \cite{PR} that there exists a critical exponent $p_{crit}(2,m,\alpha)>1$ for $\al>-2$
such that when $1<p\le p_{crit}(2,m,\alpha)$, the solution of \eqref{YH-4} can blow up in finite time with some suitable
choices of $(u_0,u_1)$, where
$p_{crit}(2,m,\alpha)=\max \left\{p_1(2,m,\alpha), p_2(2,m,\alpha)\right\}$
with $p_1(2,m,\alpha)=1+\frac{2+\alpha}{m+1}$ and $p_2(2,m,\alpha)$ being the positive root of the quadratic equation
$(m+1)p^2-(3+2\alpha)p-(m+2)=0$. Obviously, $p_{crit}(2,m,0)=p_{\crit}(2,m)$ holds. In addition,
it is not difficult to verify that when $\al>-1$, $p_{crit}(2,m,\alpha)=p_2(2,m,\alpha)$
holds; when $-2<\al\le -1$, $p_{crit}(2,m,\alpha)=p_1(2,m,\alpha)$ holds. On the other hand,
the conformal exponent $p_{conf}(2,m,\alpha)$ for \eqref{YH-4} can be determined as
(see Subsection \ref{E-4} below)
\begin{equation}\label{equ:conf}
p_{\conf}(2,m,\alpha)=\frac{m+2\alpha+5}{m+1}.
\end{equation}
For problem \eqref{YH-4} with $p_{crit}(2,m,\alpha)<p<p_{\conf}(2,m,\alpha)$ and suitable scopes
of $m$ and $\alpha$, we now establish some global existence results of small data solution $u$,
which will lead to the conclusions in Theorems \ref{TH-1}-\ref{TH-2}.

\begin{theorem}\label{thm:global existence-L}
For \(m\in(0,\infty)\), \(-1<\alpha\leq m\) and $p_{crit}(2,m,\alpha)<p<p_{conf}(2,m,\alpha)$,
there exists a constant \(\varepsilon_0>0\) such that when $\parallel u_0\parallel_{W^{\f{m+3}{m+2}+\delta,1}(\mathbb{R}^2)}+\parallel u_1\parallel_{W^{\f{m+1}{m+2}+\delta,1}(\mathbb{R}^2)}\leq\varepsilon_0$,
\eqref{YH-4} has a  global weak solution $u$ which satisfies
\begin{equation}\label{equ:1.3}
\left(1+\big|\phi_m^2(t)-|x|^2\big|\right)^{\gamma}t^{\frac{\alpha}{p+1}}u\in L^{p+1}([1,\infty)\times \mathbb{R}^2),
\end{equation}
where and below $\phi_m(t)=\f{2}{m+2}t^{\f{m+2}{2}}$, \(0<\delta<\frac{m+1}{m+2}-\gamma-\frac{1}{p+1}\),
and the positive constant $\gamma$ fulfills
\begin{equation}\label{equ:1.4}
\f{1}{p(p+1)}<\gamma<\f{(m+1)p-(2\alpha+3)}{(m+2)(p+1)}.
\end{equation}
\end{theorem}

\begin{remark}\label{W-Y}
{\it By $p_{crit}(2,m,\alpha)<p<p_{conf}(2,m,\alpha)$ with \(m\in(0,\infty)\) and \(-1<\alpha\leq m\),
we can conclude $p_{crit}(2,m,\alpha)=p_2(n,m,\alpha)$, $\f{1}{p(p+1)}<\f{(m+1)p-(2\alpha+3)}{(m+2)(p+1)}$ and \(\frac{m+1}{m+2}-\gamma-\frac{1}{p+1}>0\).
Then \eqref{equ:1.4} and the existence of $\delta>0$ hold true in Theorem \ref{thm:global existence-L}.}
\end{remark}

\begin{theorem}\label{thm:global existence-LL}
For $0<m\leq \sqrt{2}-1$, $m+2\leq p<\f{3m+7}{m+3}$, $\alpha=m-p+1\le -1$,
there exists  a small constant $\varepsilon_0>0$
such that when the initial data $(u_0, u_1)$ satisfies
\begin{equation}\label{prop:1}
  \sum_{|\beta| \leq 1}(\|Z^{\beta} u_0\|_{\dot{H}^{s}\left(\mathbb{R}^{2}\right)}+\|Z^{\beta} u_1\|_{\dot{H}^{s-\f{2}{m+2}}\left(\mathbb{R}^{2}\right)})\leq\varepsilon_0,
\end{equation}
where $s=1-\f{2(m-p+1)+4}{(m+2)(p-1)}$ and $Z=\left\{\partial_{1}, \partial_{2}, x_{1} \partial_{2}-x_{2} \partial_{1}\right\}$,
\eqref{YH-4} admits a global solution $u$ which satisfies that for $|\beta| \leq 1$,
$$
t^{\f{m-p+1}{p+1}}Z^{\beta} u \in L_{t}^{q} L_{r}^{\nu} L_{\theta}^{2}\left([1,+\infty)\times [0, +\infty)\times [0, 2\pi]\right)
\cap L_{t}^{\infty} \dot{H}^{s}\left([1,+\infty)\times\mathbb{R}^2\right),
$$
where $(r,\theta)\in [0, +\infty)\times [0, 2\pi]$ stands for the polar coordinate of $x\in\Bbb R^2$,
the positive constants $q$ and $\nu$ satisfy
\begin{equation}\label{equ:ugly}
\left\{ \enspace
\begin{aligned}
&2\leq q \leq 2p,\quad 2\leq \nu\leq 2p,\\
&\f{1}{\nu}+\f{m-p+1}{(m+2)(p+1)}>0,\\
&\f{1}{q}+\f{m+2}{\nu}=\f{m-p+3}{p-1}-\f{m-p+1}{p+1},\\
&\f{1}{q}+\f{m+2}{2\nu}\leq\f{m+1}{2}-\f{m-p+1}{p+1},\\
&\frac{1}{q}+\frac{m+2}{2 \nu} \geq \frac{3}{2 p}+\f{m-p+1}{(p+1)p}.
\end{aligned}
\right.
\end{equation}
\end{theorem}

\begin{remark}\label{JY-8-3}
{\it When \(2\sqrt{2}-1\leq\mu<2\) and \(\frac{2}{\mu-1}\leq p<p_{\conf}(2,\mu)\),
for the resulting equation of \eqref{YH-4} from \eqref{00-1-2}, the scope of $\alpha$
is \(-\frac{4}{3}<\alpha\leq-1\). If follows from direct computation that
\(p_{\crit}(2,m,\alpha)=p_1(2,m,\alpha)=1+\frac{2+\alpha}{m+1}\) holds.
This means that the crucial Strichartz estimate in Lemma \ref{th2-1} with $\alpha>-1$ for the homogeneous Tricomi equation
can not been utilized. To overcome this difficulty, we will establish another kind of weighted mixed-norm Strichartz estimate
in Lemma \ref{lem1} of Section \ref{E-8} to prove Theorem \ref{thm:global existence-LL}.}
\end{remark}

\begin{remark}\label{W-Z}
{\it Although the positive constants $q$ and $\nu$ fulfill several restriction conditions in  \eqref{equ:ugly},
it is no problem for the existence of $q$ and $\nu$ when $q$ and $\nu$ have different relations.
One can see the details in Section \ref{E-5}.}
\end{remark}

\subsection{Derivation of conformal powers $p_{\conf}(2,m,\alpha)$ and $p_{\conf}(2,\mu)$}\label{E-4}

We now derive the conformal invariant exponents $p_{\conf}(2,m,\alpha)$
and $p_{\conf}(2,\mu)$ in Theorem \ref{thm:global existence-L} and  Theorems \ref{TH-1}-\ref{TH-2}, respectively.
Note that the Lagrangian functional for \eqref{YH-4} is
$L(u, u')=\frac{1}{2}|\partial_tu|^2-\frac{1}{2}t^m|\nabla_x u|^2-\frac{t^\alpha}{p+1}|u|^{p+1}$ with $u'=
(\p_0u, \p_1u, \p_2u)$.
Then one has that
$\frac{d}{d\varepsilon}\int_{\Bbb R^3} L(u+\varepsilon \psi, (u+\varepsilon \psi)')dtdx\big|_{\varepsilon=0}
=0$  for any \(\psi\in C_0^\infty(\R^{3})\).
This means that \(u\) should satisfy the corresponding Euler-Lagrangian equation
\begin{equation}\label{equ:EL}
	\partial_{u} L(u, u')-\sum_{j=0}^2\partial_j\left(\partial_{u_j} L(u, u')\right)=0,
\end{equation}
where $u_j=\p_ju$ for $j=0,1,2$.
By direct computation, it is easy to know that the solution of \eqref{YH-4} is invariant under the scaling transformation
$u(t,x)\longrightarrow \lambda^{\frac{\alpha+2}{p-1}}u(\lambda t, \lambda^{\frac{m+2}{2}}x)$.
Set \(u_\lambda=\lambda^{\frac{\alpha+2}{p-1}}u(\lambda t, \lambda^{\frac{m+2}{2}}x)\),
then one has
\begin{equation}\label{equ:pL}
	\partial_\lambda L(u_\lambda, u_\lambda')=\p_{u}L(u_\lambda, u'_\lambda)
	\partial_\lambda u_\lambda +\sum_{j=0}^2\p_{u_j}L(u_\lambda, u'_\lambda)
    \partial_\lambda\partial_j u_\lambda.
\end{equation}
Substituting \eqref{equ:EL} into \eqref{equ:pL} derives
$\partial_\lambda L(u_\lambda, u_\lambda')=\sum_{j=0}^2\partial_j\left(\partial_{u_j}L(u_\lambda, u_\lambda')\partial_\lambda u_\lambda\right)$.
Due to
\begin{equation*}
\partial_\lambda u_\lambda|_{\lambda=1}=\frac{\alpha+2}{p-1}u+t\partial_tu+\frac{m+2}{2} \sum_{j=1}^2x_j\partial_ju,
\end{equation*}
then
\begin{equation}\label{equ:pL1}
\partial_\lambda L(u_\lambda, u_\lambda')|_{\lambda=1}=\sum_{j=0}^2\partial_j\big(\partial_{u_j} L(u, u')\big(\frac{\alpha+2}{p-1}u+t\partial_tu+\frac{m+2}{2} \sum_{j=1}^2x_j\partial_ju\big)\big).
\end{equation}
On the other hand, it follows from direct computation that
\begin{equation*}
L(u_\lambda, u_\lambda')=\lambda^{\frac{2(\alpha+2)}{p-1}+2}(L(u,u'))(\lambda t,\lambda^{\frac{m+2}{2}}x)
\end{equation*}
and
\begin{equation}\label{equ:pL2}
\partial_\lambda L(u_\lambda, u_\lambda')|_{\lambda=1}=
\big(\frac{2(\alpha+2)}{p-1}+2\big)L(u,u')+t\partial_tL(u,u')+\frac{m+2}{2}\sum_{j=1}^2x_j\partial_jL(u,u').
\end{equation}
Combining  \eqref{equ:pL1} with \eqref{equ:pL2} yields
\begin{equation}\label{equ:div}
\begin{split}
&\partial_t Q_0
+\sum_{j=1}^2\p_j Q_j=\big(\frac{2(\alpha+2)}{p-1}-m-1\big)L(u,u'),
\end{split}
\end{equation}
where
\begin{equation*}\label{equ:div}
\begin{split}
&Q_0=\partial_{u_0} L(u, u')\big(\frac{\alpha+2}{p-1}u+t\partial_tu
+\frac{m+2}{2} \sum_{j=1}^2x_j\partial_ju\big) -tL(u,u'),\\
&Q_j=\partial_{u_j} L(u, u')\big(\frac{\alpha+2}{p-1}u
+t\partial_tu+\frac{m+2}{2} \sum_{j=1}^2x_j\partial_ju\big) -\frac{m+2}{2}x_jL(u,u'),\quad j=1,2.
\end{split}
\end{equation*}
Since \(\text{supp}u(t,x)\) is compact with respect to the variable $x$, then integrating \eqref{equ:div} over \(\R^2\) yields
\begin{equation}\label{CC-2}
\frac{d}{dt} \int_{\Bbb R^2} Q_0dx=\big(\frac{2(\alpha+2)}{p-1}-m-1\big)\int_{\Bbb R^2} L(u,u')dx.
\end{equation}
When $\frac{2(\alpha+2)}{p-1}-m-1=0$,  $Q_0$ is conserved by \eqref{CC-2}.
From this, the conformal invariant power for problem \eqref{YH-4} can be determined as
\begin{equation}\label{equ:conf-1}
p_{\conf}(2,m,\alpha)=\frac{m+2\alpha+5}{m+1}.
\end{equation}
It should be pointed out that $p_{\conf}(2,0,0)=5$
and $p_{\conf}(2,m,0)=\frac{m+5}{m+1}$ in \eqref{equ:conf-1} just correspond to
the conformal powers of 2-D semilinear equations $\square u=|u|^p$ and $\p_t^2u-t^m\Delta u=|u|^p$
respectively (see \cite{Gls1}, \cite{HWY1}, \cite{HWY4} and \cite{Rua4}).

For $0<\mu<1$ in \eqref{00-1-2}, corresponding to the equation \eqref{YH-4}, we have
$m=\f{2\mu}{1-\mu}$ and $\alpha=\f{2\mu}{1-\mu}$. Then it is derived from \eqref{equ:conf-1} that
\begin{equation}\label{equ:conf-2}
p_{conf}(2,\mu)=\f{\mu+5}{\mu+1}.
\end{equation}
For $1<\mu<2$ in \eqref{00-1-2}, compared with the equation \eqref{YH-4}, there hold $m=\f{2(2-\mu)}{\mu-1}$ and $\alpha=\f{2(2-\mu)}{\mu-1}
+1-p$. In terms of \eqref{equ:conf-1}, one has
$$p_{conf}(2,\mu)=\f{6(2-\mu)+(7-2p_{conf}(2,\mu))(\mu-1)}{3-\mu},$$
which yields that for $1<\mu<2$,
\begin{equation}\label{equ:conf-3}
p_{conf}(2,\mu)=\f{\mu+5}{\mu+1}.
\end{equation}

\subsection{Sketch of proof on Theorem~\ref{thm:global existence-L}}

To prove Theorem~\ref{thm:global existence-L}, motivated by the weighted Strichartz estimates
derived in \cite{Gls1, LS} for the $n-$D ($n\ge 3$) linear wave equation $\square\tilde w_0=\tilde F_0$ with $(\t w_0(0,x), \p_t\t w_0(0,x))
=(u_0(x), u_1(x))$
and in \cite{HWY4} for the $n-$D ($n\ge 3$) linear generalized Tricomi equation $\p_t^2\t w_1-t^m\Delta \t w_1=\t F_1(t,x)$
with $(\t w_1(0,x), \p_t\t w_1(0,x))=(u_0(x), u_1(x))$ ($m\in\Bbb N$), where
the weights $1+|t^2-|x|^2|$  and $(2+\phi_m(t))^2-|x|^2$
are utilized, respectively, we will establish some new classes of
Strichartz estimates  for the following 2-D problem
\begin{equation}\label{a-1-M}
\begin{cases}\partial_t^2 \t w-t^{m} \Delta \t w=F(t,x),\\
\t w(1, x)=u_0(x),\quad \p_t \t w(1, x)=u_1(x),
\end{cases}
\end{equation}
where $m>0$, $u_0, u_1 \in C_{0}^{\infty}(\mathbb{R}^{2})$, $\operatorname{supp}(u_0, u_1) \subseteq
B(0, 1)$, and $F(t,x)\equiv0$ when $|x|>\phi_m(t)+1$.
It is emphasized that the property of $n\ge 3$ plays crucial roles in some places of  \cite{HWY4}
(see Line 8 of Page 54 and the argument of (4-42) in \cite{HWY4}). Hence,
we have to improve some techniques in \cite{HWY4} to study the 2-D problem \eqref{a-1-M}
(for examples, see the treatments in \eqref{a-72-1} and \eqref{a-73-1} of Appendix).

Let $\t w=v+w$, where $v$ solves the 2-D homogeneous problem
\begin{equation}\label{a-1}
\begin{cases}\partial_t^2 v-t^{m} \Delta v=0, \\
v(1, x)=u_0(x),\quad \p_t v(1, x)=u_1(x),
\end{cases}
\end{equation}
and $w$ is a solution of the following 2-D inhomogeneous problem
\begin{equation}\label{Y-3}
\begin{cases}
&\partial_t^2 w-t^m\triangle w=F(t,x), \\
&w(1,x)=0,\quad \partial_tw(1,x)=0.
\end{cases}
\end{equation}

For $v$, based on the pointwise estimate of $v$ in \cite[formula 2-20]{HWY4},
we can obtain such a new spacetime-weighted Strichartz estimate
\begin{equation}\label{a-2-M}
\begin{aligned}
&\|\left(\left(\phi_{m}(t)+2\right)^{2}-|x|^{2}\right)^{\gamma} t^{\frac{\alpha}{q}} v\|_{L^{q}\left(\left[1, +\infty\right) \times \mathbb{R}^{2}\right)} \\
& \leq C_{m,q,\gamma,\delta}\big(\|u_0\|_{W^{\frac{m+3}{m+2}+\delta, 1}\left(\mathbb{R}^{2}\right)}+\|u_1\|_{W^{\frac{m+1}{m+2}+\delta, 1}\left(\mathbb{R}^{2}\right)}\big),
\end{aligned}
\end{equation}
where $\alpha>-1$, $p_{\text {crit }}(2,m, \alpha)+1<q<p_{\text {conf}}(2,m,\alpha)+1$,
the positive constants $\gamma$ and $\delta$ fulfill
\begin{equation}\label{con3-M}
0<\gamma<\frac{m+1}{m+2}-\frac{m+4+2\al}{(m+2) q}, \quad 0<\delta<\frac{m+1}{m+2}-\gamma-\frac{1}{q}.
\end{equation}

For $w$, our main concern is to establish  a precise Strichartz inequality as follows
\begin{equation}
\begin{split}
\|(\big(\phi_m(t)+2\big)^2-|x|^2)^{\gamma_1}&t^{\frac{\alpha}{q}}w\|_{L^q([1, \infty)\times \mathbb{R}^{2})}
\leq C
\|(\big(\phi_m(t)+2\big)^2-|x|^2)^{\gamma_2}t^{-\frac{\alpha}{q}}F\|_{L^{\frac{q}{q-1}}([1, \infty)\times \mathbb{R}^{2})},
\end{split}
\label{equ:3.5-L}
\end{equation}
where $-1<\alpha\leq m$, $2\leq q\leq\f{2m+6+2\alpha}{m+1}$, $0<\gamma_1<\frac{m+1}{m+2}-\frac{m+4+2\alpha}{(m+2)q}$,
$\gamma_2>\frac{1}{q}$ and $C>0$ is a constant  depending on $m$, $q$, $\g_1$, $\gamma_2$ and \(\al\).
By Stein's analytic interpolation theorem (see \cite{S}), in order to prove \eqref{equ:3.5-L},
it suffices to derive the following two endpoint cases of \(q=q_0\) with $q_0=\f{2m+6+2\alpha}{m+1}$
(see Remark \ref{JY-8-5} of Section \ref{E-6} for the derivation of $q_0$) and \(q=2\):
\begin{equation}
\begin{split}
\big\|\big(\phi_m(t)^2&-|x|^2\big)^{\gamma_1}t^{\frac{\alpha}{q_0}}w\big\|_{L^{q_0}(\ra)}\\
&\leq C
\big\|\big(\phi_m(t)^2-|x|^2\big)^{\gamma_2}t^{-\frac{\alpha}{q_0}}F\big\|_{L^{\frac{q_0}{q_0-1}}(\ra)}\quad
\text{for $\gamma_1<\frac{1}{q_0}<\gamma_2$}
\end{split}
\label{equ:3.3-L}
\end{equation}
and
\begin{equation}
\begin{split}
\big\|\big(\phi_m(t)^2&-|x|^2\big)^{\gamma_1}t^{\frac{\alpha}{2}}w\big\|_{L^2(\ra)}\\
&\leq C
\big\|\big(\phi_m(t)^2-|x|^2\big)^{\gamma_2}t^{-\frac{\alpha}{2}}F\big\|_{L^2(\ra)}\quad \text{for $\gamma_1<-\f{1}{2}+\f{m-\al}{m+2}$ and  $\gamma_2>\f{1}{2}$}.
\end{split}
\label{equ:3.4.1-L}
\end{equation}

To show \eqref{equ:3.3-L}, as in  \cite{Gls1} and \cite{HWY4}, we shall split the integral
domain $\{(t,x): \phi^2_m(t)-|x|^2\geq 1\}$
in the related Fourier integral operators into the ``relatively small time" part and the ``relatively large time"
part. The ``relatively small time"  and ``relatively large time" mean $T\le t\le 2T$ with $1\le T\le 2\cdot 10^{\f{2}{m+2}}$
and $T\ge 2\cdot 10^{\f{2}{m+2}}$, respectively.

For the case of ``relatively small time", it is enough to establish the following Strichartz
estimate with time-weight:
\begin{equation}\label{equ:un-chara-L}
	\big\|t^{\frac{\alpha}{q_0}}w\big\|_{L^{q_0}(\ra)}\leq C
\big\|t^{-\frac{\alpha}{q_0}}F\big\|_{L^{\frac{q_0}{q_0-1}}(\ra)}.
\end{equation}
This will be given in terms of the stationary phase method together the dynamic decomposition.

For the case of ``relatively large time", the analysis is more involved. We will divide the related integral domain
into two parts according to the scale of \(\phi_m(t)-|x|\):

{\bf Case A.}  \(|\phi_m(t)-|x||\) is very small or very large;

{\bf Case B.} \(|\phi_m(t)-|x||\) lies in the medium scale.

With respect to Case A, \eqref{equ:3.3-L} comes from the delicate analysis by use of the support conditions of \(w\) and \(F\).
While for Case B, motivated by  \cite{Gls1} and \cite{HWY4},
it is required to introduce two kinds of Fourier integral operators
$(T_zg)(t,x)$ and $(\tilde{T}_zg)(t,x)$ (see their expressions in \eqref{a-67}
and \eqref{a-79} below). Based on the careful treatments on  $(T_zg)(t,x)$ and $(\tilde{T}_zg)(t,x)$
in terms of the different range of \(\alpha\),
and together with the complex interpolation methods,  \eqref{equ:3.3-L} can be obtained.

For the $L^2$ estimate \eqref{equ:3.4.1-L}, similarly to the proof of \eqref{equ:3.3-L},
the integral domain $\{(t,x): \phi^2_m(t)-|x|^2\geq 1\}$ is still decomposed into
the ``relatively small time" part and the ``relatively large time" part.
In the ``relatively large time" part, we will utilize the
more precise weight $(\phi_{m}(t)+|x|)^{\frac{1}{2}}(\phi_{m}(t)-|x|)^{\frac{1}{2}-\frac{1}{m+2}-\epsilon}$
($\epsilon>0$ is any small fixed constant)
instead of the weight $(\phi_{m}^2(t)-|x|^2)^{\frac{1}{2}}$ in Section 5B1 of \cite{HWY4}
to obtain the better time-decay rate of $w$
(see \eqref{SZH}-\eqref{a-94} in Subsection \ref{Sub-4.2} below).
This treatment leads to the establishment of the required spacetime-weighted Strichartz inequality
for 2-D problem \eqref{a-1-M}, which is different from that in \cite{HWY4} only for the cases of space dimensions $n\ge 3$
and $m\ge 1$.

When \eqref{a-2-M} and \eqref{equ:3.5-L} are derived, by the  standard Picard iteration and
the contractible mapping principle,
we can complete the proof of Theorem~\ref{thm:global existence-L}.

\subsection{Sketch of proof on Theorem~\ref{thm:global existence-LL}}

To prove Theorem~\ref{thm:global existence-LL},
when $0<m\leq \sqrt{2}-1$, motivated by \cite{SSW} and \cite{HWY2}, we will establish some new kinds of angular mixed-norm
Strichartz inequality
for the 2-D problem
\begin{equation}\label{q-3-L}
\begin{cases}\partial_t^2\t w-t^{m} \Delta\t w=F(t, x), \\
\t w(1, x)=u_0(x),\quad \p_t \t w(1, x)=u_1(x).
\end{cases}
\end{equation}
Under the polar coordinate $(x_1,x_2)=(r cos\theta, r sin\theta)$ with
$(r,\theta)\in [0,\infty)\times [0,2\pi]$, we can conclude
\begin{equation}\label{q-11-L}
\begin{aligned}
&\|t^{\f{m-p+1}{p+1}}\t w\|_{L_{t}^{q} L_{r}^{\nu} L_{\theta}^{2}([1,+\infty) \times [0, \infty)\times [0, 2\pi])}\\
&\leq C(\|u_0\|_{\dot{H}^{s}\left(\mathbb{R}^{2}\right)}+\|u_1\|_{\dot{H}^{s-\f{2}{m+2}}\left(\mathbb{R}^{2}\right)}
+\|t^{\f{p-m-1}{p+1}}F\|_{L_{t}^{\tilde{q}^{\prime}} L_{r}^{\tilde{\nu}^{\prime}} L_{\theta}^{2}([1,\infty)\times[0, \infty)\times [0, 2\pi])}),
\end{aligned}
\end{equation}
where $m+2\leq p<\f{3m+7}{m+3}$,  $s=1-\f{2}{\nu}-\f{2}{m+2}(\f{1}{q}+\f{m-p+1}{p+1})$ with
$(q, \nu) \neq(\infty, \infty)$ satisfying
\begin{equation}\label{q-10-L}
\left\{\enspace
\begin{aligned}
&q, \nu \geq 2, \quad \f{1}{\nu}+\f{m-p+1}{(m+2)(p+1)}>0,\\
&\frac{1}{q} \leq \f{m+1}{2}-\frac{m+2}{2\nu}-\f{m-p+1}{p+1},\\
\end{aligned}
\right.
\end{equation}
meanwhile $(\tilde{q}^{\prime}, \tilde{\nu}^{\prime})$ fulfills $\f{1}{\tilde{q}^{\prime}}+\f{1}{\t q}=1$
and $\f{1}{\tilde{\nu}^{\prime}}+\f{1}{\t\nu}=1$ with $(\tilde{q}, \tilde{\nu})$ satisfying
\begin{equation}\label{q-11-LL}
\left\{\enspace
\begin{aligned}
&\tilde{q}, \tilde{\nu} \geq 2, \quad \frac{1}{\tilde{q}} \leq \f{m+1}{2}-\frac{m+2}{2\tilde{\nu}}-\f{m-p+1}{p+1},\\
&\frac{1}{q}+\frac{1}{\t q}+(m+2)(\f{1}{\nu}+\f{1}{\t\nu})+\f{2(m-p+1)}{p+1}=m+1.\\
\end{aligned}
\right.
\end{equation}
To prove \eqref{q-11-L}, a series of inequalities are
derived by utilizing some interpolation methods between different Besov spaces and
applying an explicit formula for the solution $\t w$ of \eqref{q-3-L}.

Based on the crucial estimate \eqref{q-11-L} with the restriction conditions \eqref{q-10-L}-\eqref{q-11-LL},
through suitable choices of $(q, \nu)$ and $(\tilde{q}^{\prime}, \tilde{\nu}^{\prime})$,
Theorem~\ref{thm:global existence-LL} will be shown by the standard Picard iteration, H\"older inequality
and the contractible mapping principle.

When Theorem~\ref{thm:global existence-L} and Theorem~\ref{thm:global existence-LL}
are proved, by returning to the corresponding equation $\p_t^2u-\Delta u+\f{\mu}{t}\p_tu=|u|^p$ in \eqref{00-1},
Theorem \ref{TH-1} and Theorem \ref{TH-2} may be obtained correspondingly.

The paper is organized as follows.
In Section \ref{E-7}, some preliminaries such as the explicit expressions of solutions to 2-D homogenous
or inhomogenous Tricomi equations and some useful lemmas are given.
In Section \ref{E-8}, some weighted Strichartz estimates for homogenous Tricomi equations are established.
On the other hand, for inhomogenous Tricomi equations, we will show the time-weighted Strichartz estimates,
spacetime-weighted Strichartz estimates and spacetime-weighted angular mixed-norm Strichartz estimates
in Sections \ref{E-6}-\ref{E-5}, respectively. In Section \ref{E-1},
Theorem~\ref{thm:global existence-L} and  Theorem~\ref{TH-1} are proved.
We will complete the proofs of Theorem~\ref{thm:global existence-LL} and  Theorem~\ref{TH-2} in Section \ref{E-2}.
In addition, some useful estimates are given in Appendix.

\section{Preliminaries}\label{E-7}

In this section, for reader's convenience, in terms of \cite{Yag2} and \cite{Rua3}-\cite{Rua4},
we give the explicit expressions of solutions to 2-D homogenous
or inhomogenous generalized Tricomi equations. In addition, some useful lemmas are listed.

\subsection{Explicit solutions of homogenous generalized Tricomi equations}\label{Sub-1}

Let $v$ solve the 2-D problem
\begin{equation}\label{a-1}
\begin{cases}
&\partial_t^2 v-t^{m} \Delta v=0, \\
&v(1, x)=u_0(x),\quad \p_t v(1, x)=u_1(x),
\end{cases}
\end{equation}
where $m>0$, $t\ge 1$, $x=(x_1, x_2)$, and $(u_0, u_1)$ is given in \eqref{00-1}.
It follows from \cite{Yag2} or \cite{Rua3}-\cite{Rua4} that
\begin{equation}\label{SW-1}
\begin{split}
v(t,x)=V_1(t, D_x)u_0(x)+V_2(t, D_x)u_1(x),
\end{split}
\end{equation}
where the symbols $V_j(t, \xi)$ ($j=1,2$) of the Fourier integral operators $V_j(t, D_x)$  are
\begin{equation}\label{equ:2.3}
\begin{split}
V_1(t,|\xi|)=&\frac{\Gamma(\frac{m}{m+2})}{\Gamma(\frac{m}{2(m+2)})}\bigl[e^{\frac{z}{2}}H_+\big(\frac{m}{2(m+2)},\frac{m}{m+2};z\big) +e^{-\frac{z}{2}}H_-\big(\frac{m}{2(m+2)},\frac{m}{m+2};z\big)\bigr]
\end{split}
\end{equation}
and
\begin{equation}
\begin{split}
V_2(t,|\xi|)=&\frac{\Gamma(\frac{m+4}{m+2})}{\Gamma(\frac{m+4}{2(m+2)})}t\bigl[
e^{\frac{z}{2}}H_+\big(\frac{m+4}{2(m+2)},\frac{m+4}{m+2};z\big)
+e^{-\frac{z}{2}}H_-\big(\frac{m+4}{2(m+2)},\frac{m+4}{m+2};z\big)\bigr],
\end{split}
\label{equ:2.4}
\end{equation}
here $z=2i\phi_m(t)|\xi|$ with $\xi=(\xi_1,\xi_2)$, $i=\sqrt{-1}$, and $H_{\pm}$ are smooth functions of the variable $z$.
By \cite{Tani}, it is known that for $\beta\in\mathbb{N}_0^2$,  $\alpha>0$ and $\gamma>0$,
\begin{align}
\big| \partial_\xi^\beta H_{+}(\alpha,\gamma;z)
\big|&\leq C_{\alpha,\gamma}(\phi_m(t)|\xi|)^{\alpha-\gamma}(1+|\xi|^2)^{-\frac{|\beta|}{2}}
\quad if \quad \phi_m(t)|\xi|\geq 1, \label{equ:2.5} \\
\big| \partial_\xi^\beta H_{-}(\alpha,\gamma;z)\big|&\leq C_{\alpha,\gamma}(\phi_m(t)|\xi|)^{-\alpha}(1+|\xi|^2)^{-\frac{|\beta|}{2}}
\quad if \quad \phi_m(t)|\xi|\geq 1. \label{equ:2.6}
\end{align}
We now give illustrations on $V_1(t, D_x)u_0(x)$. Let $\chi(s)\in C^{\infty}(\Bbb R)$
with
\begin{align}\label{L-1}
\chi(s)=
\left\{ \enspace
\begin{aligned}
1 \qquad &\text{for $s\geq2$,} \\
0 \qquad &\text{for $s\leq1$.}
\end{aligned}
\right.
\end{align}
Then
\begin{equation}
\begin{split}
V_1(t,|\xi|)\hat{u_0}(\xi)&=\chi(\phi_m(t)|\xi|)V_1(t,|\xi|)\hat{u_0}(\xi)+(1-\chi(\phi_m(t)|\xi|))V_1(t,|\xi|)\hat{u_0}(\xi) \\
&=:\hat{v}_1(t,\xi)+\hat{v}_2(t,\xi).
\end{split}
\label{equ:2.7}
\end{equation}
By \eqref{equ:2.3}, \eqref{equ:2.5} and \eqref{equ:2.6}, one can derive that
\begin{equation}
{v}_1(t,x)=C_m\biggl(\int_{\mathbb{R}^2}e^{i(x\cdot\xi+\phi_m(t)|\xi|)}a_{11}(t,\xi)\hat{u_0}(\xi)\md\xi+
\int_{\mathbb{R}^2}e^{i(x\cdot\xi-\phi_m(t)|\xi|)}a_{12}(t,\xi)\hat{u_0}(\xi)\md\xi\biggr), \label{equ:2.8}
\end{equation}
where $C_m>0$ is a generic constant depending on $m$, and for $\beta\in\mathbb{N}_0^2$,
\begin{equation*}
\big| \partial_\xi^\beta a_{1l}(t,\xi)\big|\leq C_{l\beta}|\xi|^{-|\beta|}\big(1+\phi_m(t)|\xi|\big)^{-\frac{m}{2(m+2)}},
\qquad l=1,2.
\end{equation*}
We next analyze $v_2(t,x)$. It follows from \cite{Erd1} or \cite{Yag2} that
\begin{equation*}
V_1(t,|\xi|)=e^{-\frac{z}{2}}\Phi\Big(\frac{m}{2(m+2)},\frac{m}{m+2};z\Big), 
\end{equation*}
where the confluent hypergeometric function $\Phi$ is analytic on the variable
$z=2i\phi_m(t)|\xi|$. Note that
\begin{equation*}
\Big|\partial_\xi\big\{\big(1-\chi(\phi_m(t)|\xi|)\big)V_1(t,|\xi|)\big\}\Big|\leq C(1+\phi_m(t)|\xi|)^{-\frac{m}{2(m+2)}}|\xi|^{-1}
\end{equation*}
and
\begin{equation*}
\Big|\partial_\xi^{\beta}\big\{\big(1-\chi(\phi_m(t)|\xi|)\big)V_1(t,|\xi|)\big\}\Big|\leq
C(1+\phi_m(t)|\xi|)^{-\frac{m}{2(m+2)}}|\xi|^{-|\beta|}.
\end{equation*}
Then
\begin{equation}
v_2(t,x)=C_m\biggl(\int_{\mathbb{R}^2}e^{i(x\cdot\xi+\phi_m(t)|\xi|)}a_{21}(t,\xi)\hat{u_0}(\xi)\md\xi
+\int_{\mathbb{R}^2}e^{i(x\cdot\xi-\phi_m(t)|\xi|)}a_{22}(t,\xi)\hat{u_0}(\xi)\md\xi\biggr), \label{equ:2.10}
\end{equation}
where
\begin{equation*}
\big| \partial_\xi^\beta a_{2l}(t,\xi)\big|\leq C_{l\beta}\big(1+\phi_m(t)|\xi|\big)^{-\frac{m}{2(m+2)}}|\xi|^{-|\beta|},
\qquad l=1,2.
\end{equation*}

Substituting \eqref{equ:2.8} and \eqref{equ:2.10} into \eqref{equ:2.7} yields
\[V_1(t, D_x)u_0(x)=C_m\biggl(\int_{\mathbb{R}^2}e^{i(x\cdot\xi+\phi_m(t)|\xi|)}a_1(t,\xi)\hat{u_0}(\xi)\md\xi
+\int_{\mathbb{R}^2}e^{i(x\cdot\xi-\phi_m(t)|\xi|)}a_2(t,\xi)\hat{u_0}(\xi)\md\xi\biggr),\]
where $a_l$ $(l=1,2)$ satisfies
\begin{equation}
|\partial_\xi^\beta a_l(t,\xi)\big|\leq C_{l\beta}\big(1+\phi_m(t)|\xi|\big)^{-\frac{m}{2(m+2)}}|\xi|^{-|\beta|}. \label{equ:2.11}
\end{equation}

Similarly,
$$V_2(t,D_x)u_1(x)=C_m t\left(\int_{\R^2}
e^{i\left(x\cdot\xi+\phi_m(t)|\xi|\right)}b_1(t,\xi)\hat{u_1}(\xi)\,d\xi
+\int_{\R^2}e^{i\left(x\cdot\xi-\phi_m(t)|\xi|\right)}b_2(t,\xi)
\hat{u_1}(\xi)\,d\xi\right)
$$
with $b_l$ ($l=1,2$) satisfying
\begin{equation}\label{equ:3.16'}
  \bigl|\partial_\xi^\beta b_l(t,\xi)\bigr|\leq
  C_{l\beta}\left(1+\phi_m(t)|\xi|\right)^{-\f{m+4}{2(m+2)}}
  |\xi|^{-|\beta|}.
\end{equation}

\subsection{Explicit solutions of inhomogenous generalized Tricomi equations}\label{Sub-2}

Let $w$ solve the 2-D inhomogeneous problem
\begin{equation}\label{equ:3.3}
\left\{ \enspace
\begin{aligned}
&\partial_t^2 w-t^m\triangle w=F(t,x), \quad t\ge 1,\\
&w(1,x)=0,\quad \partial_tw(1,x)=0.
\end{aligned}
\right.
\end{equation}
Then by Duhamel's principle and \eqref{SW-1}, one has
\begin{equation*}
\begin{split}
w(t,x)=\int_1^t\left(V_2(t, D_x)V_1(\tau, D_x)-V_1(t, D_x)V_2(\tau,
D_x)\right)F(\tau,x)\,\mathrm{d}\tau.
\end{split}
\end{equation*}
Now we only deal with the term $\int_1^tV_2(t,
D_x)V_1(\tau, D_x)F(\tau,x)d\tau$ since the treatment on the term
$\int_1^tV_1(t, D_x)V_2(\tau, D_x)F(\tau,x)d\tau$ is completely
analogous. Set
\begin{align*}
w_1(t,x)&=\int_1^t\chi(\phi_m(t)D_x)\chi(\phi_m(\tau)D_x)V_2(t,
D_x)V_1(\tau, D_x)F(\tau,x)\,\mathrm{d}\tau, \\
w_2(t,x)&=\int_1^t\chi(\phi_m(t)D_x)\left(1-\chi(\phi_m(\tau)D_x)\right)
V_2(t, D_x)V_1(\tau, D_x)F(\tau,x)\,\mathrm{d}\tau, \\
w_3(t,x)&=\int_1^t\left(1-\chi(\phi_m(t)D_x)\right)\chi(\phi_m(\tau)D_x)V_2(t,
D_x)V_1(\tau, D_x)F(\tau,x)\,\mathrm{d}\tau, \\
w_4(t,x)&=\int_1^t\left(1-\chi(\phi_m(t)D_x)\right)
\left(1-\chi(\phi_m(\tau)D_x)\right)V_2(t, D_x)V_1(\tau, D_x)F(\tau,x)\,\mathrm{d}\tau,
\end{align*}
where $\chi(s)$ is defined in \eqref{L-1}. As in Subsection \ref{Sub-1},
by analyzing each $w_j$ ($1\le j\le 4$), we can obtain
\begin{equation}\label{SW-2}
w(t,x)=
\int_1^t\int_{\R^2}e^{i\left(x\cdot\xi-(\phi_m(t)-\phi_m(\tau))|\xi|\right)}
a(t,\tau,\xi)\hat{F}(\tau,\xi)\,\mathrm{d}\xi \mathrm{d}\tau,
\end{equation}
where $a(t,\tau,\xi)$ satisfies
\begin{equation}\label{NN}
\begin{aligned}
\bigl| \partial_\xi^\beta a(t,\tau,\xi)\bigr|\leq
C\left(1+\phi_m(t)|\xi|\right)^{-\f{m}{2(m+2)}}
\left(1+\phi_m(\tau)|\xi|\right)^{-\f{m}{2(m+2)}}|\xi|^{-\frac{2}{m+2}-|\beta|}.
\end{aligned}
\end{equation}

\subsection{Some useful Lemmas}
In this part, we recall some useful results which will be utilized later.
Assume $\vp \in C_{0}^{\infty}(\mathbb{R}^+)$, $\vp\ge 0$,  $\vp\not\equiv 0$, and
\begin{equation}\label{chi}
\operatorname{supp} \vp \subseteq(\frac{1}{2}, 2), \quad \sum_{j=-\infty}^{\infty} \vp\left(2^{-j} \tau\right) \equiv 1 \quad \text { for } \quad \tau>0.
\end{equation}
Define the Littlewood-Paley decomposition of function $f(t,x)$ with respect to the space variable $x\in\Bbb R^2$ as
$$
f_{j}(t, x)=\vp(2^{-j}|D|)f=(2 \pi)^{-2} \int_{\mathbb{R}^{2}} e^{i x \cdot \xi} \vp(\frac{|\xi|}{2^{j}}) \hat{f}(t, \xi) \mathrm{d} \xi.
$$

By  \cite[Lemma 3.8]{LS}, one has
\begin{lemma}\label{lem3-E}
It holds that
$$
\|f\|_{L_{t}^{s} L_{x}^{q}} \leq C\big(\sum_{j=-\infty}^{\infty}\left\|f_{j}\right\|_{L_{t}^{s} L_{x}^{q}}^{2}\big)^{\frac{1}{2}} \quad \text { for } \quad 2 \leq q<\infty \text { and } 2 \leq s \leq \infty
$$
and
$$
\big(\sum_{j=-\infty}^{\infty}\left\|f_{j}\right\|_{L_{t}^{r} L_{x}^{p}}^{2}\big)^{\frac{1}{2}} \leq C\|f\|_{L_{t}^{r} L_{x}^{p}} \quad \text { for } \quad 1<p \leq 2 \quad \text { and } \quad 1 \leq r \leq 2.
$$
\end{lemma}

In addition, although (2-20) of \cite{HWY4} is suitable for $n\ge 3$,
by checking the proof procedure carefully, we still have
\begin{lemma}\label{lem2-L}
For the solution $v$ of 2-D problem \eqref{a-1}, it holds that for any small fixed constant $\delta>0$,
\begin{equation}\label{equ:2.22-LL}
\begin{split}
|v(t,x)|\leq C_{m,\delta}(1+\phi_m(t))^{-\frac{m+1}{m+2}}(1+\big||x|-\phi_m(t)\big|)^{-\frac{m+1}{m+2}+\delta}
\big(\|u_0\|_{W^{\frac{m+3}{m+2}+\delta,1}(\mathbb{R}^2)}
+\|u_1\|_{W^{\frac{m+1}{m+2}+\delta,1}(\mathbb{R}^2)}\big),
\end{split}
\end{equation}
where $C_{m,\delta}>0$ is a constant depending on $m$ and $\delta$.
\end{lemma}

\section{Weighted Strichartz estimates for homogenous Tricomi equations}\label{E-8}

In this section, we will establish two classes of weighted Strichartz estimates for 2-D homogenous problem \eqref{a-1}.

\subsection{Spacetime-weighted Strichartz estimate}
\begin{lemma}\label{th2-1}
Let $v$ be a solution of \eqref{a-1}.
Then it holds that
\begin{equation}\label{a-2}
\begin{aligned}
&\|(\left(\phi_{m}(t)+2\right)^{2}-|x|^{2})^{\gamma} t^{\frac{\alpha}{q}} v\|_{L^{q}\left(\left[1, +\infty\right) \times \mathbb{R}^{2}\right)} \\
& \leq C(\|u_0\|_{W^{\frac{m+3}{m+2}+\delta, 1}\left(\mathbb{R}^{2}\right)}+\|u_1\|_{W^{\frac{m+1}{m+2}+\delta, 1}\left(\mathbb{R}^{2}\right)}),
\end{aligned}
\end{equation}
where $\alpha>-1$, $p_{\text {crit }}(2,m, \alpha)+1<q<p_{\text {conf}}(2,m,\alpha)+1$,
$C=C(m,q,\al,\gamma,\delta)>0$,
the positive constants $\gamma$ and $\delta$ fulfill
\begin{equation}\label{con3}
0<\gamma<\frac{m+1}{m+2}-\frac{m+4+2\al}{(m+2) q}, \quad 0<\delta<\frac{m+1}{m+2}-\gamma-\frac{1}{q}.
\end{equation}
\end{lemma}
\begin{proof}
By Lemma \ref{lem2-L}, $v$ satisfies that for  $\alpha>-1$ and $t\geq1$,
\begin{equation}\label{a-3}
\begin{aligned}
|t^{\frac{\alpha}{q}} v| &\leq  C \phi_{m}(t)^{-\frac{m+1}{m+2}+\frac{2 \alpha}{q(m+2)}}\left(1+\left||x|-\phi_{m}(t)\right|\right)^{-\frac{m+1}{m+2}+\delta} \\
& \qquad\times\big(\|u_0\|_{W^{\frac{m+3}{m+2}+\delta, 1}\left(\mathbb{R}^{2}\right)}
+\|u_1\|_{W^{\frac{m+1}{m+2}+\delta, 1}\left(\mathbb{R}^{2}\right)}\big).
\end{aligned}
\end{equation}
Then it follows from \eqref{con3}, $\alpha>-1$ and $t\geq1$ that
\begin{equation}\label{a-4}
\begin{aligned}
&\|(\left(\phi_{m}(t)+2\right)^{2}-|x|^{2})^{\gamma} t^{\frac{\alpha}{q}} v\|_{L^{q}\left(\left[1, \infty\right) \times \mathbb{R}^{2}\right)}^{q} \\
&\le C\int_{1}^{\infty} \phi_{m}(t)^{q\left(-\frac{m+1}{m+2}+\frac{2 \alpha}{q(m+2)}+\gamma\right)} \int_{0}^{\phi_{m}(t)+1}\left(1+\left|r-\phi_{m}(t)\right|\right)^{q\left(\gamma-\frac{m+1}{m+2}+\delta\right)} r \mathrm{d} r \mathrm{d} t\\
& \qquad\times\big(\|u_0\|_{W^{\frac{m+3}{m+2}+\delta, 1}\left(\mathbb{R}^{2}\right)}
+\|u_1\|_{W^{\frac{m+1}{m+2}+\delta, 1}\left(\mathbb{R}^{2}\right)}\big)\\
&\leq C \big(\|u_0\|_{W^{\frac{m+3}{m+2}+\delta, 1}\left(\mathbb{R}^{2}\right)}
+\|u_1\|_{W^{\frac{m+1}{m+2}+\delta, 1}\left(\mathbb{R}^{2}\right)}\big).
\end{aligned}
\end{equation}
Thus, the proof of \eqref{a-2}  is completed.
\end{proof}

\subsection{Time-weighted angular mixed-norm Strichartz estimate}

For the solution $v$ of \eqref{a-1}, it follows from \eqref{SW-1} that
\[
v(t,x)=V_1(t, D_x)u_0(x)+V_2(t, D_x)u_1(x).
\]
In terms of \eqref{equ:2.11} and \eqref{equ:3.16'},
it suffices  to deal with
$\int_{\R^2}e^{i\left(x\cdot\xi-\phi_m(t)|\xi|\right)}
a_2(t,\xi)\hat{u_0}(\xi)\mathrm{d}\xi$ since the remaining part of
$V_1(t,D_x)u_0(x)$ and $V_2(t,D_x)u_1(x)$ may be similarly or simpler treated. Set
\begin{equation}\label{equ:3.17}
t^{\f{m-p+1}{p+1}}(Au_0)(t,x)=\int_{\R^2}e^{i\left(x\cdot\xi-\phi_m(t)|\xi|\right)}t^{\f{m-p+1}{p+1}} a_2(t,\xi)
  \hat{u_0}(\xi)\,\mathrm{d}\xi.
\end{equation}
We next show
\begin{equation}\label{q-6}
\|t^{\f{m-p+1}{p+1}}(A u_0)(t, x)\|_{L_{t}^{q} L_{r}^{\nu} L_{\theta}^{2}([1,+\infty) \times [0,\infty)\times [0, 2\pi])}
\leq C\|u_0\|_{\dot{H}^{s}\left(\mathbb{R}^{2}\right)},
\end{equation}
where $q \geq 2$ and $\nu \geq 2$ are  suitable constants related to $s$.
It follows from a scaling argument that
\begin{equation}\label{q-7}
\f{m-p+1}{p+1}+\frac{1}{q}+\frac{m+2}{\nu}=\frac{m+2}{2}(1-s).
\end{equation}
On the other hand, a different scaling argument (see Page 1845 of  \cite{SK} or \cite{HWY2})
yields
\begin{equation}\label{q-8}
\frac{1}{q} \leq \f{m+1}{2}-\frac{m+2}{2\nu}-\f{m-p+1}{p+1}.
\end{equation}

In fact, for small $\ve>0$, let $\xi=\left(\xi_{1}, \xi_{2}\right) \in \mathbb{R}^{2} \backslash\{0\}$ and
$D_{\ve}=\left\{\xi \in \mathbb{R}^{2}:\left|\xi_{1}-1\right|<1 / 2,\left|\xi_{2}\right|<\ve\right\}$.
Let $\hat{u_0}(\xi)=\chi_{D_{\ve}}(\xi)$ with $\chi_{D_{\ve}}(\xi)$ being the characteristic function of domain $D_{\ve}$.
Note that for $\xi\in D_{\ve}$, $|\xi|-\xi_{1}=\frac{|\xi|^{2}-\xi_{1}^{2}}{|\xi|+\xi_{1}}
=\frac{\left|\xi_{2}\right|^{2}}{|\xi|+\xi_{1}} \sim \ve^{2}$
holds. Due to
\begin{equation}\label{q-9}
t^{\f{m-p+1}{p+1}}(A u_0)(t, x) =e^{i\left(x_{1}-\phi_m(t)\right)} \int_{D_{\ve}} e^{i\left(-\phi_m(t)\left(|\xi|-\xi_{1}\right)+\left(x_{1}-\phi_m(t)\right)\left(\xi_{1}-1\right)+x_{2} \xi_{2}\right)} t^{\f{m-p+1}{p+1}}a_{2}(t, \xi) \mathrm{d} \xi,
\end{equation}
then for $(t,x)\in R_{\ve}=\left\{(t, x): \phi_m(t) \sim \ve^{-1},\left|x_{1}-\phi_m(t)\right| \lesssim 1,\left|x_{2}\right| \lesssim \ve^{-1}\right\}$ and $\xi \in D_{\ve}$, one has
$$
t^{\f{m-p+1}{p+1}}|(A u_0)(t, x)| \geq|D_{\ve}| \ve^{\f{m}{2(m+2)}-\f{2(m-p+1)}{(m+2)(p+1)}}.
$$
By choosing  $s=0$ in \eqref{q-6}, then
\begin{equation*}
\begin{aligned}
&\f{\|t^{\f{m-p+1}{p+1}}(A u_0)(t, x)\|_{L_{t}^{q} L_{r}^{\nu} L_{\theta}^{2}([1,+\infty) \times [0,\infty)\times [0, 2\pi])}}{\|u_0\|_{L^{2}\left(\mathbb{R}^{2}\right)}}\\
&\geq\f{|D_{\ve}|\ve^{\f{m}{2(m+2)}-\f{2(m-p+1)}{(m+2)(p+1)}}
\|\chi_{R_{\ve}}\|_{L_{t}^{q} L_{r}^{\nu} L_{\theta}^{2}([1,+\infty) \times [0,\infty)\times [0, 2\pi])} }{{|D_{\ve}|}^\f{1}{2}}\\
&\sim\ve^{\f{m+1}{m+2}-\f{2}{(m+2)q}-\f{1}{\nu}-\f{2(m-p+1)}{(m+2)(p+1)}}.
\end{aligned}
\end{equation*}
Assume that $\ve>0$ is small, when \eqref{q-6} holds, we require
$$
\f{m+1}{m+2}-\f{2}{(m+2)q}-\f{1}{\nu}-\f{2(m-p+1)}{(m+2)(p+1)} \geq 0,
$$
which is equivalent to \eqref{q-8}.

\begin{lemma}\label{lem1}
Under the restrictions \eqref{q-7} and \eqref{q-8}, it holds that
\begin{equation}\label{q-11}
\|t^{\f{m-p+1}{p+1}}(A u_0)(t, x)\|_{L_{t}^{q} L_{r}^{\nu} L_{\theta}^{2}([1,+\infty)\times [0,\infty)\times [0, 2\pi])} \leq C\|u_0\|_{\dot{H}^{s}\left(\mathbb{R}^{2}\right)},
\end{equation}
where $(q, \nu) \neq(\infty, \infty)$ satisfies
\begin{equation}\label{q-10}
q, \nu \geq 2, \quad \f{1}{\nu}+\f{m-p+1}{(m+2)(p+1)}>0,
\end{equation}
\begin{equation}\label{q-10'}
\frac{1}{q} \leq \f{m+1}{2}-\frac{m+2}{2\nu}-\f{m-p+1}{p+1}, \quad -2<m-p+1\leq-1,
\end{equation}
and $s=1-\f{2}{\nu}-\f{2}{m+2}(\f{1}{q}+\f{m-p+1}{p+1})$.
\end{lemma}
\begin{proof}
Although the proof procedure is similar to that of Lemma 3.2 in \cite{HWY2} for the equation $\p_t^2v-t\Delta v=0$,
we  still give the details for reader's convenience due to the more general Tricomi equation $\p_t^2v-t^m\Delta v=0$
and the appearance of the time weight $t^{\f{m-p+1}{p+1}}$ in \eqref{q-11}.

To get \eqref{q-11}, it suffices to derive the two endpoint cases of $q=\infty, 2 \leq \nu<\infty$ and $\nu=\infty, 2 \leq q<\infty$.
If so, then \eqref{q-11} can be shown by interpolation between these two cases.

Note that the first endpoint case corresponds to $q=\infty$ and $s=1-\frac{2}{\nu}-\f{2(m-p+1)}{(m+2)(p+1)}$.
By $\dot{H}^{1-\frac{2}{\nu}}\left(\mathbb{R}^{2}\right) \subseteq L^{\nu}\left(\mathbb{R}^{2}\right)$ for $2 \leq \nu<\infty$,
then
$$
\begin{aligned}
\|t^{\f{m-p+1}{p+1}}A u_0\|_{L_{t}^{\infty} L_{r}^{\nu} L_{\theta}^{2}\left([1,\infty)\times [0,\infty)\times [0,2\pi]\right)} &\leq C\|t^{\f{m-p+1}{p+1}}A u_0\|_{L_{t}^{\infty} L_{r}^{\nu} L_{\theta}^{\nu}\left([1,\infty)\times [0,\infty)\times [0,2\pi]\right)} \\&\leq C\|t^{\f{m-p+1}{p+1}}A u_0\|_{L_{t}^{\infty} \dot{H}^{1-\frac{2}{\nu}}\left([1,\infty)\times \R^2\right)}.
\end{aligned}
$$
It follows from \eqref{equ:2.11} and \eqref{equ:3.17} that if $\hat{u_0}(\xi)=0$ for $|\xi| \notin\left[\frac{1}{2}, 1\right]$, then
$$
\begin{aligned}
\|t^{\f{m-p+1}{p+1}}A u_0&\|_{L_{t}^{\infty} \dot{H}^{1-\frac{2}{\nu}}\left([1,\infty)\times \R^2\right)}
\leq C\big\| \| \int_{\mathbb{R}^{2}} e^{i(x \cdot \xi-\phi_m(t)|\xi|)} t^{\f{m-p+1}{p+1}}a_{2}(t, \xi) \hat{u_0}(\xi) \mathrm{d} \xi\|_{\dot{H}^{1-\frac{2}{\nu}}\left(\mathbb{R}^{2}\right)}\big\|_{L_{t}^{\infty}[1, \infty)} \\
& \leq C\big\|\||\xi|^{1-\frac{2}{\nu}-\f{2(m-p+1)}{(m+2)(p+1)}}(1+\phi_m(t)|\xi|)^{\f{2(m-p+1)}{(m+2)(p+1)}
-\f{m}{2(m+2)}}|\hat{u_0}(\xi)|\|_{L^{2}\left(\mathbb{R}^{2}\right)}\big\|_{L_{t}^{\infty}[1, \infty)} \\
& \leq C\|u_0\|_{\dot{H}^{1-\frac{2}{\nu}-\f{2(m-p+1)}{(m+2)(p+1)}}\left(\mathbb{R}^{2}\right)}.
\end{aligned}
$$
This, together with \eqref{q-10}, yields
\begin{equation}\label{q-12}
\|t^{\f{m-p+1}{p+1}}A u_0\|_{L_{t}^{\infty} L_{r}^{\nu} L_{\theta}^{2}\left([1,+\infty)\times [0, +\infty)\times [0, 2\pi]\right)} \leq C\|u_0\|_{\dot{H}^{1-\frac{2}{\nu}-\f{2(m-p+1)}{(m+2)(p+1)}}\left(\mathbb{R}^{2}\right)}
=\|u_0\|_{\dot{B}_{2,2}^{1-\frac{2}{\nu}-\f{2(m-p+1)}{(m+2)(p+1)}}\left(\mathbb{R}^{2}\right)}.
\end{equation}

Next we show  \eqref{q-11} with the second  endpoint  case of $\nu=\infty, 2 \leq q<\infty$.
For this purpose, it suffices to prove that for $2 \leq q<\infty$,
\begin{equation}\label{q-13}
\|t^{\f{m-p+1}{p+1}}A u_0\|_{L_{t}^{q} L_{r}^{\infty} L_{\theta}^{2}\left([1,+\infty)\times [0, +\infty)\times [0, 2\pi]\right)} \leq C\|u_0\|_{L^{2}\left(\mathbb{R}^{2}\right)} \quad \text { if } \hat{u_0}(\xi)=0 \text { for }|\xi| \notin[\frac{1}{2}, 1].
\end{equation}

Indeed, define the dyadic decomposition of the operator $A$ as
$$
Au_{0j}(x)=\int_{\mathbb{R}^{2}} e^{i x \cdot \xi} \vp\left(2^{-j}|\xi|\right) \hat{u_0}(\xi) \mathrm{d}\xi, \quad j \in \mathbb{Z}.
$$
By \eqref{q-13} and the compact support property of $\hat{u_0}(\xi)$, one has that for $\hat{u_0}(\xi)=0$ when $|\xi| \notin[\frac{1}{2}, 1]$,
\begin{equation}\label{q-14}
\begin{aligned}
\|t^{\f{m-p+1}{p+1}}A u_0\|_{L_{t}^{q} L_{r}^{\infty} L_{\theta}^{2}\left([1,+\infty)\times [0, +\infty)\times [0, 2\pi]\right)}&\le
 C\|\hat{u_0}(\xi)\|_{L^2(\Bbb R^2)}\\
&\le C\||\xi|^{1-\frac{2}{m+2}(\f{1}{q}+\f{m-p+1}{p+1})}\hat{u_0}(\xi)\|_{L^2(\Bbb R^2)}\\
&=C\|u_0\|_{\dot{H}^{1-\frac{2}{m+2}(\f{1}{q}+\f{m-p+1}{p+1})}\left(\mathbb{R}^{2}\right)}.
\end{aligned}
\end{equation}
Since \eqref{q-14} is invariant under scaling, then one easily obtains  that for any $j \in \mathbb{Z}$,
\begin{equation}\label{q-15}
\|t^{\f{m-p+1}{p+1}}A u_{0j}\|_{L_{t}^{q} L_{r}^{\infty} L_{\theta}^{2}\left([1,+\infty)\times [0, +\infty)\times [0, 2\pi]\right)} \leq C\left\|u_{0j}\right\|_{\dot{H}^{1-\frac{2}{m+2}(\f{1}{q}+\f{m-p+1}{p+1})}\left(\mathbb{R}^{2}\right)}.
\end{equation}
Therefore,
$$
\begin{aligned}
\|t^{\f{m-p+1}{p+1}}A u_0\|_{L_{t}^{q} L_{r}^{\infty} L_{\theta}^{2}\left([1,+\infty)\times [0, +\infty)\times [0, 2\pi]\right)} & =\|\sum_{j \in \mathbb{Z}} t^{\f{m-p+1}{p+1}}A u_{0j}\|_{L_{t}^{q} L_{r}^{\infty} L_{\theta}^{2}\left([1,+\infty)\times [0, +\infty)\times [0, 2\pi]\right)} \\
& \leq \sum_{j \in \mathbb{Z}}\|t^{\f{m-p+1}{p+1}}A u_{0j}\|_{L_{t}^{q} L_{r}^{\infty} L_{\theta}^{2}\left([1,+\infty)\times [0, +\infty)\times [0, 2\pi]\right)}  \\
& \leq C\sum_{j \in \mathbb{Z}}\left\|u_{0j}\right\|_{\dot{H}^{1-\frac{2}{m+2}(\f{1}{q}+\f{m-p+1}{p+1})}\left(\mathbb{R}^{2}\right)} \\
& =C\|u_0\|_{\dot{B}_{2,1}^{1-\frac{2}{m+2}(\f{1}{q}+\f{m-p+1}{p+1})}\left(\mathbb{R}^{2}\right)}.
\end{aligned}
$$
Then by interpolation, \eqref{q-11} can be obtained.

Next we establish \eqref{q-13}. Due to the support condition of $\hat{u_0}$, one then has
\begin{equation}\label{q-16}
\|u_0\|_{L^{2}\left(\mathbb{R}^{2}\right)}^{2} \approx \int_{0}^{\infty} \int_{0}^{2 \pi}|\hat{u_0}(\rho \cos \theta, \rho \sin \theta)|^{2} \mathrm{~d} \theta \mathrm{d} \rho,
\end{equation}
where $\xi=(\rho \cos \theta, \rho \sin \theta)$. By the expansion of Fourier series for $\hat{u_0}$,
there exist coefficients $c_{k}(\rho)(k \in \mathbb{Z})$ vanishing for $\rho \notin\left[\frac{1}{2}, 1\right]$ such that
$\hat{u_0}(\xi)=\sum c_{k}(\rho) e^{i k \theta}$.
Due to $a_2(t,\xi)=a_2(t,|\xi|)$ in \eqref{equ:3.17}, then
\begin{equation}\label{q-17}
t^{\f{m-p+1}{p+1}}(A u_0)(t, \xi)=t^{\f{m-p+1}{p+1}}e^{-i \phi_m(t) \rho} a_2(t, \rho) \sum_{k} c_{k}(\rho) e^{i k \theta}.
\end{equation}
Note that
\begin{equation}\label{q-18}
\|u_0\|_{L^{2}\left(\mathbb{R}^{2}\right)}^{2} \approx \sum_{k} \int_{\mathbb{R}}\left|c_{k}(\rho)\right|^{2} \mathrm{d} \rho \approx \sum_{k} \int_{\mathbb{R}}\left|\hat{c}_{k}(s)\right|^{2} \mathrm{d} s, \end{equation}
where $\hat{c}_{k}(s)$ is the Fourier transformation of $c_{k}(\rho)$ with respect to the variable $\rho$.
Let $\beta(\tau) \in C_{0}^{\infty}(\mathbb{R})$ such that
$$
\beta(\tau)= \begin{cases}1, & \frac{1}{2} \leq \tau \leq 1, \\ 0, & \tau \notin\left[\frac{1}{4}, 2\right].\end{cases}
$$
Set
\begin{equation}
\alpha(t, \rho)=\rho \beta(\rho) a_{2}(t, \rho)t^{\f{m-p+1}{p+1}}.
\end{equation}
Analogously treated as in (3.31) of \cite{HWY2}, we have
$$
\begin{aligned}
t^{\f{m-p+1}{p+1}}(A u_0)(t, r(\cos \theta, \sin \theta))
& = \sum_{k}\big( \int_{-\infty}^{\infty} \int_{0}^{2 \pi} e^{-i k \omega} \hat{\alpha}_{t}(\phi_m(t)-s-r \cos \omega) \hat{c}_{k}(s) \mathrm{d} \omega \mathrm{d} s\big) e^{i k \theta},
\end{aligned}
$$
where $\hat{\alpha}_{t}(\xi)$ stands for the Fourier transformation of $\alpha(t, \rho)$ for $\rho$.
Direct computation yields that for any $r \geq 0$,
\begin{equation}\label{q-19}
\begin{aligned}
& \int_{0}^{2 \pi}|t^{\f{m-p+1}{p+1}}(A u_0)(t, r(\cos \theta, \sin \theta))|^{2} \mathrm{d} \theta \\
& = \sum_{k}\big|\int_{-\infty}^{\infty} \int_{0}^{2 \pi} e^{-i k \omega} \hat{\alpha}_{t}(\phi_m(t)-s-r \cos \omega) \hat{c}_{k}(s) \mathrm{d} \omega \mathrm{d} s\big|^{2}.
\end{aligned}
\end{equation}

By Lemma 3.4 in \cite{HWY2}, \eqref{q-19} and H\"{o}lder's inequality, one has that for any small fixed constant $\dl>0$,
\begin{equation*}
\begin{aligned}
&\|t^{\f{m-p+1}{p+1}}A u_0\|^2_{L_{r}^{\infty} L_{\theta}^{2}\left([0, +\infty)\times [0, 2\pi]\right)}\\
&\quad \leq C \sum_{k} \int_{-\infty}^{\infty}\big|t^{\f{m-p+1}{p+1}}(1+\phi_m(t))^{-\frac{m}{2(m+2)}}(1+|\phi_m(t)-s|)^{-\frac{1}{2}+\delta} \hat{c}_{k}(s)\big|^{2} \mathrm{d} s.
\end{aligned}
\end{equation*}
Note that for any $q \geq 2$, we can choose small $\delta>0$ such that
\begin{equation}\label{q-20}
\sigma=q\big(\f{m+1}{2}-\frac{m+2}{2} \delta-\f{m-p+1}{p+1}\big)>1.
\end{equation}
Then it follows from Minkowski's inequality that for $q \geq 2$,
\begin{equation}\label{q-21}
\begin{aligned}
&\|t^{\f{m-p+1}{p+1}}A u_0\|_{L_{t}^{q} L_{r}^{\infty} L_{\theta}^{2}\left([1,+\infty)\times [0, +\infty)\times [0, 2\pi]\right)}\\
&\leq\big(\sum_{k} \int_{-\infty}^{\infty}\big|\big(\int_{1}^{\infty}\big|t^{\f{m-p+1}{p+1}}(1+\phi_m(t))^{-\frac{m}{2(m+2)}}
(1+|\phi_m(t)-s|)^{-\frac{1}{2}+\delta}\big|^{q} \mathrm{d} t\big)^{\frac{1}{q}} \hat{c}_{k}(s)\big|^{2}\mathrm{d} s\big)^{\frac{1}{2}}.
\end{aligned}
\end{equation}
If $s \leq 0$, then by \eqref{q-20}, we arrive at
\begin{equation}\label{q-22}
\int_{1}^{\infty}\big|t^{\f{m-p+1}{p+1}}(1+\phi_m(t))^{-\frac{m}{2(m+2)}}(1+\phi_m(t))^{-\frac{1}{2}+\delta}\big|^{q} \mathrm{d} t \leq C \int_{1}^{\infty}t^{-\sigma} \mathrm{d} t \leq C.
\end{equation}
If $s>0$, set $s=\phi_m(\bar{s})$, then one has that by \eqref{q-20},
\begin{equation}\label{q-23-L}
\begin{aligned}
&
\int_{1}^{\infty}\big|t^{\f{m-p+1}{p+1}}(1+\phi_m(t))^{-\frac{m}{2(m+2)}}(1+|\phi_m(t)-s|)^{-\frac{1}{2}+\delta}\big|^{q} \mathrm{d} t\\
& \leq C \int_{1}^{\infty}(1+|t-\bar{s}|)^{\frac{(m+2) q}{2}\left(-\frac{1}{2}+\delta\right)}(1+t)^{-\frac{mq}{4}+\f{m-p+1}{p+1}} \mathrm{d} t\\
& \leq C,
\end{aligned}
\end{equation}
Substituting \eqref{q-22} and \eqref{q-23-L} into \eqref{q-21} derives that for $\hat{u_0}(\xi)=0$ if $|\xi| \notin[\frac{1}{2}, 1]$,
$$
\|t^{\f{m-p+1}{p+1}}A u_0\|_{L_{t}^{q} L_{r}^{\infty} L_{\theta}^{2}\left([0, +\infty)\times [0, 2\pi]\right)}
\leq C\big(\sum_{k} \int_{-\infty}^{\infty}\left|\hat{c}_{k}(s)\right|^{2} \mathrm{d} s\big)^{\frac{1}{2}} \leq C\|u_0\|_{L^{2}\left(\mathbb{R}^{2}\right)}.
$$
Therefore, \eqref{q-13} is shown and further \eqref{q-11} holds.
\end{proof}

\section{Time-weighted Strichartz estimates for inhomogenous Tricomi equations}\label{E-6}

Note that by a scaling argument as in \cite{LS} and \cite{Rua4},
the equation $\partial_t^2 \t u-t^{m} \Delta\t  u=t^\alpha|\t u|^p$
with initial data $(\t u, \p_t\t u)(1,x)\in (\dot{H}^s, \dot{H}^{s-\f{2}{m+2}})(\Bbb R^2)$ is ill-posed for \(s<1
-\frac{2}{m+2}\cdot\frac{\alpha+2}{p-1}\).
On the other hand, it follows from a concentration argument that
$\t u$ is ill-posed (see Remark 1.3 of \cite{Rua4}) for
$$s<\frac{3}{4}-\frac{m}{(m+2)(2m+6)}
-\frac{3(\alpha+2)}{(2m+6)(p-1)}.$$
In addition, a direct computation shows that for \(p\geq p_{\conf}(2,m,\al)\),
\begin{equation}\label{YJY-1}
1-\frac{2}{m+2}\cdot\frac{\alpha+2}{p-1}
\geq\frac{3}{4}-\frac{m}{(m+2)(2m+6)}
-\frac{3(\alpha+2)}{(2m+6)(p-1)}.
\end{equation}
Especially, for \(p= p_{\conf}(2,m,\al)\), the equality  in \eqref{YJY-1} holds
and \(1-\frac{2}{m+2}\cdot\frac{\alpha+2}{p-1}=\frac{1}{m+2}\).

For the 2-D linear problem
\begin{equation}\label{equ:3.2}
\left\{ \enspace
\begin{aligned}
&\partial_t^2\t v-t^m\triangle\t  v=0,\\
&\t v(0,x)=f(x),\quad \partial_t\t v(0,x)=0,
\end{aligned}
\right.
\end{equation}
such a time-weighted Strichartz inequality is expected
\begin{equation}\label{YC-4}
\| t^\beta \t v\|_{L^q_tL^r_x([1,\infty)\times\R^{2})}\leq C\,\| f\|_{\dot{H}^s(\R^2)},
\end{equation}
where \(\beta>0\), $q\ge 1$ and $r\ge 1$ are some suitable constants related to $s$ ($0<s<1$).
By a scaling argument, one has from \eqref{YC-4} that
\begin{equation}\label{equ:3.4}
  \beta+\frac{1}{q}+\frac{m+2}{2}\cdot\frac{1}{r}
  =\frac{m+2}{2}\left(1-s\right).
\end{equation}
When $r=q$ in \eqref{equ:3.4}, one can obtain
an endpoint $q=\t q_0$ for $s=\frac{1}{m+2}$ (the value of $s=\frac{1}{m+2}$ comes from the equality
in \eqref{YJY-1} with \(p= p_{\conf}(2,m,\al)\)),
\begin{equation}\label{equ:3.5-0}
\t q_0\equiv \frac{2m+6}{m+1-2\beta}>2.
\end{equation}
It is pointed out that the value of $\t q_0$ will be used repeatedly later.

We now investigate the following linear inhomogeneous problem
\begin{equation}\label{a-5}
\begin{cases}\partial_t^2 w-t^{m} \Delta w=F(t, x), &(t,x)\in [1,\infty)\times \R^2,\\
w(1, x)=0,\quad \p_t w(1, x)=0.
\end{cases}
\end{equation}
\begin{lemma}\label{lem3.4-L}
Let $w$ solve \eqref{a-5}. Then
\begin{equation}\label{equ:3.34}
\|t^\beta w\|_{L^q([1,\infty)\times\R^{2})}\leq C\,\bigl\|t^{-\beta}|D_x|^{\gamma-\frac{1}{m+2}}
F\bigr\|_{L^{\t p_0}([1,\infty)\times\R^{2})},
\end{equation}
where $\gamma=1-\frac{2\beta}{m+2} -\frac{2m+6}{q(m+2)}$, \(0<\beta\leq\frac{m}{4} \), $\t q_0\leq
q<\infty$, $\t p_0=\frac{2m+6}{m+2\beta+5}$ with
$\f{1}{\t p_0}+\f{1}{\t q_0}=1$, and the generic constant $C>0$ only depends on $m$ and $q$.
\end{lemma}

\begin{proof}
It follows from \eqref{SW-2} that
\begin{equation}\label{equ:3.40-L}
t^\beta w(t,x)=(AF)(t,x)\equiv
\int_1^t\int_{\R^2}e^{i\left(x\cdot\xi-(\phi_m(t)-\phi_m(\tau))|\xi|\right)}
a(t,\tau,\xi)\hat{F}(\tau,\xi)\,\mathrm{d}\xi \mathrm{d}\tau,
\end{equation}
where $a(t,\tau,\xi)$ satisfies
\begin{equation}\label{equ:3.40--L}
\bigl| \partial_\xi^\kappa a(t,\tau,\xi)\bigr|\leq
Ct^\beta \left(1+\phi_m(t)|\xi|\right)^{-\f{m}{2(m+2)}}
\left(1+\phi_m(\tau)|\xi|\right)^{-\f{m}{2(m+2)}}|\xi|^{-\frac{2}{m+2}-|\kappa|}
\leq|\xi|^{-\frac{4\beta+2}{m+2}-|\kappa|}\tau^{-\beta}.
\end{equation}

To treat $(AF)(t,x)$, the following more general operator is introduced as in (3.35) of \cite{HWY1}
\begin{equation}\label{equ:3.37}
(A^\nu F)(t,x)=\int_1^t\int_{\R^2}e^{i\left(x\cdot\xi-(\phi_m(t)-\phi_m(\tau))|\xi|\right)}
 a(t,\tau,\xi)\hat{F}(\tau,\xi)\frac{d\xi}{|\xi|^\nu}\,\mathrm{d}\tau,
\end{equation}
where the parameter $\nu\in(0, 1)$.

Set
\begin{equation}\label{equ:3.38}
  A^\nu_j F(t,x)=\int_1^t\int_{\R^2}e^{i\left(x\cdot\xi-(\phi_m(t)-\phi_m(\tau))|\xi|\right)}
  \vp(\frac{|\xi|}{2^j})a(t,\tau,\xi)
\hat{F}(\tau,\xi)\frac{d\xi}{|\xi|^\nu} \, \mathrm{d}\tau,
\end{equation}
where the function $\vp$ is given in \eqref{chi}.
We next show the following inequality for
$\gamma=1-\frac{2\beta}{m+2}-\frac{2m+6}{q(m+2)}$ and $\t q_0\le q<\infty$,
\begin{equation*}
\|t^\beta w\|_{L^q([1,\infty)\times \R^{2})}\leq C\bigl\|t^{-\beta}|D_x|^{\gamma-\frac{1}{m+2}}
F\bigr\|_{L^{\t p_0}([1,\infty)\times \R^{2})}.
\end{equation*}
That is,
\begin{equation*}
\|t^\beta|D_x|^{-\gamma+\frac{1}{m+2}}w\|_{L^q([1,\infty)\times \R^{2})}\leq
C\| t^{-\beta}F\|_{L^{\t p_0}([1,\infty)\times \R^{2})}.
\end{equation*}
By the definition of $A^\nu$ in \eqref{equ:3.37} with
$\nu=\gamma-\frac{1}{m+2}$, once
\begin{equation}\label{equ:3.39}
\| A^\nu F\|_{L^q([1,\infty)\times \R^{2})}\le C\| t^{-\beta}F\|_{L^{\t p_0}([1,\infty)\times \R^{2})}
\end{equation}
is established, then \eqref{equ:3.34} holds true.

To derive \eqref{equ:3.39}, in terms of Lemma \ref{lem3-E}, it requires to show
\begin{equation}\label{equ:3.40-M}
\| A^\nu_j F\|_{L^q([1,\infty)\times \R^{2})}\le C\| t^{-\beta}F\|_{L^{\t p_0}([1,\infty)\times \R^{2})}.
\end{equation}
By applying the interpolation argument, it only needs to prove \eqref{equ:3.40-M} for
the two endpoint cases of $q=\t q_0$ and $q=\infty$. Meanwhile, the corresponding indices $\nu$ are denoted by $\nu_0$ and $\nu_1$,
respectively.
It follows from direct computation
that $\nu_0=1-\frac{2\beta}{m+2}-\frac{2m+6}{\t q_0(m+2)}
-\frac{1}{m+2}=0$ and $\nu_1=1-\frac{2\beta+1}{m+2}$.
For treating $A^{\nu_0}_j=A^{0}_j$, set
\[
T_j^0F(t,\tau,x)=\int_{\R^2}
e^{i\left(x\cdot\xi-(\phi_m(t)-\phi_m(\tau))|\xi|\right)}
\vp(\frac{|\xi|}{2^j})a(t,\tau,\xi)\hat{F}(\tau,\xi)\, \mathrm{d}\xi.
\]
Next we show
\begin{equation}\label{d-2}
\|T_j^0F(t,\tau,x)\|_{L^{\t q_{0}}\left( \mathbb{R}^{2}\right)}
\leq C|t-\tau|^{-\frac{m+1-2\beta}{m+3}}\| \tau^{-\beta }F\|_{L^{\t p_0}(\R^2)}.
\end{equation}
For \( 0<\beta\leq\frac{m}{m+6}\), without loss of generality, $t\geq \tau$ is assumed. Note that
\begin{equation*}
\begin{aligned}
&\big|t^{\,\frac{(\beta+1)m}{2m+6}
-\beta}\vp(\frac{|\xi|}{2^j})a(t,\tau,\xi)\big| \\
&\leq C t^{\frac{(\beta+1)m}{2m+6}-\beta}t^\beta
\left(1+\phi_m(t)|\xi|\right)^{-\f{m}{2(m+2)}}
\left(1+\phi_m(\tau)|\xi|\right)^{-\f{m}{2(m+2)}}|\xi|^{-\f{2}{m+2}}  \\
&\leq C |\xi|^{-\f{2}{m+2}-\frac{m(\beta+1)}{(m+2)((m+3)}-\frac{2\beta}{m+2}}\tau^{-\beta}.
\end{aligned}
\end{equation*}
Denote
\begin{equation*}
b(t,\tau,\xi)=2^{j(\f{2}{m+2}+\frac{m(\beta+1)}{(m+2)(m+3)}+\frac{2\beta}{m+2})}t^{\,\frac{(\beta+1)m}{2m+6}-\beta}
\vp(\frac{|\xi|}{2^j})a(t,\tau,\xi).
\end{equation*}
Then
$$
|b(t,\tau,\xi)|\leq C \tau^{-\beta}.
$$
Define
\[
T_{j1}^0F(t,\tau,x)=\int_{\R^2}\int_{\R^2}
e^{i\left((x-y)\cdot\xi-(\phi_m(t)-\phi_m(\tau))|\xi|\right)}t^{-\frac{(\beta+1)m}{2m+6}
+\beta}
b(t,\tau,\xi)F(\tau,y)\, \mathrm{d}\xi \mathrm{d}y.
\]
By the assumptions of $t\geq \tau$ and $0<\beta\leq\frac{m}{m+6}$, one has
$$
\| T_{j1}^0F(t,\tau,x)\|_{L^2(\R^2)}\leq C\left|t-\tau\right|^{\beta-\frac{(\beta+1)m}{2m+6}}
\|\tau^{-\beta}F\|_{L^2(\R^2)}
$$
and
$$
\begin{aligned}
\| T_{j1}^0F(t,\tau,x)\|_{L^\infty(\R^2)}&\leq
C 2^{\frac{3}{2}j}t^{\, \beta-\frac{(\beta+1)m}{2m+6}}
\left|\phi_m(t)-\phi_m(\tau)\right|^{-\frac{1}{2}}\| \tau^{-\beta }F\|_{L^1(\R^2)} \\
&\leq C2^{\frac{3}{2}j}\left|t-\tau\right|^{\beta-\frac{(\beta+1)m}{2m+6}-\frac{m+2}{4}}
\| \tau^{-\beta }F\|_{L^1(\R^2)}.
\end{aligned}
$$
By interpolation, we have
$$
  \| T_{j1}^0F\|_{L^{\t q_0}(\R^2)}\leq C 2^{\frac{3(\beta+1)}{m+3}j}
  \left|t-\tau\right|^{-\frac{m+1-2\beta}{m+3}}\| \tau^{-\beta }F\|_{L^{\t p_0}(\R^2)}.
$$
Thus
\begin{equation}\label{d-4}
\begin{aligned} \|T_{j}^0F\|_{L^{\t q_0}(\R^2)}
&\leq C 2^{-(\f{2}{m+2}+\frac{m(\beta+1)}{(m+2)(m+3)}+\frac{2\beta}{m+2}-\frac{3(\beta+1)}{m+3})j}
  \left|t-\tau\right|^{-\frac{m+1-2\beta}{m+3}}\| \tau^{-\beta }F\|_{L^{\t p_0}(\R^2)}\\
&= C\left|t-\tau\right|^{-\frac{m+1-2\beta}{m+3}}\| \tau^{-\beta }F\|_{L^{\t p_0}(\R^2)}.
\end{aligned}
\end{equation}
We consider  the remaining case of \(\frac{m}{m+6}<\beta\leq\frac{m}{4}\). Let
$b(t,\tau,\xi)=|\xi|^{\frac{4\beta+2}{m+2}}\vp(\frac{|\xi|}{2^j})a(t,\tau,\xi)$ and
\[
T_{j2}^0F(t,\tau,x)=\int_{\R^2}\int_{\R^2}
e^{i\left((x-y)\cdot\xi-(\phi_m(t)-\phi_m(\tau))|\xi|\right)}
b(t,\tau,\xi)F(\tau,y)\, \mathrm{d}\xi \mathrm{d}y.
\]
Then
\begin{equation}\label{d-3}
\| T_{j2}^0F(t,\tau,x)\|_{L^2(\R^2)}\leq C
\|\tau^{-\beta}F\|_{L^2(\R^2)},
\end{equation}
while for the $L^\infty$ estimate, it follows from the  stationary phase method and \(\beta>\frac{m}{m+6}\) that
$$
\begin{aligned}
\| T_{j2}^0F(t,\tau,x)\|_{L^\infty(\R^2)}
&\leq C2^{2j}
(1+2^j \left|\phi_m(t)-\phi_m(\tau)\right|)^{-\frac{m+1-2\beta}{(\beta+1)(m+2) }}\| \tau^{-\beta}F\|_{L^1(\R^2)} \\
&\leq C 2^{2j-\frac{m+1-2\beta}{(\beta+1)(m+2)}j}
\left|t-\tau\right|^{ -\frac{m+1-2\beta}{2(\beta+1) }}\| \tau^{-\beta}F\|_{L^1(\R^2)}.
\end{aligned}
$$
This, together with \eqref{d-3}, yields
\begin{equation}
  \| T_{j2}^0F(t,\tau,x)\|_{L^{\t q_0}(\R^2)}\leq C 2^{\frac{4\beta+2 }{m+2}j}
  \left|t-\tau\right|^{-\frac{m+1-2\beta}{m+3}}\| \tau^{-\beta}F\|_{L^{\t p_0}(\R^2)}.
\end{equation}
Then
\begin{equation}\label{d-5}
 \|T_{j}^0F\|_{L^{\t q_0}(\R^2)}=2^{-\frac{4\beta+2 }{m+2}j}\| T_{j2}^0F\|_{L^{\t q_0}(\R^2)}\leq C\left|t-\tau\right|^{-\frac{m+1-2\beta}{m+3}}\| \tau^{-\beta }F\|_{L^{\t p_0}(\R^2)}.
\end{equation}
Thus \eqref{d-2} follows from \eqref{d-4} and \eqref{d-5} immediately.

Note that $\ds A^{0}_j F(t,x)=\int_1^tT_j^0F(t,\tau,x)\,\mathrm{d}\tau$. In addition,
by \eqref{d-2} and the Hardy-Littlewood-Sobolev inequality, we
may conclude
\[
\|\int_{1}^{\infty}\|
T_j^0F(t,\tau,x)\|_{L^{\t q_0}_x}\,\mathrm{d}\tau\|_{L^{\t q_0}_t[1,\infty)}\leq C\|
t^{-\beta}F\|_{L^{\t p_0}([1,\infty)\times\R^2)}.
\]
Let
\begin{equation*}
K(t,\tau)=
\begin{cases}
|t-\tau|^{-\frac{m+1-2\beta}{m+3}}, &\tau\geq 1, \\
0, &\tau<1
\end{cases}
\end{equation*}
and set $q=\t q_0$ in \cite[Lemma~3.8]{LS},
then \eqref{equ:3.40-M} is obtained for $q=\t q_0$.

Next we establish \eqref{equ:3.40-M} for $q=\infty$.  In this case, the
kernel of $A^{\nu_1}_j$ admits such a form
\[
K^{\nu_1}_j(t,x;\tau,y)=\int_{\R^2}\vp(\frac{|\xi|}{2^j})
e^{i\left((x-y)\cdot\xi-(\phi_m(t)-\phi_m(\tau))|\xi|\right)}a(t,\tau,\xi)\tau^{\beta} \,\frac{\mathrm{d}\xi}{|\xi|^{\nu_1}}.
\]
We now assert
\begin{equation}\label{equ:3.42}
\sup\limits_{t,x}\int_{[1,\infty)\times\R^2}|K^{\nu_1}_j(t,x;\tau,y)|^{\t q_0}\,\mathrm{d}\tau \mathrm{d}y<\infty.
\end{equation}

Indeed, by (3.29) of \cite{LS}, one can obtain
\begin{multline}\label{equ:3.43}
\big|K^{\nu_1}_j(t,x;\tau,y)\big| \leq \,
C_{N,\nu_1}\lambda^{\frac{3}{2}}
t^\beta \left(1+\phi_m(t)|\xi|\right)^{-\f{m}{2(m+2)}}
\tau^{\beta}\left(1+\phi_m(\tau)|\xi|\right)^{-\f{m}{2(m+2)}} \\
\times\left(|\phi_m(t)-\phi_m(\tau)|+\lambda^{-1}\right)^{-\frac{1}{2}}
\left(1+\lambda\big||x-y|-|\phi_m(t)-\phi_m(\tau)|\big|\right)^{-N}\lambda^{-\nu_1}|\xi|^{-\frac{2}{m+2}} \\
\leq C_{N,\nu_1}\lambda^{\frac{3}{2}-{\bar\nu}_1}
\left(|\phi_m(t)-\phi_m(\tau)|+\lambda^{-1}\right)^{-\frac{1}{2}}
\left(1+\lambda\big||x-y|-|\phi_m(t)-\phi_m(\tau)|\big|\right)^{-N},
\end{multline}
where $\lambda=2^j$, $N\in \R^+$, and
${\bar\nu}_1=\frac{2}{m+2}+\nu_1+\frac{4\beta}{m+2}
=1+\frac{1+2\beta}{m+2}.$

Without loss of generality, $x=0$ is assumed in \eqref{equ:3.42}. It follows from the choice of $N \t q_0=3$ and direct
computation that
\begin{multline*}
\int_{[1,\infty)\times\R^2}|K^{\nu_1}_j(t,0;\tau,y)|^{\t q_0}\,\mathrm{d}\tau \mathrm{d}y \\
\begin{aligned}
&\leq\int_{1}^{\infty}\int_{\R^2}\lambda^{(\frac{3}{2}-{\bar\nu}_1)\cdot \t q_0}
  \left(|\phi_m(t)-\phi_m(\tau)|+\lambda^{-1}\right)^{-\frac{1}{2}\cdot \t q_0}
  \left(1+\lambda\bigl||y|-|\phi_m(t)-\phi_m(\tau)|\bigr|\right)^{-N\t q_0}\,\mathrm{d}\tau\mathrm{d}y \\
&\leq C\int_{1}^{\infty}\lambda^{\frac{m-4\beta}{2(m+2)}\cdot \t q_0}
  \left(|\phi_m(t)-\phi_m(\tau)|+\lambda^{-1}\right)^{-\frac{\t q_0}{2}}\lambda^{-1}
  \left(|\phi_m(t)-\phi_m(\tau)|+\lambda^{-1}\right)\,\mathrm{d}\tau \\
&\leq C \int_{1}^{\infty}\lambda^{\frac{m-4\beta}{2(m+2)}\cdot \t q_0-1}
(|t-\tau|+\lambda^{-\frac{2}{m+2}})^{\frac{m+2}{2}(1-\frac{\t q_0}{2})}\,\mathrm{d}\tau \\
&\leq C\lambda^{\frac{m-4\beta}{2(m+2)}\cdot \t q_0-1-\frac{2}{m+2} -(1-\frac{\t q_0}{2})} \\
&= C.
\end{aligned}
\end{multline*}
 Therefore, the assertion \eqref{equ:3.42} is proved. By \eqref{equ:3.42} and
H\"{o}lder's inequality, it is easy to know that \eqref{equ:3.40-M} holds for $q=\infty$.
Therefore, by interpolation, \eqref{equ:3.40-M} and consequently \eqref{equ:3.34} are shown.
\end{proof}

Based on Lemma \ref{lem3.4-L}, it follows from the expression of $\t q_0$  with $\beta=\f{\al}{\t q_0}\leq \f{m}{4}$ in \eqref{equ:3.5-0}
that $\t q_0=\frac{2(m+3+\alpha)}{m+1}$ and $\alpha \leq m\cdot\f{m+3}{m+2}$ hold true.
Denote $\t q_0$ by $q_0$, one naturally has
\begin{lemma}\label{lem4}
For $q_{0}=\frac{2(m+3+\alpha)}{m+1}$ with  $0<\alpha \leq m\cdot\f{m+3}{m+2}$, if $F(t, x) \equiv 0$ when $|x|>\phi_{m}(t)-1$, then the solution $w$ of problem \eqref{a-5} satisfies
\begin{equation}\label{a-11}
\|t^{\frac{\alpha}{q_{0}}} w\|_{L^{q_{0}}\left(\left[1, \infty\right) \times \mathbb{R}^{2}\right)} \leq C\|t^{-\frac{\alpha}{q_{0}}} F\|_{L^{p_0}\left(\left[1, \infty\right) \times \mathbb{R}^{2}\right)},
\end{equation}
where $p_0=\frac{q_{0}}{q_{0}-1}$.
\end{lemma}
\begin{remark}\label{JY-8-5}
{\it It is pointed out that the proof of Theorem \ref{TH-1} (i) for $\mu\in(0, 1)$ will require the global existence
 result of \eqref{YH-4} with $\al=m< m\cdot\f{m+3}{m+2}$. Hence Lemma \ref{lem4} is  crucial for establishing \eqref{equ:3.3-L}.}
\end{remark}

In Lemmas ~\ref{lem3.4-L}-\ref{lem4}, we have established the related spacetime-weighted Strichartz inequalities for
the case of $0<\alpha\leq m\cdot\f{m+3}{m+2}$, next we  investigate the case of $-1<\alpha<0$.
In this situation, one cannot use some crucial estimates such as \eqref{equ:3.40--L} in Lemma~\ref{lem3.4-L}.
To overcome this essential difficulty, we will derive some weak Strichartz estimate in finite time
instead of the ones in infinite time, which is enough for our applications later.

\begin{lemma}\label{lem5}
For $q_{0}=\frac{2(m+3+\alpha)}{m+1}$ with  $-1<\alpha<0$, if $F(t, x) \equiv 0$ when $|x|>\phi_{m}(t)-1$, then for any fixed large $\bar{T}>0$, the solution $w$ of problem \eqref{a-5} satisfies
\begin{equation}
\|t^{\frac{\alpha}{q_{0}}} w\|_{L^{q_{0}}\left([1, \bar{T}] \times \mathbb{R}^{2}\right)}
\leq C\|t^{-\frac{\alpha}{q_{0}}} F\|_{L^{p_0}\left([1, \bar{T}] \times \mathbb{R}^{2}\right)},
\end{equation}
where $p_0=\frac{q_{0}}{q_{0}-1}$ and the constant $C>0$ depends on $q_{0}$ and $\bar{T}$.
\end{lemma}
\begin{proof}
It follows from \eqref{SW-2} that
\begin{equation}\label{aa-28}
w(t, x)=AF(t,x)=\int_{\mathbb{R}^{2}\times [1, t)}\int_{\mathbb{R}^{2}} e^{i\left((x-y) \cdot \xi-\left(\phi_{m}(t)-\phi_{m}(s)\right)|\xi|\right)} a(t, s, \xi) F(s, y) \mathrm{d} \xi \mathrm{d} s \mathrm{d} y,
\end{equation}
where
\begin{equation}\label{a-28}
\left|\partial_{\xi}^{\kappa} a(t, s, \xi)\right| \leq\left(1+\phi_{m}(t)|\xi|\right)^{-\frac{m}{2(m+2)}}\left(1+\phi_{m}(s)|\xi|\right)^{-\frac{m}{2(m+2)}}|\xi|^{-\frac{2}{m+2}-|\kappa|} \leq|\xi|^{-\frac{2}{m+2}-|\kappa|} .
\end{equation}
Set $a_{\lambda}(t, s, \xi)=\vp(|\xi| / \lambda) a(t, s, \xi)$ for $\lambda>0$ and $\vp$ is defined in \eqref{chi}. Define
$$
A_{\lambda} F(t, x)=\int_{\left[1, t\right) \times \mathbb{R}^{2}} \int_{\mathbb{R}^{2}} e^{i\left((x-y) \cdot \xi-\left(\phi_{m}(t)-\phi_{m}(s)\right)|\xi|\right)}a_{\lambda}(t, s, \xi) F(s, y) \mathrm{d} \xi \mathrm{d} y \mathrm{d} s,
$$
where $a_{\lambda}(t, s, \xi)$ satisfies
$$
\left|\partial_{\xi}^{\kappa} a_{\lambda}(t, s, \xi)\right|\leq|\lambda|^{-\frac{2}{m+2}-|\kappa|}.
$$
Let
$$
T_{t, s} F(x)=\int_{\mathbb{R}^{2}} \int_{\mathbb{R}^{2}} e^{i\left((x-y) \cdot \xi-\left(\phi_{m}(t)-\phi_{m}(s)\right)|\xi|\right)} a_{\lambda}(t, s, \xi) F(s,y) \mathrm{d} \xi \mathrm{d}y.
$$
Then
\begin{equation}\label{a-29}
\|t^{\frac{\alpha}{q_{0}}} T_{t, s} F(\cdot)\|_{L^{2}\left(\mathbb{R}^{2}\right)} \leq C t^{\frac{\alpha}{q_{0}}} s^{\frac{\alpha}{q_{0}}} \lambda^{-\frac{2}{m+2}}\|s^{-\frac{\alpha}{q_{0}}} F(s,\cdot)\|_{L^{2}\left(\mathbb{R}^{2}\right)} \leq C \lambda^{-\frac{2}{m+2}}\|s^{-\frac{\alpha}{q_{0}}} F(s,\cdot)\|_{L^{2}\left(\mathbb{R}^{2}\right)}.
\end{equation}
In addition, when $\lambda<1$, one has
\begin{equation}\label{a-30}
\|t^{\frac{\alpha}{q_{0}}} T_{t, s} F(\cdot)\|_{L^{\infty}\left(\mathbb{R}^{2}\right)} \leq C \lambda^{2-\frac{2}{m+2}}t^{\frac{\alpha}{q_{0}}} s^{\frac{\alpha}{q_{0}}}\|s^{-\frac{\alpha}{q_{0}}} F(s,\cdot)t\|_{L^{1}\left(\mathbb{R}^{2}\right)} \leq C \lambda^{2-\frac{2}{m+2}}\|s^{-\frac{\alpha}{q_{0}}} F(s,\cdot)\|_{L^{1}\left(\mathbb{R}^{2}\right)}.
\end{equation}
Interpolating \eqref{a-29} with \eqref{a-30} yields that for $1 \leq s \leq t \leq \bar{T}$ and $\lambda<1$,
\begin{equation}\label{a-31}
\begin{aligned}
\|t^{\frac{\alpha}{q_{0}}} T_{t, s} F(\cdot)\|_{L^{q_{0}}\left(\mathbb{R}^{2}\right)} &\leq C \lambda^{\frac{4 (\al+q_0)}{(m+3)q_0}-\frac{2}{m+2}}\|s^{-\frac{\alpha}{q_{0}}} F(s,\cdot)\|_{L^{p_{0}}\left(\mathbb{R}^{2}\right)}\\
&  \leq C(\bar{T})|t-s|^{-\frac{(m+1)q_0-2 \al}{(m+3)q_0}}\|s^{-\frac{\alpha}{q_{0}}} F(s,\cdot)\|_{L^{p_{0}}\left(\mathbb{R}^{2}\right)},
\end{aligned}
\end{equation}
where we have used the facts of $\frac{4 (\al+q_0)}{(m+3)q_0}-\frac{2}{m+2}>0$ and $-\frac{(m+1)q_0-2 \al}{(m+3)q_0}<0$ with $-1<\al<0$ in the last inequality. On the other hand, it follows from stationary phase method that for $\lambda \geq 1$,
\begin{equation}\label{a-33}
\begin{aligned}
\|t^{\frac{\alpha}{q_{0}}} T_{t, s} F(\cdot)\|_{L^{\infty}\left(\mathbb{R}^{2}\right)} & \leq C \lambda^{2}\left(1+\lambda\left|\phi_{m}(t)-\phi_{m}(s)\right|\right)^{-\frac{1}{2}} t^{\frac{\alpha}{q_{0}}-\frac{m}{4}} s^{\frac{\alpha}{q_{0}}-\frac{m}{4}} \lambda^{-1}\|s^{-\frac{\alpha}{q_{0}}} F(s,\cdot)\|_{L^{1}\left(\mathbb{R}^{2}\right)} \\
& \leq C \lambda^{\frac{1}{2}}|t-s|^{-\f{m+2}{4}}\|s^{-\frac{\alpha}{q_{0}}} F(s,\cdot)\|_{L^{1}\left(\mathbb{R}^{2}\right)},
\end{aligned}
\end{equation}
while for the $L^{2}$ estimate, it holds that
\begin{equation}\label{a-34}
\|t^{\frac{\alpha}{q_{0}}} T_{t, s} F(\cdot)\|_{L^{2}\left(\mathbb{R}^{2}\right)}
\leq C \lambda^{-1}\|s^{-\frac{\alpha}{q_{0}}} F(s,\cdot)\|_{L^{2}\left(\mathbb{R}^{2}\right)}.
\end{equation}
By interpolation, we can obtain that for $\lambda \geq 1$ and $|t-s| \leq 2 \bar{T}$,
\begin{equation}\label{a-35}
\begin{aligned}
\|t^{\frac{\alpha}{q_{0}}} T_{t, s} F(\cdot)\|_{L^{q_{0}}\left(\mathbb{R}^{2}\right)} &\leq C \lambda^{\frac{3(\al+q_0)}{(m+3)q_0}-1}|t-s|^{-\frac{(\al+q_0)(m+2)}{2(m+3)q_0}}\|s^{-\frac{\alpha}{q_{0}}} F(s,\cdot)\|_{L^{p_{0}}\left(\mathbb{R}^{2}\right)}\\
& \leq C(\bar{T})|t-s|^{-\frac{(m+1)q_0-2 \al}{(m+3)q_0}}\|s^{-\frac{\alpha}{q_{0}}} F(s,\cdot)\|_{L^{p_{0}}\left(\mathbb{R}^{2}\right)},
\end{aligned}
\end{equation}
where  the facts of $\frac{3(\al+q_0)}{(m+3)q_0}-1<0$ and $-\frac{(\al+q_0)(m+2)}{2(m+3)q_0}>-\frac{(m+1)q_0-2 \al}{(m+3)q_0}$
have been used in the last inequality.
Due to the remaining part of the proof procedure is completely analogous to that of Lemma~\ref{lem3.4-L},  we omit
the details here.
\end{proof}

\section{Spacetime-weighted Strichartz estimate for inhomogenous Tricomi equations}\label{E-9}

In this section, we establish the spacetime-weighted Strichartz estimate for the 2-D inhomogenous problem
\begin{equation}\label{a-5-L}
\begin{cases}\partial_t^2 w-t^{m} \Delta w=F(t, x), \quad (t,x)\in [1,\infty)\times \R^2,\\
w(1, x)=0,\quad \p_t w(1, x)=0.
\end{cases}
\end{equation}
It is pointed out that the proof procedure is somewhat analogous to that in Sections 3-5 of \cite{HWY4} for
the 3-D generalized Tricomi equation $\partial_t^2 \t w-t^{m} \Delta \t w=\t F(t, x)$ with $m\ge 1$, but
we  will give the details for reader's convenience due to the 2-D case of $\partial_t^2 w-t^{m} \Delta w
=F(t, x)$ with $m>0$ and the appearance of the related time weight. On the other hand, we need more careful computations on
the 2-D solution $w$ than the 3-D solution $\t w$ since the solution $w$ admits lower time-decay rate (see  Line 8 of Page 54 and
the argument of (4-42) in \cite{HWY4} for the case of $n\ge 3$ and $m\ge 1$).

Our main result is
\begin{theorem}\label{th2-2}
For problem \eqref{a-5-L}, if $F(t, x) \equiv 0$ when $|x|>\phi_{m}(t)-1$, then there exist some constants $\gamma_{1}$ and $\gamma_{2}$ satisfying $0<\gamma_{1}<\frac{m+1}{m+2}-\frac{m+4+2\al}{(m+2) q}$ and $\gamma_{2}>\frac{1}{q}$, such that
\begin{equation}\label{a-6}
\big\|\left(\phi_{m}(t)^{2}-|x|^{2}\right)^{\gamma_{1}} t^{\frac{\alpha}{q}} w\big\|_{L^{q}\left(\left[1, \infty\right) \times \mathbb{R}^{2}\right)} \leq C\big\|\left(\phi_{m}(t)^{2}-|x|^{2}\right)^{\gamma_{2}} t^{-\frac{\alpha}{q}} F\big\|_{L^{\frac{q}{q-1}}\left(\left[1, \infty\right) \times \mathbb{R}^{2}\right)},
\end{equation}
where $2 \leq q \leq \frac{2(m+3+\alpha)}{m+1}$, $-1<\al\leq m$ and $C>0$ is a constant depending on $m, \al,  q, \gamma_{1}$ and $\gamma_{2}$.
\end{theorem}
By the Stein analytic interpolation theorem (see \cite{S}), in order to prove \eqref{a-6}, we only need to establish the following inequalities for the two endpoint cases of $q=q_{0}=\frac{2(m+3+\alpha)}{m+1}$ and $q=2$ :
\begin{equation}\label{a-7}
\|\left(\phi_{m}(t)^{2}-|x|^{2}\right)^{\gamma_{1}} t^{\frac{\alpha}{q_{0}}} w\|_{L^{q_{0}}\left(\left[1, \infty\right) \times \mathbb{R}^{2}\right)} \leq C\|\left(\phi_{m}(t)^{2}-|x|^{2}\right)^{\gamma_{2}} t^{-\frac{\alpha}{q_{0}}} F\|_{L^{\frac{q_{0}}{q_{0}-1}}\left(\left[1, \infty\right) \times \mathbb{R}^{2}\right)},
\end{equation}
where $\gamma_{1}<\frac{1}{q_{0}}<\gamma_{2}$, $-1<\al\leq m$; and
\begin{equation}\label{a-8}
\|\left(\phi_{m}(t)^{2}-|x|^{2}\right)^{\gamma_{1}} t^{\frac{\alpha}{2}} w\|_{L^{2}\left(\left[1, \infty\right) \times \mathbb{R}^{2}\right)} \leq C\|\left(\phi_{m}(t)^{2}-|x|^{2}\right)^{\gamma_{2}} t^{-\frac{\alpha}{2}} F\|_{L^{2}\left(\left[1, \infty\right) \times \mathbb{R}^{2}\right)},
\end{equation}
where $\gamma_{1}<-\frac{1}{2}+\frac{m-\alpha}{m+2}$, $\gamma_{2}>\frac{1}{2}$ and $-1<\al\leq m$. \eqref{a-7} and \eqref{a-8}
will be established in Subsections \ref{Sub-4.1} and \ref{Sub-4.2} below, respectively.

\subsection{Proof of Theorem~\ref{th2-2} for the endpoint case $q=q_0$}\label{Sub-4.1}

Based on Lemma~\ref{lem4} and Lemma~\ref{lem5}, we now prove \eqref{a-7}. At this time, \eqref{a-7} can be rewritten as
\begin{equation}\label{a-37}
\|\left(\phi_{m}^{2}(t)-|x|^{2}\right)^{\frac{1}{q_{0}}-\nu} t^{\frac{\alpha}{q_{0}}} w\|_{L^{q_{0}}\left(\left[1, \infty\right) \times \mathbb{R}^{2}\right)} \leq C\|\left(\phi_{m}^{2}(t)-|x|^{2}\right)^{\frac{1}{q_{0}}+\nu} t^{-\frac{\alpha}{q_{0}}} F\|_{L^{\frac{q_{0}}{q_{0}-1}}\left(\left[1, \infty\right) \times \mathbb{R}^{2}\right)},
\end{equation}
where $\nu>0$. Notice that for any fixed $\bar{T} \gg 1$,  $\phi_{m}^{2}(t)-|x|^{2}$ and $t^{\frac{\alpha}{q_{0}}}$ have upper and lower bounds for $1 \leq t \leq \bar{T}$. This, together with Lemma~\ref{lem4} and Lemma~\ref{lem5}, yields
\begin{equation}\label{a-38}
\begin{aligned}
&\|\left(\phi_{m}^{2}(t)-|x|^{2}\right)^{\frac{1}{q_{0}}-\nu} t^{\frac{\alpha}{q_{0}}} w\|_{L^{q_{0}}\left([1, \bar{T}] \times \mathbb{R}^{2}\right)} \\
& \leq C|\ln \bar{T}|^{\frac{1}{q_{0}}} \phi_{m}(\bar{T})^{-\nu}\|\left(\phi_{m}^{2}(t)-|x|^{2}\right)^{\frac{1}{q_{0}}+\nu} t^{-\frac{\alpha}{q_{0}}} F\|_{L^{\frac{q_{0}}{q_{0}-1}}\left([1, \bar{T}] \times \mathbb{R}^{2}\right)}.
\end{aligned}
\end{equation}
Thus in order to prove \eqref{a-37}, it suffices to prove that for $T \geq \bar{T}$,
\begin{equation}\label{a-39}
\begin{aligned}
&\|\left(\phi_{m}^{2}(t)-|x|^{2}\right)^{\frac{1}{q_{0}}-\nu} t^{\frac{\alpha}{q_{0}}} w\|_{L^{q_{0}}\left([T, 2 T] \times \mathbb{R}^{2}\right)} \\
& \leq C|\ln T|^{\frac{1}{q_{0}}} \phi_{m}(T)^{-\nu}\|\left(\phi_{m}^{2}(t)-|x|^{2}\right)^{\frac{1}{q_{0}}+\nu} t^{-\frac{\alpha}{q_{0}}} F\|_{L^{\frac{q_{0}}{q_{0}-1}}\left([1,+\infty) \times \mathbb{R}^{2}\right)},
\end{aligned}
\end{equation}
where $\bar{T}>1$ is a fixed large constant. Next we will focus on the proof of \eqref{a-39}.

Due to supp $F \subseteq\left\{(t, x):|x|^{2} \leq \phi_{m}^{2}(t)-1\right\}$, then we decompose $F=F^{0}+F^{1}$, where
\begin{equation}\label{a-40}
F^{0}=\begin{cases}
F, \quad t \geq \frac{T}{2 \cdot 10^{\frac{2}{m+2}}},\\
0, \quad t<\frac{T}{2 \cdot 10^{\frac{2}{m+2}}}.
\end{cases}
\end{equation}
And $w=w^{0}+w^{1}$, where $w^{j}(j=0,1)$ solves
$$
\left\{\begin{array}{l}
\partial_{t}^{2} w^{j}-t^{m} \Delta w^{j}=F^{j},\\
w^{j}\left(1, x\right)=0, \quad \partial_{t} w^{j}\left(1, x\right)=0.
\end{array}\right.
$$
To prove \eqref{a-39}, by an  analogous simplification procedure in Section 4A of \cite{HWY4},  we only need to show that for $j=0,1$,
\begin{equation}\label{a-44}
\begin{aligned}
\left(\frac{\phi_{m}(T)}{\phi_{m}\left(\bar{T_{0}}\right)}\right)^{\frac{1}{q_{0}}-\frac{\nu}{2}} \delta^{\frac{1}{q_{0}}+\frac{\nu}{2}}&\|t^{\frac{\alpha}{q_{0}}} w^{j}\|_{L^{q_{0}}\left(\left\{(t, x): T \leq t \leq 2 T, \delta \phi_{m}\left(\bar{T_{0}}\right) \leq \phi_{m}(t)-|x| \leq 2 \delta \phi_{m}\left(\bar{T_{0}}\right)\right\}\right)} \\
&\leq C \delta_{0}^{\frac{1}{q_{0}}}\|t^{-\frac{\alpha}{q_{0}}} F^{j}\|_{L^{\frac{q_{0}}{q_{0}-1}}\left(\left[\bar{T_{0}}, 2\bar{T_{0}}\right] \times \mathbb{R}^{2}\right)},
\end{aligned}
\end{equation}
where $\bar T_0=2^k\bar T$ for any $k=0, 1, 2, ...$ but $\bar T_0\le T$.

For convenience, set $F^{j}(t, x)=: \bar{T}_{0}{ }^{2} F^{j}(\bar{T}_{0} t, \bar{T}_{0} ^{\frac{m+2}{2}} x)$ and $v^{j}(t, x)=: w^{j}(\bar{T}_{0} t, \bar{T}_{0}{ }^{\frac{m+2}{2}} x)$ for $j=0,1$. Then $v^{j}$ solves
\begin{equation}\label{a-45}
\partial_{t}^{2} v^{j}-t^{m} \triangle v^{j}=F^{j}(t, x),
\end{equation}
where $\operatorname{supp} F^{j} \subseteq D_{t, x}^{\delta_0}=\left\{(t, x): 1 \leq t \leq 2, \delta_{0} \phi_{m}(1) \leq \phi_{m}(t)-|x| \leq 2 \delta_{0} \phi_{m}(1)\right\}$. If $T / \bar{T}_{0}$ is written as $T$, then \eqref{a-44} becomes
\begin{equation}\label{a-46}
\begin{aligned}
&\phi_{m}(T)^{\frac{1}{q_{0}}-\frac{\nu}{2}} \delta^{\frac{1}{q_{0}}+\frac{\nu}{2}}\|t^{\frac{\alpha}{q_{0}}} v^{j}\|_{L^{q_{0}}\left(\left\{(t, x): T \leq t \leq 2 T, \delta \phi_{m}(1) \leq \phi_{m}(t)-|x| \leq 2 \delta \phi_{m}(1)\right\}\right)} \\
& \leq C \delta_{0}^{\frac{1}{q_{0}}}\|t^{-\frac{\alpha}{q_{0}}} F^{j}\|_{L^{\frac{q_{0}}{q_{0}-1}}(D_{t, x}^{\delta_0})}.
\end{aligned}
\end{equation}
At this time, by \eqref{a-40},  $1 \leq T \leq \phi_{m}^{-1}\left(10 \phi_{m}(2)\right)$ holds for $(t, x) \in \operatorname{supp} v^{0}$ and $T \leq t \leq 2 T$, which is called the ``relatively small time". On the other hand, it holds that  $T \geq \phi_{m}^{-1}\left(10 \phi_{m}(2)\right)$  for $(t, x) \in \operatorname{supp} v^{1}$ and $T \leq t \leq 2 T$, which is called the ``relatively large time".
In the next two subsections, we shall deal with the two cases, respectively. To simplify notation, we shall omit the superscript $j$  and define
\begin{equation}\label{a-47}
D_{t, x}^{T, \delta}=\left\{(t, x): T \leq t \leq 2 T, \delta \phi_{m}(1) \leq \phi_{m}(t)-|x| \leq 2 \delta \phi_{m}(1)\right\} .
\end{equation}
Thus \eqref{a-46} is equivalent to
\begin{equation}\label{a-48}
\phi_{m}(T)^{\frac{1}{q_{0}}-\frac{\nu}{2}} \delta^{\frac{1}{q_{0}}+\frac{\nu}{2}}\|t^{\frac{\alpha}{q_{0}}} v\|_{L^{q_{0}}(D_{t, x}^{T, \delta})} \leq C \delta_{0}^{\frac{1}{q_{0}}}\|t^{-\frac{\alpha}{q_{0}}} F\|_{L^{\frac{q_{0}}{q_{0}-1}}(D_{t, x}^{\delta_0})}.
\end{equation}

\subsubsection{\large{Large time case}}\label{subsub-1}

 In this case, it is known from \eqref{a-48} and \eqref{a-47}  that the integral domain in the expression of $v$
is $\left\{\delta \phi_{m}(1) \leq \phi_{m}(t)-|x| \leq 2 \delta \phi_{m}(1) \right\}$ and the support of $F$
 is included in $\{\delta_0 \phi_{m}(1) \leq \phi_{m}(t)-|x| \leq$ $2 \delta_0 \phi_{m}(1)\}$. In terms of the  range of $\delta / \delta_{0}$,
 we have the following three cases:\\
(i) $\delta_{0} \leq \delta \leq 40 \cdot 2^{\frac{m+2}{2}} \delta_{0}$;\\
(ii) $\delta \geq 10 \cdot 2^{\frac{m+2}{2}}$;\\
(iii) $40 \cdot 2^{\frac{m+2}{2}} \delta_{0} \leq \delta \leq 10 \cdot 2^{\frac{m+2}{2}}$ for $\delta_0<\f{1}{4}$.\\
Here we point out that in Cases (i)-(iii), the inequality similar to \eqref{a-48} can be established  (see  Section 4B of \cite{HWY4}
for the case of space dimensions $n\ge 3$). However, due to the appearance of the time weight $t^{\f{\alpha}{q_0}}$
and the 2-D case in \eqref{a-48}, we  shall give the details.
\vskip 0.2 true cm

\textbf{Case (i). Small $\delta$}
\vskip 0.2 true cm

Note that due to $\phi_{m}(1)>0$ and $\delta \phi_{m}(1) \leq \phi_m(t)\leq2^{\f{m+2}{2}}\phi_{m}(T)$ in
the support of $v$, then $\delta \lesssim \phi_{m}(T)$ holds.
To prove  \eqref{a-48}, it  suffices to show
\begin{equation}\label{a-49}
\phi_{m}(T)^{\frac{1}{q_{0}}}\|t^{\frac{\alpha}{q_{0}}} v\|_{L^{q_{0}}\left(D_{t, x}^{T, \delta}\right)} \leq C\|t^{-\frac{\alpha}{q_{0}}} F\|_{L^{\frac{q_{0}}{q_{0}-1}}(D_{t, x}^{\delta_0})}.
\end{equation}
Note that
\begin{equation}\label{a-50}
v(t, x)=\int_{1}^{t} H(t, s, x) \mathrm{d} s,
\end{equation}
where
$$
H(t,s,x)=\int_{\mathbb{R}^{2}} e^{i\left\{x \cdot \xi-\left[\phi_{m}(t)-\phi_{m}(s)\right]|\xi|\right\}} a(t, s, \xi) \hat{F}(s, \xi) \mathrm{d} \xi
$$
with
\begin{equation}\label{a-51}
\left|\partial_{\xi}^{\kappa} a(t, s, \xi)\right| \leq C\left(1+\phi_{m}(t)|\xi|\right)^{-\frac{m}{2(m+2)}}\left(1+\phi_{m}(s)|\xi|\right)^{-\frac{m}{2(m+2)}}
|\xi|^{-\frac{2}{m+2}-|\kappa|}.
\end{equation}
Since
\begin{equation}\label{a-52}
\|t^{\frac{\alpha}{q_{0}}} H(t, s, \cdot)\|_{L^{q_{0}}\left(\mathbb{R}^{2}\right)} \leq C|t-s|^{-\frac{2}{q_{0}}\left(1+\frac{m}{4}\right)}\|s^{-\frac{\alpha}{q_{0}}} F(s, \cdot)\|_{L^{\frac{q_{0}}{q_{0}-1}}\left(\mathbb{R}^{2}\right)}
\end{equation}
is shown in Appendix, then by the support condition of $v$ and $F$,  one has
$$
\begin{aligned}
\|t^{\frac{\alpha}{q_{0}}} v\|_{L^{q_{0}}([T, 2 T] \times \mathbb{R}^{2})} \leq\|\int_{1}^{t}\| t^{\frac{\alpha}{q_{0}}} v(t, s, \cdot)\|_{L^{q_{0}}(\mathbb{R}^{2})} \mathrm{d} s\|_{L^{q_{0}}([T, 2 T])}
\leq C \phi_{m}(T)^{-\frac{1}{q_{0}}}\|s^{-\frac{\alpha}{q_{0}}} F\|_{L^{\frac{q_{0}}{q_{0}-1}}([1, 2]\times \mathbb{R}^{2})},
\end{aligned}
$$
which implies \eqref{a-49}.

\vskip 0.2 true cm
{\bf Case (ii). large $\delta$}
\vskip 0.2 true cm

In this case, we have $\phi_{m}(t)-|x| \geq \delta \phi_m(1)\geq 10 \phi_{m}(2)$. As in \eqref{a-50} and \eqref{a-53}, set
$$
v=\sum_{j=-\infty}^{\infty} v_{j}=\sum_{j=-\infty}^{\infty} \int_{1}^{t} \int_{\mathbb{R}^{2}} K_{j}(t, x ; s, y) F(s, y) \mathrm{d} y \mathrm{d}s,
$$
where
\begin{equation}\label{a-59}
K_{j}(t, x ; s, y)=\int_{\mathbb{R}^{2}} e^{i\left\{(x-y) \cdot \xi-\left[\phi_{m}(t)-\phi_{m}(s)\right]|\xi|\right\}} \vp(\frac{|\xi|}{2^{j}}) a(t, s, \xi)  \mathrm{d} \xi.
\end{equation}
In addition, from (3.41) of \cite{HWY1}, one has
\begin{equation}\label{a-60}
\begin{aligned}
\left|K_{j}(t, x ; s, y)\right| \leq & C_{N} \lambda_{j}^{\frac{3}{2}-\frac{2}{m+2}}\left(\left|\phi_{m}(t)-\phi_{m}(s)\right|+\lambda_{j}^{-1}\right)^{-\frac{1}{2}}\left(1+\phi_{m}(t) \lambda_{j}\right)^{-\frac{m}{2(m+2)}} \\
& \times\left(1+\lambda_{j}| | \phi_{m}(t)-\phi_{m}(s)|-|x-y||\right)^{-N},
\end{aligned}
\end{equation}
where $\lambda_j=2^j$ and $N \in \mathbb{R}^{+}$. Then it follows from H\"{o}lder's inequality and the choice of $N=\f{3}{2}-\f{2}{m+2}-\f{m}{2(m+2)}=\frac{m+1}{m+2}$
that
$$
\begin{aligned}
|t^{\frac{\alpha}{q_{0}}} v_{j}| &\leq\|t^{\frac{\alpha}{q_{0}}} K_{j}(t, x ; s, y)\left(\phi_{m}^{2}(t)-|x|^{2}\right)^{\frac{1}{q_{0}}} s^{\frac{\alpha}{q_{0}}}\|_{L^{q_{0}}(D_{s, y}^{\delta_0})}\|\left(\phi_{m}^{2}(t)-|x|^{2}\right)^{-\frac{1}{q_{0}}} s^{-\frac{\alpha}{q_{0}}}
F(s, y)\|_{L^{\frac{q_{0}}{q_{0}-1}}(D_{s, y}^{\delta_0})}\\
& \leq C \phi_{m}(t)^{-\frac{m+1}{m+2}+\frac{1}{q_{0}}+\frac{2 \alpha}{(m+2) q_{0}}}\left(\phi_{m}(t)-|x|\right)^{-\frac{m+1}{m+2}+\frac{1}{q_{0}}}\delta_{0}^{\frac{1}{q_{0}}} \left(\delta \phi_{m}(T)\right)^{-\frac{1}{q_{0}}}\|s^{-\frac{\alpha}{q_{0}}} F\|_{L^{\frac{q_{0}}{q_{0}-1}}(D_{s, y}^{\delta_0})}.
\end{aligned}
$$
In addition, for $-1<\al\leq m$ and $T>1$,
\begin{equation}\label{a-63}
\begin{aligned}
&\|\phi_{m}(t)^{-\frac{m+1}{m+2}+\frac{1}{q_{0}}+\frac{2 \alpha}{(m+2) q_{0}}}\left(\phi_{m}(t)-|x|\right)^{-\frac{m+1}{m+2}+\frac{1}{q_{0}}}\|_{L^{q_{0}}\left(D_{t, x}^{T, \delta}\right)} \\
& \leq C\big(\int_{T}^{2 T} \phi_{m}(t)^{-\f{m+1}{m+2}\cdot q_0+1+\f{2\al}{(m+2)q_0}} \int_{0}^{\phi_{m}(t)-10 \phi_{m}(2)}\left(\phi_{m}(t)-r\right)^{-\f{m+1}{m+2}\cdot q_0+1} r \mathrm{d} r \mathrm{d} t\big)^{\frac{1}{q_{0}}} \\
& \leq C\big(\int_{T}^{2 T} \phi_{m}(t)^{-\f{m+1}{m+2}\cdot q_0+\f{2\al}{(m+2)q_0}+2} \mathrm{d} t\big)^{\frac{1}{q_{0}}} \\
& \leq C \big(\int_{T}^{2T} \phi_m(t)^{-\f{m+4}{m+2}}\mathrm{d} t\big)^{\f{1}{q_0}}\\
& \leq C.
\end{aligned}
\end{equation}
Therefore,
\begin{equation}\label{aa-64}
\begin{aligned}
\|t^{\frac{\alpha}{q_{0}}}v_{j}\|_{L^{q_{0}}(D_{t, x}^{T, \delta})} & \leq C \delta_{0}^{\frac{1}{q_{0}}}\left(\delta \phi_{m}(T)\right)^{-\frac{1}{q_{0}}}\|s^{-\frac{\alpha}{q_{0}}} F\|_{L^{\frac{q_{0}}{q_{0}-1}}(D_{s, y}^{\delta_0})} \\
& \leq C \delta_{0}^{\frac{1}{q_{0}}} \phi_{m}(T)^{-\frac{1}{q_{0}}+\frac{\nu}{2}} \delta^{-\frac{1}{q_{0}}-\frac{\nu}{2}}\|t^{-\frac{\alpha}{q_{0}}} F\|_{L^{\frac{q_{0}}{q_{0}-1}}(D_{t, x}^{\delta_0})},
\end{aligned}
\end{equation}
where we have used the fact of $\delta \lesssim \phi_{m}(T)$ for $\delta \phi_{m}(1) \leq 2^\f{m+2}{2} \phi_{m}(T)$.
Then by Lemma \ref{lem3-E},
\eqref{aa-64} implies the estimate \eqref{a-48} for large $\delta$.

\vskip 0.2 true cm
\textbf{Case (iii). medium $\delta$}
\vskip 0.2 true cm

First, by \eqref{aa-28}, the solution $v$ of \eqref{a-45} can be written as
\begin{equation}\label{a-64}
v(t, x)=\int_{1}^{t} \int_{\mathbb{R}^{2}} e^{i\left(x \cdot \xi-\left(\phi_{m}(t)-\phi_{m}(s)\right)|\xi|\right)} a(t, s, \xi) \hat{F}(s, \xi) \mathrm{d} \xi \mathrm{d} s,
\end{equation}
where  $a(t,s,\xi)$  satisfies that for $\kappa \in \mathbb{N}_{0}^{2}$,
\begin{equation}\label{a-65}
\begin{aligned}
\left|\partial_{\xi}^{\kappa} a(t, s, \xi)\right| & \leq C\left(1+\phi_{m}(t)|\xi|\right)^{-\frac{m}{2(m+2)}}\left(1+\phi_{m}(s)|\xi|\right)^{-\frac{m}{2(m+2)}}|\xi|^{-\frac{2}{m+2}-|\kappa|} \\
& \leq C
\left(1+\phi_{m}(t)|\xi|\right)^{-\frac{m}{2(m+2)}}|\xi|^{-\frac{2}{m+2}-|\kappa|}
\end{aligned}
\end{equation}
or
\begin{equation}\label{aa-65}
\left|\partial_{\xi}^{\kappa} a(t, s, \xi)\right|  \leq C \phi_{m}(t)^{-\frac{m}{2(m+2)}}|\xi|^{-1-|\kappa|}
\end{equation}
We emphasize that the appearance of the decay factors $\left(1+\phi_{m}(t)|\xi|\right)^{-\frac{m}{2(m+2)}}|\xi|^{-\frac{2}{m+2}}$ in \eqref{a-65} and $\phi_{m}(t)^{-\frac{m}{2(m+2)}}|\xi|^{-1}$ in \eqref{aa-65}
is crucial for the proof of \eqref{a-48}.

Set $\tau=\phi_{m}(s)-|y|$, then
\begin{equation}\label{a-66}
\begin{aligned}
|t^{\frac{\alpha}{q_{0}}} v|
&\leq \int_{\delta_{0}}^{2 \delta_{0}}\left|\int_{\mathbb{R}^{2}} \int_{\mathbb{R}^{2}} e^{i\left\{(x-y) \cdot \xi-\left[\phi_{m}(t)-\tau-|y|\right] \xi \mid\right\}}t^{\f{\alpha}{q_0}} a(t, \phi_{m}^{-1}(\tau+|y|), \xi) F\left(\phi_{m}^{-1}(\tau+|y|), y\right) \mathrm{d} \xi \mathrm{d} y\right| \mathrm{d} \tau \\
&\leq C \delta_{0}^{\f{1}{q_0}}(\int_{\delta_{0}}^{2 \delta_{0}}|\int_{\mathbb{R}^{2}} \int_{\mathbb{R}^{2}}e^{i(x-y) \cdot \xi-i[\phi_{m}(t)-\tau-|y|]|\xi|}t^{\frac{\alpha}{q_{0}}}a(t, \phi_{m}^{-1}(\tau+|y|), \xi)\\
&\qquad\qquad\qquad\qquad\qquad\qquad\qquad \times F(\phi_{m}^{-1}(\tau+|y|), y)\mathrm{d}\xi \mathrm{d} y|^{\f{q_0}{q_0-1}}\mathrm{d} \tau )^{\f{q_0-1}{q_0}}.
\end{aligned}
\end{equation}
Due to
$
\phi_{m}(t) \geq\phi_{m}(T)\geq 10 \phi_{m}(2)
$
and $\tau<\phi_{m}(s)<\phi_{m}(2)$, one can replace $\phi_{m}(t)-\tau$ with $\phi_{m}(t)$ in \eqref{a-66} and denote
\begin{equation}\label{a-66-M}
\begin{aligned}
T g(t, x)=\int_{\mathbb{R}^{2}} \int_{\mathbb{R}^{2}} e^{i\left\{(x-y) \cdot \xi -\left[\phi_{m}(t)-|y|\right]|\xi|\right\}}t^{\frac{\alpha}{q_{0}}} b(t, \xi) g(y) \mathrm{d} \xi \mathrm{d} y .
\end{aligned}
\end{equation}
Note that the estimate of $\|Tg\|_{{L^{q_{0}}(D_{t, x}^{T, \delta})}}$ with $\alpha=0$ has been established in \cite{HWY4}
for space dimensions $n\ge 3$, from now on \(\alpha\not=0\) will be assumed.
The proof procedure is divided into two parts with respect to the range of $\al$.

\vskip 0.2 true cm

\textbf{Part 1. $-1<\alpha<0$}

\vskip 0.2 true cm

In this case,  it follows from \eqref{a-65} and direct computation that the symbol $b(t, \xi)$ in \eqref{a-66-M}
satisfies
\[|\partial_\xi^\kappa b(t,\xi)|\leq C\big(1+\phi_m(t)|\xi|\big)^{-\frac{m}{2(m+2)}}|\xi|^{-\frac{2}{m+2}-|\kappa|}.\]
Based on this and \eqref{a-66-M}, motivated by the ideas in  $\S 3$ of \cite{Gls1} and Section 4B3 of \cite{HWY4}, we
introduce the following  operator for $z \in \mathbb{C}$,
\begin{equation}\label{a-67}
\begin{aligned}
\left(T_{z} g\right)(t, x)= & \big(z-\frac{2(m+3+\alpha)}{(m+2)(\alpha+2)}\big) e^{z^{2}} \\
& \times \int_{\mathbb{R}^{2}} \int_{\mathbb{R}^{2}} e^{i\left\{(x-y) \cdot \xi-\left[\phi_{m}(t)-|y|\right] \xi \mid\right\}} t^{\frac{\alpha}{q_{0}}}\left(1+\phi_{m}(t)|\xi|\right)^{-\frac{m}{2(m+2)}} g(y) \frac{\mathrm{d} \xi}{|\xi|^{z}} \mathrm{~d} y,
\end{aligned}
\end{equation}
where $\phi_{m}(t) \geq \phi_{m}(t)-\phi_m(s) \geq 9 \phi_{m}(2)$ and $\delta<10 \cdot 2^{\frac{m+2}{2}}$.
Then by Lemma A. 2 in \cite{HWY4} (the result in Lemma A.2 does not depend on the space dimensions)
and direct computation, \eqref{a-48} can be derived from
\begin{equation}\label{a-69}
\begin{aligned}
\|\left(T_{z} g\right)(t, \cdot) & \|_{L^{q_{0}}\left(\left\{x: \delta \phi_{m}(1) \leq \phi_{m}(t)-|x| \leq 2 \delta \phi_{m}(1)\right\}\right)} \\
& \leq C \phi_{m}(t)^{\frac{\nu}{q_{0}}-\frac{m+4}{q_{0}(m+2)}} \delta^{-\frac{\nu}{q_{0}}-\frac{1}{q_{0}}}\|g\|_{L^{\frac{q_{0}}{q_{0}-1}}\left(\mathbb{R}^{2}\right)} \quad \text { with }\operatorname{Re} z=\frac{2}{m+2} .
\end{aligned}
\end{equation}
Note that
\begin{equation}\label{a-70}
\left\|\left(T_{z} g\right)(t, \cdot)\right\|_{L^{\infty}\left(\mathbb{R}^{2}\right)} \leq C t^{\frac{\alpha}{q_{0}}} \phi_{m}(t)^{-\frac{1}{2}-\frac{m}{2(m+2)}}\|g\|_{L^{1}\left(\mathbb{R}^{2}\right)} \quad \text { with } \operatorname{Re} z=\frac{2(m+3+ \alpha)}{(m+2)(\alpha+2)}
\end{equation}
and
\begin{equation}\label{a-71}
\left\|\left(T_{z} g\right)(t, \cdot)\right\|_{L^{2}\left(\mathbb{R}^{2}\right)} \leq C t^{\frac{\alpha}{q_{0}}} \phi_{m}(t)^{-\frac{m}{2(m+2)}}\left(\phi_{m}(t)^{\nu} \delta^{-(\nu+1)}\right)^{\frac{1}{m+2}}\|g\|_{L^{2}\left(\mathbb{R}^{2}\right)}
\quad \text { with }   \operatorname{Re} z=0
\end{equation}
have been shown in Lemma \ref{A2} and Lemma \ref{A3}, respectively.
Then by interpolation between \eqref{a-70} and \eqref{a-71}, we arrive at \eqref{a-69} and
\begin{equation}\label{a-78}
\|(T g)(t, \cdot)\|_{L^{q_{0}}\left(\left\{x: \delta \phi_{m}(1) \leq \phi_{m}(t)-|x| \leq 2 \delta \phi_{m}(1)\right\}\right)} \leq C \phi_{m}(t)^{\frac{\nu}{2}-\frac{m+4}{q_{0}(m+2)}} \delta^{-\frac{\nu}{2}-\frac{1}{q_{0}}}\|g\|_{L^{\frac{q_{0}}{q_{0}-1}}\left(\mathbb{R}^{2}\right)}.
\end{equation}

\vskip 0.2 true cm

\textbf{Part 2. $0<\alpha\leq m$}

\vskip 0.2 true cm
It follows from \eqref{aa-65} and direct computation that the symbol $b(t, \xi)$ in \eqref{a-66-M}
satisfies
\[|\partial_\xi^\kappa b(t,\xi)|\leq C\phi_m(t)^{-\frac{m}{2(m+2)}}|\xi|^{-1-|\kappa|}.\]
For  $|\xi| \geq 1$, as in Part 1, introduce the operator $\tilde{T}_{z}$ that for $z \in \mathbb{C}$,
\begin{equation}\label{a-79}
\begin{aligned}
(\tilde{T}_{z} g)(t, x)=\big(z-\frac{m+3+ \alpha}{\alpha+2}\big)& e^{z^{2}} \int_{\mathbb{R}^{2}} \int_{\mathbb{R}^{2}}
e^{i\left((x-y) \cdot \xi-\left(\phi_{m}(t)-|y|\right)|\xi|\right)} \\
& \times t^{\frac{\alpha}{q_{0}}} \phi_{m}(t)^{-\frac{m}{2(m+2)}} g(y) \frac{\mathrm{d} \xi}{|\xi|^{z}} \mathrm{d} y.
\end{aligned}
\end{equation}
Note that for $z=\frac{m+3+ \alpha}{\alpha+2}+i \theta$ with $\theta \in \mathbb{R}$ and $0<\alpha \leq m$,
 there exists $\sigma>0$ such that
$
-\frac{m+3+ \alpha}{\alpha+2}<-\frac{m+3}{m+2}-\sigma
$ holds.
Similarly to the proof of \eqref{a-72-1}, one has
\begin{equation}\label{a-81}
\|(\tilde{T}_z g)(t, \cdot)\|_{L^{\infty}\left(\mathbb{R}^2\right)} \leq C t^{\frac{\alpha}{q_0}} \phi_m(t)^{-\frac{1}{2}-\frac{m}{2(m+2)}}\|g\|_{L^1\left(\mathbb{R}^2\right)} \quad \text { with } \operatorname{Re} z=\frac{m+3+ \alpha}{\alpha+2} .
\end{equation}
On the other hand,  for the $L^{2}$ estimate, as in the proofs of \eqref{a-76} and \eqref{a-77} in Appendix, we can get
\begin{equation}\label{a-82}
\|(\tilde{T}_z g)(t, \cdot)\|_{L^2\left(\mathbb{R}^2\right)} \leq C t^{\frac{\alpha}{q_0}} \phi_m(t)^{-\frac{m}{2(m+2)}}\left(\phi_m(t)^\nu \delta^{-(\nu+1)}\right)^{\frac{1}{m+2}}\|g\|_{L^2\left(\mathbb{R}^2\right)} \quad \text { with } \operatorname{Re} z=0 .
\end{equation}
Interpolating  between \eqref{a-81} with \eqref{a-82} yields
$$
\begin{aligned}
\|(\tilde{T}_{z} g)(t, \cdot)  \|_{L^{q_{0}}\left(\R^2\right)} \leq C \phi_{m}(t)^{\frac{\nu}{q_{0}}-\frac{m+4}{q_{0}(m+2)}} \delta^{-\frac{\nu}{q_{0}}-\frac{1}{q_{0}}}\|g\|_{L^{\frac{q_{0}}{q_{0}-1}}\left(\mathbb{R}^{2}\right)} \quad \text { with } \operatorname{Re} z=1.
\end{aligned}
$$
When $0<|\xi|<1$, the analysis on $\|Tg\|_{{L^{q_{0}}(D_{t, x}^{T, \delta})}}$ is completely analogous to that in \eqref{a-73-1},
the related details are omitted.
Then \eqref{a-78} can be established for $0<\alpha \leq m$ and the corresponding proof for the medium $\delta$ case is completed.

Collecting all the results above, \eqref{a-48} is proved for the relatively large time.

\vskip 0.1 true cm
\subsubsection{\large{Small time case}}\label{subsub-2}

In this case, $1\leq T \leq 2\cdot 10^\f{2}{m+2}$ holds.
We now prove \eqref{a-48}.

As in the Subsection \ref{subsub-1}, we shall divide the proof of \eqref{a-48} into the following two parts
in terms of the range of $\delta / \delta_{0}$:\\
(i) $\delta_{0} \leq \delta \leq 10 \cdot 2^{\frac{m+2}{2}} \delta_{0}$;\\
(ii) $10 \cdot 2^{\frac{m+2}{2}} \delta_{0} \leq \delta \leq (2T)^{\frac{m+2}{2}}$ .

We first consider case (i). Note that $\delta / \delta_{0} \in [1,10 \cdot 2^{\frac{m+2}{2}}]$ holds.
Hence in order to prove \eqref{a-48}, it is sufficient to show
\begin{equation}\label{a-83}
\phi_{m}(T)^{\frac{1}{q_{0}}}\|t^{\frac{\alpha}{q_{0}}} v\|_{L^{q_{0}}\left(D_{t, x}^{T, \delta}\right)} \leq C\|t^{-\frac{\alpha}{q_{0}}} F\|_{L^{\frac{q_{0}}{q_{0}-1}}(D_{t, x}^{\delta_0})}.
\end{equation}
Since $\phi_{m}(T)$  has a positive upper and lower bound, it suffices to derive
\begin{equation}\label{a-84}
\|t^{\frac{\alpha}{q_{0}}} v\|_{L^{q_{0}}(D_{t, x}^{T, \delta})} \leq C\|t^{-\frac{\alpha}{q_{0}}} F\|_{L^{\frac{q_{0}}{q_{0}-1}}(D_{t, x}^{\delta_0})}.
\end{equation}
In fact, combining Lemma \ref{lem4} and Lemma \ref{lem5} yields \eqref{a-84} immediately. Then the
related proof for the small $\delta$ case is completed.

Next we deal with case (ii).
In this case, in order to prove \eqref{a-48}, it  suffices to show
\begin{equation}\label{a-85}
\delta^{\frac{1}{q_{0}}}\|t^{\frac{\alpha}{q_{0}}} v\|_{L^{q_{0}}\left(D_{t, x}^{T, \delta}\right)} \leq C \delta_{0}^{\frac{1}{q_{0}}}\|t^{-\frac{\alpha}{q_{0}}} F\|_{L^{\frac{q_{0}}{q_{0}-1}}(D_{t, x}^{\delta_0})},
\end{equation}
where
\begin{equation}\label{a-86}
v=\int_{1}^{2} \int_{\mathbb{R}^{2}} \int_{\mathbb{R}^{2}} e^{\left\{i(x-y) \cdot \xi -\left[\phi_{m}(t)-\phi_{m}(s)\right]|\xi|\right\}} a(t, s, \xi) F(s, y) \mathrm{d} \xi \mathrm{d} y \mathrm{d} s
\end{equation}
with
$$
|a(t, s, \xi)| \leq\left(1+\phi_m(t)|\xi|\right)^{-\frac{m}{2(m+2)}}\left(1+\phi_m(s)|\xi|\right)^{-\frac{m}{2(m+2)}}|\xi|^{-\frac{2}{m+2}}.
$$
As in the proof of Lemma \ref{lem5},  we  denote the dyadic decomposition of $v$ by
\begin{equation}\label{a-87}
v_{\lambda}=\int_{\left[1, 2\right] \times \mathbb{R}^{2}} \int_{\mathbb{R}^{2}} e^{i\left\{(x-y) \cdot \xi-\left[\phi_{m}(t)-\phi_{m}(s)\right]|\xi|\right\}} a_{\lambda}(t, s, \xi) F(s, y) d \xi \mathrm{d} s \mathrm{d} y,
\end{equation}
where $a_{\lambda}(t, s, \xi)=\vp(|\xi| / \lambda) a(t, s, \xi)$ and $\vp$ is defined in \eqref{chi}.
Then we only need to show \eqref{a-85} for $v_{\lambda}$. Set
$$
\tilde{T}_{t, s} F(x)=\int_{\mathbb{R}^{2}} \int_{\mathbb{R}^{2}} e^{i\left\{(x-y) \cdot \xi+\left[\phi_{m}(t)-\phi_{m}(s)\right]|\xi|\right\}} a_{\lambda}(t, s, \xi) F(s, y) d \xi d y.
$$
Then by an analogous analysis in the proof process of \eqref{d-2}, one has
\begin{equation}\label{a-88}
\|t^{\frac{\alpha}{q_{0}}} \tilde{T}_{t, s} F(\cdot)\|_{L^{q_{0}}\left(\mathbb{R}^{2}\right)} \leq C|t-s|^{-\frac{(m+1)q_0-2 \al}{(m+3)q_0}}\|s^{-\frac{\alpha}{q_{0}}} F(s, \cdot)\|_{L^{p_0}\left(\mathbb{R}^{2}\right)}.
\end{equation}
It follows from the Young's inequality that
\begin{equation}\label{a-89}
\begin{aligned}
\|t^{\frac{\alpha}{q_{0}}} v_{\lambda}\|_{L^{q_{0}}(D_{t, x}^{T, \delta})} & \leq C\|\int \|t^{\frac{\alpha}{q_{0}}} \tilde{T}_{t, s} F(\cdot)\|_{L^{q_{0}}\left(\mathbb{R}^{2}\right)} \mathrm{d} s\|_{L^{q_{0}}(\left[T, 2T\right] )}  \\
& \leq C\|\int_{I}|t-s|^{-\frac{(m+1)q_0-2 \al}{(m+3)q_0}}\| s^{-\frac{\alpha}{q_{0}}}F(s, \cdot)\|_{L^{p_{0}}(\mathbb{R}^{2})} d s\|_{L_t^{q_{0}}}  \\
& \leq C\||s|^{-\frac{2}{q_{0}}}\|_{L^{\frac{q_{0}}{2}}(I)}\|s^{-\frac{\alpha}{q_{0}}}F\|_{L^{p_{0}}(\left[1, 2\right] \times \mathbb{R}^{2})}\\
& \leq C \left(\frac{\delta_{0}}{\delta}\right)^{\frac{1}{q_{0}}}\|s^{-\frac{\alpha}{q_{0}}}F\|_{L^{p_{0}}(\left[1, 2\right] \times \mathbb{R}^{2})},
\end{aligned}
\end{equation}
here we have used the facts that $\frac{(m+1)q_0-2 \al}{(m+3)q_0}=\frac{2}{q_0}=1-\left(\frac{1}{p_{0}}-\frac{1}{q_{0}}\right)$
holds and the length of the related integral interval $I$ is controlled by $c\delta_0$. Therefore,
in light of  \eqref{a-89}, then \eqref{a-85} and further \eqref{a-48} are proved.

\subsection{Proof of Theorem~\ref{th2-2} for the endpoint case $q=2$}\label{Sub-4.2}

In this subsection, we shall prove the other endpoint estimate \eqref{a-8} for $q=2$ in Theorem \ref{th2-2}.
Through the analogous proof procedure in Section 5A of \cite{HWY4}, by omitting the related superscript $j$,
it  suffices to show
\begin{equation}\label{a-93}
\phi_{m}(T)^{-\frac{1}{2}+\frac{m-\alpha}{m+2}-\frac{\nu}{2}} \delta^{-\frac{1}{2}+\frac{m-\alpha}{m+2}+\frac{\nu}{2}}\|t^{\frac{\alpha}{2}} v\|_{L^{2}\left(D_{t, x}^{T, \delta}\right)} \leq C \delta_{0}^{\frac{1}{2}}\|s^{-\frac{\alpha}{2}} F\|_{L^{2}\left(\left[1, 2\right] \times \mathbb{R}^{2}\right)} ,
\end{equation}
where $D_{t, x}^{T, \delta}$ is defined in \eqref{a-47}, $\operatorname{supp} F\subseteq D_{s, y}^{\delta_{0}}$,
and $\delta \geq \delta_{0}$.
As in \eqref{a-45}, we decompose $v=v^1+v^0$.

\vskip 0.2 true cm
\subsubsection{\large{Estimate of $v^{1}$}}\label{S-1}
\vskip 0.1 true cm

Note that  $T \geq \phi_{m}^{-1}\left(10 \phi_{m}(2)\right)$ holds for $(t, x) \in \operatorname{supp} v^{1}$. As in Subsection \ref{Sub-4.1},  we shall divide the proof of \eqref{a-93} for $v^1$ into the following two parts with respect to the range of $\delta$.
\vskip 0.2 true cm
\textbf{Case 1. The case of $\delta\geq 10\phi_m(2)$}
\vskip 0.1 true cm

As in Case (ii) of  Subsection \ref{subsub-1}, we now set
$$
v=\sum_{j=-\infty}^{\infty} v_{j}=\sum_{j=-\infty}^{\infty} \iint_{D_{s, y}^{\delta_{0}}} K_{j}(t, x ; s, y) F(s, y) \mathrm{d} y \mathrm{d} s,
$$
where
$$
K_{j}(t, x ; s, y)=\int_{\mathbb{R}^{2}} e^{i\left((x-y) \cdot \xi-\left(\phi_{m}(t)-\phi_{m}(s)\right)|\xi|\right)}
\varphi(\frac{|\xi|}{2^{j}}) a(t, s, \xi)  \mathrm{d} \xi.
$$
Similarly, it follows from \eqref{a-60} with $N=\frac{m+1}{m+2}$ and H\"{o}lder's inequality that
for any small fixed constant $\ve>0$,
\begin{equation}\label{SZH}
\begin{aligned}
|t^{\frac{\alpha}{2}}v_{j}|& \leq\|t^\f{\al}{2}K_{j}(t, x ; s, y)\left(\phi_{m}(t)+|x|\right)^{\frac{1}{2}}\left(\phi_{m}(t)-|x|\right)^{\frac{1}{2}-\frac{1}{m+2}-\epsilon} s^{\frac{\alpha}{2}}\|_{L^{2}(D_{s, y}^{\delta_{0}})} \\
& \qquad\times\|\left(\phi_{m}(t)+|x|\right)^{-\frac{1}{2}}\left(\phi_{m}(t)-|x|\right)^{-\frac{1}{2}+\frac{1}{m+2}+\epsilon} s^{-\frac{\alpha}{2}} F(s, y)\|_{L^{2}(D_{s, y}^{\delta_{0}})}\\
& \leq C \delta_{0}^{\frac{1}{2}} \phi_{m}(t)^{-\frac{m}{2(m+2)}+\frac{\alpha}{m+2}}\left(\phi_{m}(t)-|x|\right)^{-\frac{1}{2}-\epsilon}\left(\delta \phi_{m}(T)\right)^{-\frac{1}{2}} \delta^{\frac{1}{m+2}+\epsilon}\|s^{-\frac{\alpha}{2}} F\|_{L^{2}\left(D_{s, y}^{\delta_{0}}\right)}.
\end{aligned}
\end{equation}
On the other hand,
\begin{equation}\label{SZH-1}
\begin{aligned}
&\|\phi_{m}(t)^{-\frac{m}{2(m+2)}+\frac{\alpha}{m+2}}\left(\phi_{m}(t)-|x|\right)^{-\frac{1}{2}-\epsilon}\|_{L^{2}(D_{t, x}^{T, \delta})}\\
&\quad \leq C\big(\int_{\frac{T}{2}}^{T} \phi_{m}(t)^{\frac{2(1+\alpha)}{m+2}} \mathrm{d} t\big)^{\frac{1}{2}}
\leq C T^{1+\frac{\alpha}{2}}.
\end{aligned}
\end{equation}
Then for $\phi_{m}(t) \geq \delta \geq 10 \phi_{m}(2)$, we have
\begin{equation}\label{a-94}
\begin{aligned}
& \phi_{m}(T)^{-\frac{1}{2}+\frac{m-\alpha}{m+2}-\frac{\nu}{2}} \delta^{-\frac{1}{2}+\frac{m-\alpha}{m+2}+\frac{\nu}{2}}\|t^{\frac{\alpha}{2}} v\|_{L^{2}\left(D_{t, x}^{T, \delta}\right)} \\
& \leq \phi_{m}(T)^{-\frac{1}{2}+\frac{m-\alpha}{m+2}-\frac{\nu}{2}} \delta^{-\frac{1}{2}+\frac{m-\alpha}{m+2}+\frac{\nu}{2}} \delta_{0}^{\frac{1}{2}}\left(\delta \phi_{m}(T)\right)^{-\frac{1}{2}} \delta^{\frac{1}{m+2}+\epsilon} T^{1+\frac{\alpha}{2}}\|t^{-\frac{\alpha}{2}} F\|_{L^{2}(D_{t, x}^{\delta_{0}})} \\
& \leq C \delta_{0}^{\frac{1}{2}} \delta^{-\frac{\alpha+1}{m+2}+\epsilon}\|t^{-\frac{\alpha}{2}} F\|_{L^{2}(D_{t, x}^{\delta_{0}})}\leq C \delta_{0}^{\frac{1}{2}} \|t^{-\frac{\alpha}{2}} F\|_{L^{2}(D_{t, x}^{\delta_{0}})} .
\end{aligned}
\end{equation}
In the last inequality, we have used the fact  of $
-\frac{\alpha+1}{m+2}+\epsilon<0$ for $\alpha>-1$ and small $\epsilon>0$. Thus \eqref{a-93} is proved.

\vskip 0.2 true cm
\textbf{Case 2. $\delta_{0} \leq \delta \leq 10 \phi_{\mathrm{m}}(2)$}
\vskip 0.2 true cm

As in \eqref{a-71} of Subsection  \ref{subsub-1},  we  take
$$
t^{\frac{\alpha}{2}} v=\int_{1}^{t} \int_{\mathbb{R}^{2}} \int_{\mathbb{R}^{2}} e^{i\left\{(x-y) \cdot \xi-\left[\phi_{m}(t)-\phi_{m}(s)\right]|\xi|\right\}} t^{\frac{\alpha}{2}}
\phi_{m}(t)^{-\frac{m}{2(m+2)}}|\xi|^{-1} F(s, y) \mathrm{d} y \mathrm{d} \xi \mathrm{d} s .
$$
Write $v=v_{0}+v_{1}$ with
$$
t^{\frac{\alpha}{2}} v_{1}=\int_{1}^{t} \int_{\mathbb{R}^{2}} \int_{\mathbb{R}^{2}} e^{i\left\{(x-y) \cdot \xi-\left[\phi_{m}(t)-\phi_{m}(s)\right]|\xi|\right\}} \phi_{m}(t)^{-\frac{m}{2(m+2)}}
t^{\frac{\alpha}{2}} \frac{1-\rho_1(\delta \xi)}{|\xi|} F(s, y) \mathrm{d} y \mathrm{d} \xi \mathrm{d} s,
$$
where $\rho_1\in C^{\infty}_0(\R^2)$ and $\rho_1=1$ near the origin. By setting $\phi_{m}(s)=|y|+\tau$ and
using H\"{o}lder's inequality, we arrive at
$$
\begin{aligned}
|t^{\frac{\alpha}{2}} v_{1}| & \leq C \delta_{0}^{\frac{1}{2}}(\int_{\delta_{0}}^{2 \delta_{0}} \mid \int_{\mathbb{R}^{2}} \int_{\mathbb{R}^{2}} e^{i\{(x-y) \cdot \xi-|\phi_{m}(t)-|y|-\tau]|\xi|\}} \phi_{m}(t)^{-\frac{m}{2(m+2)}}\\
&\qquad\qquad\qquad\quad   \times t^{\frac{\alpha}{2}} \frac{1-\rho_1(\delta \xi)}{|\xi|} F(\phi_{m}^{-1}(|y|+\tau), y) \mathrm{d} y \mathrm{d} \xi|^{2} \mathrm{d} \tau)^{\frac{1}{2}} \\
& =: C \delta_{0}^{\frac{1}{2}}(\int_{\delta_{0}}^{2 \delta_{0}}|\tilde{F}(t, \tau, x)|^{2} \mathrm{d} \tau)^{\frac{1}{2}},
\end{aligned}
$$
where $\frac{1-\rho_1(\delta \xi)}{|\xi|}=O(\delta)$. Then one can apply the method of \eqref{a-71} to obtain
\begin{equation}\label{aa-95}
\begin{aligned}
\|\tilde{F}(t, \tau, \cdot)\|_{L^{2}\left(\mathbb{R}^{2}\right)} &\leq t^{\frac{\alpha}{2}} \phi_{m}(t)^{-\frac{m}{2(m+2)}}\left(\phi_{m}(t)^{\nu} \delta^{-(\nu+1)}\right)^{\frac{1}{2}}\delta\|s^{-\f{\al}{2}}F(s, \cdot)\|_{L^{2}\left(\mathbb{R}^{2}\right)}\\ &= C\phi_{m}(t)^{\frac{\nu}{2}-\frac{m}{2(m+2)}} t^{\frac{\alpha}{2}} \delta^{-\frac{\nu+1}{2}+1}\|s^{-\f{\al}{2}}F(s, \cdot)\|_{L^{2}\left(\mathbb{R}^{2}\right)},
\end{aligned}
\end{equation}
which yields
\begin{equation}\label{a-95}
\|t^{\frac{\alpha}{2}} v_{1}\|_{L^{2}\left(D_{t, x}^{T, \delta}\right)} \leq C \delta_{0}^{\frac{1}{2}} \delta^{-\frac{\nu+1}{2}+1} \phi_{m}(T)^{\frac{\nu}{2}-\frac{m-2}{2(m+2)}} T^{\frac{\alpha}{2}}\|s^{-\f{\al}{2}}F\|_{L^{2}\left(\left[1, 2\right] \times \mathbb{R}^{2}\right)}.
\end{equation}
Therefore, by $\delta \leq 10 \phi_{m}(2)$, the estimate \eqref{a-93} for $v_{1}$ can be gotten.

We next treat $v_{0}$. Notice that
\begin{equation}\label{v01}
\begin{aligned}
& |\int_{|\xi| \leq 1} e^{i\left\{(x-y) \cdot \xi-\left[\phi_m(t)-\phi_m(s)\right]|\xi|\right\}} \phi_m(t)^{-\frac{m}{2(m+2)}} t^{\frac{\alpha}{2}} \frac{\rho_1(\delta \xi)}{|\xi|} \mathrm{d} \xi|\leq C(1+|x-y|)^{-\frac{1}{2}} \phi_m(t)^{-\frac{m}{2(m+2)} }t^{\frac{\alpha}{2}},
\end{aligned}
\end{equation}
here we have used the fact of $\phi_{m}(t)-\phi_{m}(s) \geq|x-y|$ for  $(s, y) \in \operatorname{supp} F$
and $(t, x) \in \operatorname{supp} v$, which has been proved in (5-7) of \cite{HWY4}.
In addition, the corresponding inequality \eqref{a-93}  holds when $t^{\frac{\alpha}{2}} v$ is replaced by
$$
t^{\frac{\alpha}{2}} v_{01}=\int_{1}^{2} \int_{\mathbb{R}^{2}} \int_{|\xi| \leq 1} e^{i\left\{(x-y) \cdot \xi-\left[\phi_{m}(t)-\phi_{m}(s)\right]|\xi|\right\}}t^{\frac{\alpha}{2}}
\phi_{m}(t)^{-\frac{m}{2(m+2)}} \frac{\rho_1(\delta \xi)}{|\xi|} F(s, y) \mathrm{d} y \mathrm{d} \xi \mathrm{d} s.
$$
By \eqref{v01} and $|y| \leq \phi_{m}(2)$, it follows from  direct computation that
$$
\begin{aligned}
&\|t^{\frac{\alpha}{2}} v_{01}\|_{L_{t,x}^{2}(D_{t, x}^{T, \delta})} \\
&\leq C\|\iint(1+|x-y|)^{-\frac{1}{2}} t^{\frac{\alpha}{2}} \phi_m(t)^{-\frac{m}{2(m+2)}} F(s, y) \mathrm{d} y \mathrm{d} s\|_{L_{t, x}^2(D_{t, x}^{T, \delta})}\\
& \leq C\phi_m(T)^{-\f{m}{2(m+2)}}T^{\f{\alpha}{2}}\|s^{-\frac{\alpha}{2}} F\|_{L_{s, y}^2(D_{s, y}^{\delta_0})}\|(\int_{\frac{T}{2}}^T \int_{\delta\phi_{m}(1) \leq \phi_m(t)-|x| \leq 2 \delta\phi_{m}(1)}(1+|x-y|)^{-1} \mathrm{d} x \mathrm{d} t)^{\frac{1}{2}}\|_{L_{s, y}^2(D_{s, y}^{\delta_0})}\\
&\leq C\phi_m(T)^{-\f{m}{2(m+2)}}T^{\f{\alpha}{2}}\|s^{-\frac{\alpha}{2}} F\|_{L_{s, y}^2(D_{s, y}^{\delta_0})}\times C(\delta_0 \delta T)^{\frac{1}{2}} .
\end{aligned}
$$
Then by $\delta \leq 10 \phi_{m}(2)$, $T \geq \phi_{m}^{-1}\left(10 \phi_{m}(2)\right)$ and $\alpha \leq m$, the left side of \eqref{a-93} can be controlled by
$$
\begin{aligned}
& \phi_{m}(T)^{-\frac{1}{2}+\frac{m-\alpha}{m+2}-\frac{\nu}{2}} \delta^{-\frac{1}{2}+\frac{m-\alpha}{m+2}+\frac{\nu}{2}} \phi_{m}(T)^{\frac{\alpha}{m+2}-\frac{m}{2(m+2)}}\left(\delta_{0} \delta T\right)^{\frac{1}{2}}\|s^{-\frac{\alpha}{2}} F\|_{L^{2}} \\
& \leq C \delta^{\frac{m-\alpha}{m+2}} \delta_{0}^{\frac{1}{2}}\|s^{-\frac{\alpha}{2}} F\|_{L_{s, y}^{2}\left(D_{s, y}^{\delta_{0}}\right)} \leq C \delta_{0}^{\frac{1}{2}}\|s^{-\frac{\alpha}{2}} F\|_{L^{2}\left(\left[1, 2\right] \times \mathbb{R}^{2}\right)}.
\end{aligned}
$$
Note that
\begin{equation}\label{a-96}
\phi_{m}(T)^{-\frac{1}{2}+\frac{m-\alpha}{m+2}-\frac{\nu}{2}} \delta^{-\frac{1}{2}+\frac{m-\alpha}{m+2}+\frac{\nu}{2}}\|t^{\frac{\alpha}{2}} v_{02}\|_{L^{2}\left(D_{t, x}^{T, \delta}\right)} \leq C \delta_{0}^{\frac{1}{2}}\|s^{-\frac{\alpha}{2}} F\|_{L^{2}\left(\left[1, 2\right] \times \mathbb{R}^{2}\right)}
\end{equation}
with
$$
v_{02}=\int_{1}^{2} \int_{\mathbb{R}^{2}} \int_{|\xi| \geq 1} e^{i\left((x-y) \cdot \xi-\left(\phi_{m}(t)-\phi_{m}(s)\right)|\xi|\right)} \phi_{m}(t)^{-\frac{m}{2(m+2)}} \frac{\rho_1(\delta \xi)}{|\xi|} F(s, y) \mathrm{d} y \mathrm{d} \xi \mathrm{d} s
$$
has been shown in Lemma \ref{A4} in Appendix,
then \eqref{a-93} is derived.

\vskip 0.2 true cm
\subsubsection{\large{Estimate of $v^{0}$}}\label{S-2}
\vskip 0.1 true cm

In this case, it suffices to show \eqref{a-93}  for $1\leq\phi_{m}(T) \leq 10 \phi_{m}(2)$ and $\delta_{0} \leq \delta \lesssim \phi_m(T)\leq 10 \phi_{m}(2)$. Then it is only required to prove
\begin{equation}\label{a-97}
\delta^{-\frac{1}{2}+\frac{m-\alpha}{m+2}+\frac{\nu}{2}}\|t^{\frac{\alpha}{2}} v\|_{L^{2}(D_{t, x}^{T, \delta})} \leq C \delta_{0}^{\frac{1}{2}}\|s^{-\frac{\alpha}{2}} F\|_{L^{2}\left(\left[1, 2\right] \times \mathbb{R}^{2}\right)}.
\end{equation}
As in Case 2 of  Subsection \ref{S-1}, set
$$
\begin{aligned}
v&=\int_{1}^{2} \int_{\mathbb{R}^{2}} \int_{\mathbb{R}^{2}} e^{i\left((x-y) \cdot \xi-\left(\phi_{m}(t)-\phi_{m}(s)\right)|\xi|\right)}|\xi|^{-1} F(s, y) \mathrm{d} y \mathrm{d} \xi \mathrm{d} s\\
&=v_{0}+v_{1}
\end{aligned}
$$
with
$$
v_{1}=\int_{1}^{2} \int_{\mathbb{R}^{2}} \int_{\mathbb{R}^{2}} e^{i\left((x-y) \cdot \xi-\left(\phi_{m}(t)-\phi_{m}(s)\right)|\xi|\right)} \frac{1-\rho_1(\delta \xi)}{|\xi|} F(s, y) \mathrm{d} y \mathrm{d} \xi \mathrm{d} s,
$$
where $\rho_1\in C^{\infty}_0(\R^2)$ and $\rho_1=1$ near the origin. By setting $\phi_{m}(s)=|y|+\tau$, then
$$
\begin{aligned}
|t^{\frac{\alpha}{2}}v_{1}| \leq C \delta_{0}^{\frac{1}{2}}\big(\int_{\delta_{0}}^{2 \delta_{0}}\big|\int_{\mathbb{R}^{2}} \int_{\mathbb{R}^{2}} e^{i\left((x-y) \cdot \xi-\left(\phi_{m}(t)-\tau-|y|\right)|\xi|\right)} t^{\frac{\alpha}{2}}\delta F\left(\phi_{m}^{-1}(|y|+\tau), y\right) \mathrm{d} y \mathrm{d} \xi\big|^{2} \mathrm{d} s\big)^{\frac{1}{2}}.
\end{aligned}
$$
It follows from  an analogous proof procedure in \eqref{a-71} with $Rez=0$ that
\begin{equation}\label{a-98}
\|t^{\frac{\alpha}{2}} v_{1}\|_{L^{2}} \leq C \delta_{0}^{\frac{1}{2}} \delta^{-\frac{\nu+1}{2}+1} \phi_{m}(T)^{\frac{\nu}{2}+\frac{\alpha}{m+2}}\|s^{-\frac{\alpha}{2}} F\|_{L^{2}\left(\left[1, 2\right] \times \mathbb{R}^{2}\right)}.
\end{equation}
This, together with $\delta_0\leq\delta \lesssim  10 \phi_{m}(2)$ and $\phi_{m}(T) \geq 1$,
yields the estimate \eqref{a-97} for $v_{1}$.

We next estimate $v_{0}$. Similarly to the treatment in \eqref{v01}, one has
$$
\big|\int_{|\xi| \leq 1} e^{i\left((x-y) \cdot \xi-\left(\phi_{m}(t)-\phi_{m}(s)\right)|\xi|\right)}t^{\frac{\alpha}{2}} \frac{\rho_1(\delta \xi)}{|\xi|} \mathrm{d} \xi\big| \leq C(1+|x-y|)^{-\frac{1}{2}}t^{\frac{\alpha}{2}} .
$$
If we replace $t^{\frac{\alpha}{2}}v$ by
$$
t^{\frac{\alpha}{2}}v_{01}=\int_{1}^{2} \int_{\mathbb{R}^{2}} \int_{|\xi| \leq 1} e^{i\left((x-y) \cdot \xi-\left(\phi_{m}(t)-\phi_{m}(s)\right)|\xi|\right)} t^{\frac{\alpha}{2}}\frac{\rho_1(\delta \xi)}{|\xi|} F(s, y) \mathrm{d} y \mathrm{d} \xi \mathrm{d} s,
$$
due to $|x| \leq \phi_{m}(t) \lesssim 10 \phi_{m}(2)$, then
$$
\begin{aligned}
&\|t^{\frac{\alpha}{2}} v_{01}\|_{L^{2}\left(D_{t, x}^{T, \delta}\right)} \\
&\leq\big\|t^{\frac{\alpha}{2}} \iiint_{|\xi| \leq 1} e^{i\left\{(x-y) \cdot \xi-\left[\phi_m(t)-\phi_m(s)\right]|\xi|\right\}} \frac{\rho_1(\delta \xi)}{|\xi|} \mathrm{d} \xi F(s, y) \mathrm{d} y \mathrm{d} s\big\|_{L^2(D_{t, x}^{T, \delta})}\\
&\leq C\|s^{-\frac{\alpha}{2}} F\|_{L_{s, y}^{2}(D_{s, y}^{\delta_{0}})}\big\|\big(\int_{T}^{2T} \int_{\delta\phi_m(1) \leq \phi_{m}(t)-|x| \leq 2 \delta\phi_m(1)}(1+|x-y|)^{-1} \mathrm{d} x \mathrm{d} t\big)^{\frac{1}{2}}\big\|_{L_{s, y}^{2}(D_{s, y}^{\delta_{0}})}\\
& \leq C\|s^{-\frac{\alpha}{2}} F\|_{L_{s, y}^{2}(D_{s, y}^{\delta_{0}})}\left(\delta_{0} \delta T\right)^{\frac{1}{2}} .
\end{aligned}
$$
Then
$$
\delta^{-\frac{1}{2}+\frac{m-\alpha}{m+2}+\frac{\nu}{2}}\|t^{\frac{\alpha}{2}} v_{01}\|_{L^{2}\left(D_{t, x}^{T, \delta}\right)}
\leq C \delta_{0}^{\frac{1}{2}}\|s^{-\frac{\alpha}{2}} F\|_{L^{2}\left(\left[1, \infty\right) \times \mathbb{R}^{2}\right)}.
$$
Note that
$$
t^{\frac{\alpha}{2}} v_{02}=\int_{1}^{2} \int_{\mathbb{R}^{2}} \int_{|\xi| \geq 1} e^{i\left((x-y) \cdot \xi-\left(\phi_{m}(t)-\phi_{m}(s)\right)|\xi|\right)}t^{\frac{\alpha}{2}}  \frac{\rho_1(\delta \xi)}{|\xi|} F(s, y) \mathrm{d} y \mathrm{d} \xi \mathrm{d} s.
$$
Due to $\delta \lesssim \phi_{m}(T) \leq 10 \phi_{m}(2)$, then it follows from the estimate of $v_{02}$ in
Case 2 of Subsection \ref{S-1} that
\begin{equation}\label{a-99}
\delta^{-\frac{1}{2}+\frac{m-\alpha}{m+2}+\frac{\nu}{2}}\|t^{\frac{\alpha}{2}} v_{02}\|_{L^{2}\left(D_{t, x}^{T, \delta}\right)} \leq C \delta_{0}^{\frac{1}{2}}\|s^{-\frac{\alpha}{2}} F\|_{L^{2}\left(\left[1, 2\right] \times \mathbb{R}^{2}\right)}.
\end{equation}

Combining the estimates on $v^{0}$ and $v^{1}$ in Subsection \ref{S-1} and Subsection \ref{S-2}, respectively,
then \eqref{a-8} and further Theorem \ref{th2-2} are proved.

\section{Angular mixed-norm Strichartz estimates for $\partial_t^2 u-t^m\Delta u =t^{m-p+1}|u|^p$}\label{E-5}

At first, we establish some angular mixed-norm Strichartz estimates for the 2-D inhomogeneous Tricomi equation.
Let $w$ solve
\begin{equation}\label{q-5}
\begin{cases}\partial_t^2 w-t^{m} \Delta w=F(t, x), &(t,x)\in [1,\infty)\times \R^2,\\
w(1, x)=0,\quad \p_t w(1, x)=0.
\end{cases}
\end{equation}
Note that
\begin{equation}\label{equ:3.40}
t^{\f{m-p+1}{p+1}} w(t,x)=(AF)(t,x)\equiv
\int_1^t\int_{\R^2}e^{i\left(x\cdot\xi-(\phi_m(t)-\phi_m(\tau))|\xi|\right)}
a(t,\tau,\xi)\hat{F}(\tau,\xi)\,\mathrm{d}\xi \mathrm{d}\tau,
\end{equation}
where
\[
\bigl| \partial_\xi^\kappa a(t,\tau,\xi)\bigr|\leq
Ct^{\f{m-p+1}{p+1}} \left(1+\phi_m(t)|\xi|\right)^{-\f{m}{2(m+2)}}
\left(1+\phi_m(\tau)|\xi|\right)^{-\f{m}{2(m+2)}}|\xi|^{-\frac{2}{m+2}-|\kappa|}.
\]
By a similar argument in Lemma \ref{lem3.4-L}, we can conclude that if $\hat{F}(\tau, \xi)=0$ when $|\xi| \notin\left[\frac{1}{2}, 1\right]$, then
\begin{equation}\label{q-023}
\big\|t^{\f{m-p+1}{p+1}}\int_{\mathbb{R}} V_{2}\left(t, D_{x}\right) V_{1}\left(\tau, D_{x}\right) F(\tau, x) \mathrm{d} \tau\big\|_{L_{t}^{q} L_{r}^{\nu} L_{\theta}^{2}\left([1,+\infty)\times [0, +\infty)\times [0, 2\pi]\right)} \leq C\|t^{\f{p-m-1}{p+1}}F\|_{L_{t}^{\tilde{q}^{\prime}} L_{r}^{\tilde{\nu}^{\prime}} L_{\theta}^{2}},
\end{equation}
where
\begin{equation}\label{q-23}
\frac{1}{q} \leq \f{m+1}{2}-\frac{m+2}{2\nu}-\f{m-p+1}{p+1}, \quad \frac{1}{\tilde{q}} \leq \f{m+1}{2}-\frac{m+2}{2\tilde{\nu}}-\f{m-p+1}{p+1}
\end{equation}
and
\begin{equation}\label{q-24}
\f{m-p+1}{p+1}+\frac{1}{q}+\frac{m+2}{\nu}=\frac{1}{\tilde{q}^{\prime}}+\frac{m+2}{\tilde{\nu}^{\prime}}+\f{p-m-1}{p+1}-2
\end{equation}
with $\f{1}{\tilde{q}}+\f{1}{\tilde{q}^{\prime}}=1$ and $\f{1}{\tilde{r}}+\f{1}{\tilde{r}^{\prime}}=1$.

Analogously treated as in  \eqref{q-14}-\eqref{q-15}, we can remove the restriction on the support of $\hat{F}$
in \eqref{q-023} and further get the following estimate.
\begin{lemma}\label{lem2}
Let $w$ be the solution of \eqref{q-5}. If $q, \nu, \tilde{q}, \tilde{\nu} \geq 2$ and satisfy \eqref{q-23}-\eqref{q-24}, then
$$
\big\|t^{\f{m-p+1}{p+1}}w\big\|_{L_{t}^{q} L_{r}^{\nu} L_{\theta}^{2}\left([1,+\infty)\times [0, +\infty)\times [0, 2\pi]\right)} \leq C\|t^{\f{p-m-1}{p+1}}F\|_{L_{t}^{\tilde{q}^{\prime}} L_{r}^{\tilde{\nu}^{\prime}} L_{\theta}^{2}\left([1,+\infty)\times [0, +\infty)\times [0, 2\pi]\right)}.
$$
\end{lemma}

In order to prove Theorem \ref{TH-2}, we  now study some  priori estimates on the following semilinear problem
\begin{equation}\label{q-2}
\left\{ \enspace
\begin{aligned}
&\partial_t^2 u-t^m\Delta u =t^{m-p+1}|u|^p, &&
 (t,x)\in [1,\infty)\times \R^{2},\\
&u(1,x)=u_0(x), \quad \partial_{t} u(1,x)=u_1(x),
\end{aligned}
\right.
\end{equation}
where $0<m\leq \sqrt{2}-1$, $m+2\leq p< \f{3m+7}{m+3}$ and $u_0(x), u_1(x)\in C_0^{\infty}(\Bbb R^2)$.

Let $Z$ be the operator in Theorem \ref{TH-2}, then by Lemma~\ref{lem1} and Lemma \ref{lem2}, we have
\begin{equation}\label{q-25}
\begin{aligned}
& \sum_{|\beta| \leq 1}(\|t^{\f{m-p+1}{p+1}}Z^{\beta} u\|_{L_{t}^{q} L_{r}^{\nu} L_{\theta}^{2}\left([1,+\infty)\times [0, +\infty)\times [0, 2\pi]\right)}+\|t^{\f{m-p+1}{p+1}}Z^{\beta} u\|_{L_{t}^{\infty} \dot{H}^{s}\left([1,\infty)\times \R^2\right)}) \\
& \leq C \sum_{|\beta| \leq 1}(\|Z^{\beta} u_0\|_{\dot{H}^{s}\left(\mathbb{R}^{2}\right)}+\|Z^{\beta} u_1\|_{\dot{H}^{s-\frac{2}{m+2}}\left(\mathbb{R}^{2}\right)}+\|t^{\f{p-m-1}{p+1}}Z^{\beta}( |u|^p)\|_{L_{t}^{\tilde{q}^{\prime}} L_{r}^{\tilde{\nu}^{\prime}} L_{\theta}^{2}\left([1,+\infty)\times [0, +\infty)\times [0, 2\pi]\right)}),
\end{aligned}
\end{equation}
where $q, \nu, \tilde{q}, \tilde{\nu} \geq 2$ satisfy
\begin{equation}\label{q-26}
\f{m-p+1}{p+1}+\frac{1}{q}+\frac{m+2}{\nu}=\frac{1}{\tilde{q}^{\prime}}+\frac{m+2}{\tilde{\nu}^{\prime}}+\f{p-m-1}{p+1}-2=\f{m+2}{2}(1-s),
\end{equation}
\begin{equation}\label{q-27}
\frac{1}{q} \leq \f{m+1}{2}-\frac{m+2}{2\nu}-\f{m-p+1}{p+1}, \quad \frac{1}{\tilde{q}} \leq \f{m+1}{2}-\frac{m+2}{2\tilde{\nu}}-\f{m-p+1}{p+1}
\end{equation}
and
\begin{equation}\label{q-28}
\f{1}{\nu}+\f{m-p+1}{(m+2)(p+1)}>0.
\end{equation}
Set
$$
q=p \tilde{q}^{\prime}, \quad \nu=p \tilde{\nu}^{\prime}.
$$
This, together with \eqref{q-26}, yields
$$
s=1-\frac{2(m-p+1)+4}{(m+2)(p-1)}.
$$
Meanwhile the restriction conditions on $\nu$ and $q$ become
\begin{equation}\label{q-29}
\f{1}{q}+\f{m+2}{\nu}=\f{m-p+3}{p-1}-\f{m-p+1}{p+1},
\end{equation}
\begin{equation}\label{q-30}
\f{1}{q}+\f{m+2}{2\nu}\leq\f{m+1}{2}-\f{m-p+1}{p+1},
\end{equation}
\begin{equation}\label{q-31}
2\leq q \leq 2p,
\end{equation}
\begin{equation}\label{q-32}
2\leq \nu\leq 2p.
\end{equation}
Substituting $q=p \tilde{q}^{\prime}$ and $\nu=p \tilde{\nu}^{\prime}$ into \eqref{q-27} yields
\begin{equation}\label{q-33}
\frac{1}{q}+\frac{m+2}{2 \nu} \geq \frac{3}{2 p}+\f{m-p+1}{(p+1)p}.
\end{equation}

Under the restrictions \eqref{q-28}-\eqref{q-33}, we now investigate the indices in \eqref{q-25}.
Note that by \eqref{q-29}, if $\nu=\nu(p, m)$ is chosen, then $q=q(p, m)$ can be determined.
When $p\in [m+2, \f{3m+7}{m+3})$ with $0<m\leq \sqrt{2}-1$, we next illustrate how to select suitable $(q, \nu)$ in \eqref{q-25}
to fulfill \eqref{q-28}-\eqref{q-33}.

\vskip 0.2 true cm

\textbf{\large{Case 1. $\nu=p$}}

\vskip 0.2 true cm

In this case, it follows from \eqref{q-29} that
\begin{equation}\label{c2}
q=\f{p(p-1)(p+1)}{(p+1)(m+2-mp)+2p(m-p+1)}.
\end{equation}
Due to $0<m\leq \sqrt{2}-1$ and $p< \f{3m+7}{m+3}$, then $q>0$ holds.
At this time, by direct computation we have from \eqref{q-30} that
$$
\f{1}{q}+\f{m+2}{2\nu}\leq\f{m+1}{2}-\f{m-p+1}{p+1}\Leftrightarrow
(p+1)[(m+1)p^2-(2m+5)p-(m+2)]\geq0.
$$
Because of
$$
2p\geq\nu=p \geq m+2>2,
$$
\eqref{q-30} and \eqref{q-32} are fulfilled. In addition, \eqref{q-28} and \eqref{q-33} lead to
$$
\text{$p^2-(2m+3)p-(m+2)<0$ and $(p+1)[(m+3)p-(3m+7)]\leq0$, respectively.}
$$
Then \eqref{q-28} and \eqref{q-33} are satisfied by $p<\f{3m+7}{m+3}$.

On the other hand,
by \eqref{q-31} and \eqref{c2}, $q \geq 2$ and $q\leq 2p$ are equivalent to
\begin{equation}\label{c1}
p^3+(2m+4)p^2-(4m+9)p-(2m+4)\geq0
\end{equation}
and
\begin{equation}\label{c8}
(2m+5)p^2-(4m+8)p-(2m+5)\leq0 \Leftrightarrow p\leq p_1(m),\quad\text{respectively},
\end{equation}
where $p_1(m)$ is the positive root of the quadratic equation
\begin{equation}\label{p1}
(2m+5)p^2-(4m+8)p-(2m+5)=0.
\end{equation}

It is not difficult to verify that for $m>0$ and $p>m+2$, the left hand side function in \eqref{c1} is monotonically increasing with respect to the variable $p$ and
$(m+2)^3+(2m+4)(m+2)^2-(4m+9)(m+2)-(2m+4)>0$ holds.
Then \eqref{c1} is fulfilled. In addition, a direct computation yields that when $0<m\leq \sqrt{2}-1$, $p_1(m)\leq \f{3m+7}{m+3}$ holds; when $0<m<0.092$, $m+2<p_1(m)$ holds; when $0.092\leq m<\sqrt{2}-1$, $p_1(m)<m+2$ holds.

In summary, for $0<m<0.092$ and $m+2<p<p_1(m)$, by choosing $\nu=p$ the restrictions \eqref{q-28}-\eqref{q-33} can be fulfilled simultaneously.
Note that there still exist some gaps  of $p$ and $m$ to satisfy \eqref{q-28}-\eqref{q-33}.
Next, we continue to treat these gaps.

\vskip 0.2 true cm

\textbf{\large{Case 2. $\nu=p+\f{1}{3}$}}

\vskip 0.2 true cm

Note that by $\nu=p \tilde{\nu}^{\prime}$ and $\tilde{\nu}^{\prime}\geq1$, $p$ is the lower bound of $\nu$.
Let $\nu=p+\frac{1}{3}$. Then it follows from \eqref{q-29} that
$$
q=\f{(p-1)(p+1)(3p+1)}{(m+2)(-3p^2+6p+5)}.
$$
Due to $0<m\leq \sqrt{2}-1$ and $p< \f{3m+7}{m+3}$, then $q>0$ holds.
In addition, by $p\geq m+2>2$, we have $2\leq\nu=p+\f{1}{3}\leq2p$. Note that $q \geq 2$ is equivalent to
\begin{equation}\label{c3}
3p^3+(6m+13)p^2-(12m+27)p-(10m+21)\geq0.
\end{equation}
It follows from direct computation that for $m>0$ and $p>m+2$, \eqref{c3} holds since
the left hand side function in
\eqref{c3} is monotonically increasing with respect to the variable $p$ and
$3(m+2)^3+(6m+13)(m+2)^2-(12m+27)(m+2)-(10m+21)>0$.

On the other hand, one has
\begin{equation}\label{c4}
\f{1}{\nu}+\f{m-p+1}{(m+2)(p+1)}>0 \Longleftrightarrow 3p^2-(6m+8)p-(4m+7)<0,
\end{equation}
\begin{equation}\label{c5}
q \leq 2 p \Longleftrightarrow (6m+15)p^3-(12m+31)p^2-(10m+23)p-1\leq0
\end{equation}
and
\begin{equation}\label{c6}
\frac{1}{q}+\frac{m+2}{2 \nu} \geq \frac{3}{2 p}+\f{m-p+1}{(p+1)p} \Longleftrightarrow (3m+9)p^3-(6m+11)p^2-(11m+25)p-(2m+5)\leq0.
\end{equation}

Similarly to the verification of \eqref{c3}, we can obtain that for $0<m\leq\sqrt{2}-1$ and $m+2\leq p<\f{3m+7}{m+3}$,
\eqref{c4} - \eqref{c6} hold. Thus,  \eqref{q-28}, \eqref{q-31} and \eqref{q-33} are fulfilled.
Meanwhile, \eqref{q-30} is equivalent to
$$
G_m(p)=(3m+9)p^3-(2m+3)p^2-(11m+25)p-(6m+13)\geq0.
$$
One can verify that for $0<m\leq\sqrt{2}-1$ and $p>m+2$,  $G_m(p)$ is monotonically increasing with respect to
$p$. Denote ${\bar p}^*(m)$ by the positive root of the cubic equation
\begin{equation}\label{c7}
G_m(p)=0.
\end{equation}
Substituting $p=m+2$ into \eqref{c7} yields that when $0<m<m_1=0.048$, $m+2<{\bar p}^*(m)$ holds; when $m_1< m\leq\sqrt{2}-1$,
${\bar p}^*(m)<m+2$ holds.

In conclusion, either $0<m<m_1$ for ${\bar p}^*(m)\leq p<\f{3m+7}{m+3}$ or $ m_1<m\leq\sqrt{2}-1$ for $m+2\leq p<\f{3m+7}{m+3}$,
by choosing $\nu=p+\f{1}{3}$ the restrictions \eqref{q-28}-\eqref{q-33} are satisfied simultaneously.

Note that by \eqref{p1} and \eqref{c7}, for $m>0$
$${\bar p}^*(m)<p_1(m).$$
Therefore, collecting the analyses in Case 1 and Case 2, one obtains that $\nu$ can be determined as follows:
when $m_1=0.048<m\leq\sqrt{2}-1$, $\nu=p+\f{1}{3}$ is taken for $m+2\leq p<\f{3m+7}{m+3}$;
when $0<m<m_1$, $\nu=p$ for $m+2\leq p<{\bar p}^*(m)$ and $\nu=p+\f{1}{3}$ for ${\bar p}^*(m)\leq p<\f{3m+7}{m+3}$
are taken. Based on the suitable selections of $(q, \nu)$ above,
we will make use of \eqref{q-25} to show the global existence of weak solution $u$ to
problem \eqref{q-2} for $m+2 \leq p < \f{3m+7}{m+3}$ and $0<m\leq \sqrt{2}-1$.

\section{Proofs of Theorem~\ref{thm:global existence-L} and  Theorem~\ref{TH-1}}\label{E-1}
Based on Lemmas \ref{th2-1} and \ref{th2-2}, we next show Theorems \ref{thm:global existence-L} and \ref{TH-1}.
\vskip 0.2 true cm

\begin{proof}
[Proof of Theorem~\ref{thm:global existence-L}]

We shall use a standard Picard iteration  to prove Theorem \ref{thm:global existence-L}.
Let $u_{-1} \equiv 0$ and $u_{k}$ ($k\in\Bbb N_0$) solve
\begin{equation}\label{b-4}
\begin{cases}
\partial_{t}^{2} u_{k}-t^{m} \triangle u_{k}=t^{\alpha}\left|u_{k-1}\right|^{p}, \quad(t, x) \in\left(1, \infty\right) \times \mathbb{R}^{2}, \\
u_{k}\left(1, x\right)=u_0(x) \quad \partial_{t} u_{k}\left(1, x\right)=u_1(x) .
\end{cases}
\end{equation}
For $p \in\left(p_{c r i t}(m, \alpha), p_{\text {conf }}(m, \alpha)\right)$, one can choose a
fixed number $\gamma>0$ such that
$$
\frac{1}{p(p+1)}<\gamma<\frac{m+1}{m+2}-\frac{m+4+2\al}{(m+2)(p+1)}=\f{(m+1)p-(3+2\al)}{(m+2)(p+1)}.
$$
Denote
$$
\begin{aligned}
& M_{k}=\|(\left(\phi_{m}(t)+2\right)^{2}-|x|^{2})^{\gamma} t^{\frac{\alpha}{p+1}} u_{k}\|_{L^{q}\left(\left[1, \infty\right) \times \mathbb{R}^{2}\right)}, \\
& N_{k}=\|(\left(\phi_{m}(t)+2\right)^{2}-|x|^{2})^{\gamma} t^{\frac{\alpha}{p+1}}\left(u_{k}-u_{k-1}\right)\|_{L^{q}\left(\left[1, \infty\right) \times \mathbb{R}^{2}\right)}.
\end{aligned}
$$
Note that $u_{0}$ solves
\begin{equation}\label{b-5}
\begin{cases}
\partial_{t}^{2} u_{0}-t^{m} \triangle u_{0}=0, \quad(t, x) \in\left[1, \infty\right) \times \mathbb{R}^{2},  \\
u_{0}\left(1, x\right)=u_0(x), \quad \partial_{t} u_{0}\left(1, x\right)=u_1(x).
\end{cases}
\end{equation}
By Lemma \ref{th2-1} and the assumption on the initial data in Theorem \ref{thm:global existence-L}, one has
$$
M_{0} \leq C (\|u_0\|_{W^{\frac{m+3}{m+2}+\delta, 1} \left(\mathbb{R}^{2}\right)}+ \|u_1\|_{W^{\frac{m+1}{m+2}+\delta, 1} \left(\mathbb{R}^{2}\right)}) \leq C_{0} \varepsilon_0.
$$
For $j, k \geq 0$, we have
$$
\begin{cases}
\partial_{t}^{2}\left(u_{k+1}-u_{j+1}\right)-t^{m} \Delta\left(u_{k+1}-u_{j+1}\right)=V\left(u_{k}, u_{j}\right)\left(u_{k}-u_{j}\right), \\
\left(u_{k+1}-u_{j+1}\right)\left(1, x\right)=0, \quad \partial_{t}\left(u_{k+1}-u_{j+1}\right)\left(1, x\right)=0,
\end{cases}
$$
where
$$
\left|V\left(u_{k}, u_{j}\right)\right| \leq C t^{\alpha} \left(\left|u_{k}\right|+\left|u_{j}\right|\right)^{p-1}.
$$
By \eqref{equ:1.4}, one can see
$$
\gamma<\f{(m+1)p-(3+2\al)}{(m+2)(p+1)} \quad \text { and } \quad p \gamma>\frac{1}{p+1}.
$$
Then it follows from Lemma \ref{th2-1} and H\"{o}lder's inequality that
\begin{equation}\label{b-6}
\begin{aligned}
\| & \left(\left(\phi_{m}(t)+2\right)^{2}-|x|^{2}\right)^{\gamma} t^{\frac{\alpha}{p+1}}\left(u_{k+1}-u_{j+1}\right) \|_{L^{p+1}\left(\left[1, \infty\right) \times \mathbb{R}^{2}\right)} \\
\leq & C\left\|\left(\left(\phi_{m}(t)+2\right)^{2}-|x|^{2}\right)^{p \gamma} t^{-\frac{\alpha}{p+1}} V\left(u_{k}, u_{j}\right)\left(u_{k}-u_{j}\right)\right\|_{L^{\frac{p+1}{p}}\left(\left[1, \infty\right) \times \mathbb{R}^{2}\right)} \\
\leq & C\big\|\left(\left(\phi_{m}(t)+2\right)^{2}-|x|^{2}\right)^{\gamma} t^{\frac{\alpha}{p+1}}\left(\left|u_{k}\right|+\left|u_{j}\right|\right)\big\|_{L^{p+1}\left(\left[1, \infty\right] \times \mathbb{R}^{2}\right)}^{p-1}  \\
& \times\big\|\left(\left(\phi_{m}(t)+2\right)^{2}-|x|^{2}\right)^{\gamma} t^{\frac{\alpha}{p+1}}\left(u_{k}-u_{j}\right)\big\|_{L^{p+1}\left(\left[1, \infty\right) \times \mathbb{R}^{2}\right)} \\
\leq & C\left(M_{k}+M_{j}\right)^{p-1}\big\|\left(\left(\phi_{m}(t)+2\right)^{2}-|x|^{2}\right)^{\gamma} t^{\frac{\alpha}{p+1}}\left(u_{k}-u_{j}\right)\big\|_{L^{p+1}\left(\left[1, \infty\right) \times \mathbb{R}^{2}\right)} .
\end{aligned}
\end{equation}
By $M_{-1}=0$, \eqref{b-6} implies
$$
M_{k+1} \leq M_{0}+\frac{M_{k}}{2} \quad \text { for } \quad C M_{k}^{p-1} \leq \frac{1}{2},
$$
and
$$
M_{k} \leq 2 M_{0} \quad \text { if } \quad C\left(C_{0} \varepsilon\right)^{p-1} \leq \frac{1}{2}.
$$
Thus the boundedness of $\{M_k\}(k\in\Bbb N_0)$ is obtained when  $\varepsilon_0>0$ is small. Similarly, we have
$$
N_{k+1} \leq \frac{1}{2} N_{k}.
$$
Hence, there exists a function $u$ with $\big((\phi_{m}(t)+2)^{2}-|x|^{2}\big)^{\gamma} t^{\frac{\alpha}{p+1}} u \in L^{p+1}\left(\left[1, \infty\right) \times \mathbb{R}^{2}\right)$ such that $$\big(\left(\phi_{m}(t)+2\right)^{2}-|x|^{2}\big)^{\gamma} t^{\frac{\alpha}{p+1}} u_{k} \rightarrow\big(\left(\phi_{m}(t)+2\right)^{2}-|x|^{2}\big)^{\gamma} t^{\frac{\alpha}{p+1}} u \quad\text{in}~ L^{p+1}\left(\left[1, \infty\right) \times \mathbb{R}^{2}\right).$$
In addition, by $M_k\leq 2C_0\ve_0$ and direct computation, one has that for any compact set $K\subseteq \left[1, \infty\right)\times \R^2$,
$$
\begin{aligned}
& \left\|t^\alpha\left|u_{k+1}\right|^p-t^\alpha\left|u_k\right|^p\right\|_{L^{\frac{p+1}{p}}(K)} \\
& \leq C(K)\big\|(\left(\phi_m(t)+2\right)^2-|x|^2)^{p \gamma} t^{-\frac{\alpha}{p+1}}(t^\alpha|u_{k+1}|^p-t^\alpha|u_k|^p)\big\|_{L^{\frac{p+1}{p}}(K)} \\
& \leq C(K)\big\|((\phi_m(t)+2)^2-|x|^2)^\gamma t^{\frac{\alpha}{p+1}}(u_{k+1}-u_k)\big\|_{L^{p+1}(K)} \\
& \leq C(K) N_{k+1} \leq C 2^{-k} .
\end{aligned}
$$
This means $t^{\alpha}\left|u_{k}\right|^{p} \rightarrow t^{\alpha}|u|^{p}$ in $L_{l o c}^{1}\left(\left[1, \infty\right) \times \mathbb{R}^{2}\right)$. Therefore $u$ is a weak solution of \eqref{YH-4} in the sense of distributions.

On the other hand, in order to prove Theorem~\ref{thm:global existence-L}, it is required to derive
\begin{equation}\label{equ-1.3}
\left(1+\big|\phi_m^2(t)-|x|^2\big|\right)^{\gamma}\le C (\left(\phi_{m}(t)+2\right)^{2}-|x|^{2})^{\gamma}.
\end{equation}
In fact, due to $|x|\leq \phi_m(t)+1$, one then has that for $\phi^2_m(t)<|x|^2$,
$$
\begin{aligned}
1+\big|\phi_m^2(t)-|x|^2\big|&=1+\left(|x|-\phi_m(t) \right)\left(|x|+\phi_m(t) \right)\\
&\lesssim |x|+\phi_m(t) \\
&\lesssim(\phi_m(t)+2-|x|)(\phi_m(t)+2+|x|)\\
&= \left(\phi_{m}(t)+2\right)^{2}-|x|^{2}.
\end{aligned}
$$
While for $\phi^2_m(t)\geq|x|^2$,
$$
\begin{aligned}
\left(\phi_{m}(t)+2\right)^{2}-|x|^{2}&=\phi^2_m(t)-|x|^2+4\phi_m(t)+4\\
&\geq \phi^2_m(t)-|x|^2+1\\
&=1+\big|\phi^2_m(t)-|x|^2\big|.
\end{aligned}
$$
Thus \eqref{equ-1.3} and further Theorem~\ref{thm:global existence-L} are proved.
\end{proof}
\begin{proof}[Proof of Theorem \ref{TH-1}]

The proof procedure is divided into two parts.
\vskip 0.2 true cm

\textbf{Part 1. $0<\mu<1$}

\vskip 0.2 true cm
In this case, set $\mu=\frac{m}{m+2}$ and $T=\frac{2}{m+2} t^{\frac{m+2}{2}}$ with any $m \in(0, \infty)$, then
the equation in  \eqref{YH-4} can be rewritten as
\begin{equation}\label{b-7}
\partial_{T}^{2} u-\Delta u+\frac{m}{(m+2) T} \partial_{T} u=T^{\frac{2(\alpha-m)}{m+2}}|u|^{p}.
\end{equation}
Let $\alpha=m$, then \eqref{b-7} is equivalent to the equation in \eqref{00-1-2} with $\mu=\frac{m}{m+2}$.
Substituting $\alpha=m=\frac{2 \mu}{1-\mu}$ into \eqref{equ:conf} and the expression of $p_2(2,m,\alpha)$ yields that
$$
p_{conf} (2,m, \alpha)=\f{m+2\al+5}{m+1}=\frac{\mu+5}{\mu+1}=p_{conf}(2,\mu)
$$
and $p_{crit}(2,m, \alpha)=p_{s}(2+\mu)$.

Analogously, the condition \eqref{con1} in Theorem \ref{TH-1} comes from condition \eqref{equ:1.4} in Theorem \ref{thm:global existence-L}.
Thus Theorem \ref{TH-1} (i) with $\mu \in(0,1)$ follows immediately from Theorem \ref{thm:global existence-L}.
\vskip 0.2 true cm

\textbf{Part 2. $1<\mu<2$}

\vskip 0.2 true cm

For  $p \in\left(p_{s}(2+\mu), p_{conf }(2,\mu)\right)$, by $v(t, x)=(1+t)^{\mu-1} u(t, x)$,
the equation in \eqref{00-1-2} becomes
\begin{equation}\label{b-8}
\partial_{t}^{2} v-\Delta v+\frac{2-\mu}{1+t} \partial_{t} v=(1+t)^{(p-1)(1-\mu)}|v|^{p}.
\end{equation}
This means that the global existence result of \eqref{00-1-2} can be derived from \eqref{b-7} with $\mu=\frac{m+4}{m+2}$ and $\frac{2(\alpha-m)}{m+2}=(p-1)(1-\mu)$. In this case, one knows
\begin{equation}\label{b-9}
\alpha=1+m-p .
\end{equation}
By $1<p_{s}(2+\mu)<p<p_{conf}(2,\mu)$, \eqref{b-9} implies
$$
1+m-p_{conf }(2,\mu)<\alpha<1+m-p_{s}(2+\mu)<m.
$$
On the other hand, one has
\begin{equation}
p_{\text {conf }}(2, \mu)=\frac{\mu+5}{\mu+1}=\frac{3m+7}{m+3}.
\end{equation}
It follows from  direct computation that $1+m-p_{\text {conf }}(2,\mu)=m-\frac{2(m+2)}{m+3}>-1$ when $m>\sqrt{2}-1$.
Then the global existence in Theorem \ref{TH-1} (i) with $p\in\left(p_{s}(2+\mu), p_{conf}(2,\mu)\right)$ and $1<\mu<2\sqrt{2}-1$
can be proved.

While for the remaining part of $2\sqrt{2}-1 \leq \mu<2$, a direct computation yields that
$$
1+m-p_{s}(2+\mu)=1+\frac{2(2-\mu)}{\mu-1}-p_{s}(2+\mu)>-1
$$
is equivalent to
\begin{equation}\label{b-10}
p_{s}(2+\mu)<p<\frac{2}{\mu-1}\quad \text { for } \quad 2\sqrt{2}-1 \leq \mu<2.
\end{equation}
Therefore, the proof of Theorem \ref{TH-1} (ii) is completed.
\end{proof}

\section{Proofs of Theorem~\ref{thm:global existence-LL} and  Theorem~\ref{TH-2}}\label{E-2}

In this section, based on the estimate \eqref{q-25} and the restriction conditions \eqref{q-28}-\eqref{q-33},
we now prove Theorems~\ref{thm:global existence-LL} and  \ref{TH-2}.

\begin{proof}
[Proof of Theorem~\ref{thm:global existence-LL}]
Let $u_{-1}=0$ and  $u_{k}(k\in\Bbb N_0)$  solve
\begin{equation}\label{q-34}
\left\{ \enspace
\begin{aligned}
&\partial_t^2 u_k-t^m\Delta u_k =t^{m-p+1}|u_{k-1}|^p, &&
  (t,x)\in [1,\infty)\times \R^{2},\\
&u_k(1,x)=u_0(x), \quad \partial_{t} u_k(1,x)=u_1(x),
\end{aligned}
\right.
\end{equation}
where $0<m\leq \sqrt{2}-1$, $m+2\leq p<\f{3m+7}{m+3}$, and $(u_0,u_1)\in C_0^{\infty}(\Bbb R^2)$.

Denote
$$
\begin{aligned}
&M_{k}=\sum_{|\beta| \leq 1}(\|t^{\f{m-p+1}{p+1}}Z^{\beta} u\|_{L_{t}^{q} L_{r}^{\nu} L_{\theta}^{2}\left([1,+\infty)\times [0, +\infty)\times [0, 2\pi]\right)}+\|t^{\f{m-p+1}{p+1}}Z^{\beta} u\|_{L_{t}^{\infty} \dot{H}^{s}\left([1,\infty)\times \R^2\right)}), \\
&N_{k}=\|t^{\f{m-p+1}{p+1}}(u_{k}-u_{k-1})\|_{L_{t}^{q} L_{r}^{\nu} L_{\theta}^{2}\left([1,+\infty)\times [0, +\infty)\times [0, 2\pi]\right)},
\end{aligned}
$$
where $(q,\nu)$ satisfies \eqref{q-28}-\eqref{q-33}.

For $k=0$, it follows from \eqref{q-25} and \eqref{prop:1} that
$$
M_{0} \leq C_0\varepsilon_{0}.
$$

For $k \geq 1$, by \eqref{q-25} and $m-p+1+\f{p-m-1}{p+1}=p\cdot\f{m-p+1}{p+1}$, we can conclude
\begin{equation}\label{q-36}
M_{k} \leq C_0 \varepsilon_0+C\sum_{|\beta| \leq 1}\|Z^{\beta}(|t^{\f{m-p+1}{p+1}}u_{k-1}|^{p})\|_{L_{t}^{\tilde{q}^{\prime}}L_{r}^{\tilde{\nu}^{\prime}} L_{\theta}^{2}([1,+\infty)\times [0, +\infty)\times [0, 2\pi])}.
\end{equation}
Note that for $f(x)=f(r, \theta)\left(x \in \mathbb{R}^{2}\right)$ with $\sum_{|\beta| \leq 1}\left|Z^{\beta} f\right| \in L_{\theta}^{2}$,
one has
$$
\sum_{|\beta| \leq 1}|Z^{\beta}(|t^{\f{m-p+1}{p+1}}f|^{p})|\leq C|t^{\f{m-p+1}{p+1}}f|^{p-1}
\sum_{|\beta| \leq 1}|Z^{\beta}(t^{\f{m-p+1}{p+1}}f)|.
$$
Due to $\partial_{\theta}=x_{1} \partial_{2}-x_{2} \partial_{1} \in Z$, then
\begin{equation}\label{q-37}
\|t^{\f{m-p+1}{p+1}}f(r, \cdot)\|_{L_{\theta}^{\infty}}
\leq C \sum_{|\beta| \leq 1}\|Z^{\beta}(t^{\f{m-p+1}{p+1}} f(r, \cdot))\|_{L_{\theta}^{2}}.
\end{equation}
In addition,
\begin{equation}\label{q-38}
\begin{aligned}
\|\||t^{\f{m-p+1}{p+1}}u_{k-1}|^{p-1}& \sum_{|\beta| \leq 1}|Z^{\beta}(t^{\f{m-p+1}{p+1}} u_{k-1})|\|_{L_{\theta}^{2}}\|_{L_{t}^{\tilde{q}^{\prime}} L_{r}^{\tilde{\nu}^{\prime}}}\\
& \leq C\| \|t^{\f{m-p+1}{p+1}}u_{k-1}\|_{L_{\theta}^{\infty}}^{p-1} \sum_{|\beta| \leq 1}\| Z^{\beta}(t^{\f{m-p+1}{p+1}} u_{k-1})\|_{L_{\theta}^{2}}\|_{L_{t}^{\tilde{q}^{\prime}} L_{r}^{\tilde{\nu}^{\prime}}} \\
& \leq C\|(\sum_{|\beta| \leq 1}\|Z^{\beta}(t^{\f{m-p+1}{p+1}} u_{k-1})\|_{L_{\theta}^{2}})^{p-1} \sum_{|\beta| \leq 1}\| Z^{\beta}(t^{\f{m-p+1}{p+1}} u_{k-1})\|_{L_{\theta}^{2}}\|_{L_{t}^{\tilde{q}^{\prime}} L_{r}^{\tilde{\nu}^{\prime}}} \\
& \leq C\|(\sum_{|\beta| \leq 1}\|Z^{\beta}(t^{\f{m-p+1}{p+1}} u_{k-1})\|_{L_{\theta}^{2}})^{p}\|_{L_{t}^{\tilde{q}^{\prime}} L_{r}^{\tilde{\nu}^{\prime}}} \\
& \leq C(\sum_{|\beta| \leq 1}\|Z^{\beta}(t^{\f{m-p+1}{p+1}} u_{k-1})\|_{L_{t}^{q}L_{r}^{\nu}L_{\theta}^{2}})^{p}, \quad\text{where $q=p \tilde{q}^{\prime}$ and $\nu=p \tilde{\nu}^{\prime}$}.
\end{aligned}
\end{equation}
By the assumption $M_{k-1} \leq 2 C_{0} \varepsilon_0$ and the smallness of $\varepsilon_0$, we arrive at
\begin{equation}\label{q-39}
M_{k} \leq C_0 \varepsilon_0+C_1 M_{k-1}^{p}\leq C_{0} \varepsilon_0+C_{1} M_{k-1}^{p-1} M_{k-1}\leq 2 C_0 \varepsilon_0.
\end{equation}
Thus the boundedness of $\{M_k\}(k\in\Bbb N_0)$ is obtained when  $\varepsilon_0>0$ is small.

On the other hand, it follows from \eqref{q-39} and direct computation as in \eqref{q-38} that for small $\varepsilon_0>0$,
$$
N_{k} \leq CN_{k-1}(M_{k-1}+M_{k-2})^{p-1}\leq \frac{1}{2} N_{k-1}.
$$
Then there exists a function $u \in L_{t}^{q} L_{r}^{\nu} L_{\theta}^{2}$ such that $u_{k} \rightarrow u$ in $L_{t}^{q} L_{r}^{\nu} L_{\theta}^{2}$. Moreover,
$$
\|t^{m-p+1}(|u_k|^{p}-|u_{k-1}|^{p})\|_{L_{t}^{\tilde{q}^{\prime}} L_{r}^{\tilde{\nu}^{\prime}} L_{\theta}^{2}} \leq C\|t^{\f{m-p+1}{p+1}}( u_k-u_{k-1}) \|_{L_{t}^{q} L_{r}^{\nu} L_{\theta}^{2}} \rightarrow 0.
$$
This means $t^{m-p+1}|u_{k}|^{p} \rightarrow t^{m-p+1}|u|^{p}$ in $L_{\mathrm{loc}}^{1}\left([1,+\infty)\times [0, +\infty)\times [0, 2\pi]\right)$. Therefore, $u$ is a global weak solution of \eqref{q-2} and admits the regularities in Theorem \ref{thm:global existence-LL}.
\end{proof}

\begin{proof}
[Proof of Theorem~\ref{TH-2}]

As in Part 2 of Section \ref{E-1}, Theorem~\ref{TH-2} can be derived from Theorem~\ref{thm:global existence-LL}.
with $m=\f{2(2-\mu)}{\mu-1}$.

Indeed, it follows from direct computations that for $m=\f{2(2-\mu)}{\mu-1}$,
\begin{equation}\label{c11}
m+2\leq p<\f{3m+7}{m+3}\Leftrightarrow \f{2}{\mu-1}\leq p< \f{\mu+5}{\mu+1}=p_{conf}(2,\mu),
\end{equation}
\begin{equation}
s=1-\f{2(m-p+1)+4}{(m+2)(p-1)}=1-\frac{\mu+1-p(\mu-1)}{p-1}, \quad t^{\f{m-p+1}{p+1}}=t^{\f{3-\mu-p(\mu-1)}{(\mu-1)(p+1)}}.
\end{equation}
In addition, the restriction conditions \eqref{q-28}-\eqref{q-33} can be rewritten as
\begin{equation}\label{WS}
\begin{cases}\f{1}{\nu}+\f{3-\mu-p(\mu-1)}{2(p+1)}>0,\\
\f{1}{q}+\f{2}{\nu(\mu-1)}=\f{\mu+1-p(\mu-1)}{(p-1)(\mu-1)}-\f{3-\mu-p(\mu-1)}{(p+1)(\mu-1)},\\
\f{1}{q}+\f{1}{\nu(\mu-1)}\leq\f{3-\mu}{2(\mu-1)}-\f{3-\mu-p(\mu-1)}{(p+1)(\mu-1)},\\
2\leq q \leq 2p,\\
2\leq \nu\leq 2p,\\
\f{1}{q}+\f{1}{\nu(\mu-1)}\geq\f{3}{2p}+\f{3-\mu-p(\mu-1)}{p(p+1)(\mu-1)}.
\end{cases}
\end{equation}

Next we look for the suitable values of $\nu$ in \eqref{WS}. Recall that when $m_1=0.048<m\leq\sqrt{2}-1$,
the positive root ${\bar p}^*(m)$  of \eqref{c7} satisfies ${\bar p}^*(m)<m+2$ due to
\begin{equation}\label{c9}
\begin{aligned}
&(3m+9)(m+2)^3-(2m+3)(m+2)^2-(11m+25)(m+2)-(6m+13)\\
&=(m+3)(3m^3+16m^2+20m-1)>0.
\end{aligned}
\end{equation}

Substituting $m=\f{2(2-\mu)}{\mu-1}$ into \eqref{c9} yields
\begin{equation}\label{c10}
\begin{aligned}
&(\f{2(2-\mu)}{\mu-1}+3)\left(3(\f{2(2-\mu)}{\mu-1})^3+16(\f{2(2-\mu)}{\mu-1})^2+20\f{2(2-\mu)}{\mu-1}-1\right)\\
&=-\f{(\mu+1)(\mu^3+13\mu^2-21\mu-17)}{(\mu-1)^4}>0.
\end{aligned}
\end{equation}
By direct computation, we have that when $2\sqrt{2}-1\leq\mu<\mu_1=1.977$, \eqref{c10} holds.
As discussed in Section \ref{E-5},  $\nu$ can be determined as follows:  if  $2\sqrt{2}-1\leq\mu<\mu_1$, $\nu=p+\f{1}{3}$
is taken for $\f{2}{\mu-1}\leq p<p_{conf}(2,\mu)$; if $\mu_1\leq\mu<2$, either $\nu=p$ for $\f{2}{\mu-1}\leq p<p^*(\mu)$
or $\nu=p+\f{1}{3}$ for $p^*(\mu)\leq p<p_{conf}(2,\mu)$ are taken, where $p^*(\mu)={\bar p}^*(\f{2(2-\mu)}{\mu-1})$.
Therefore, Theorem \ref{TH-2}  is shown.

\end{proof}

\appendix
\setcounter{equation}{1}
\section{Some useful estimates}

\begin{lemma}\label{A1}
\eqref{a-52} holds.
\end{lemma}
\begin{proof}
Note that $H=\sum_{j=-\infty}^{\infty} H_{j}$ with
\begin{equation}\label{a-53}
H_{j}=\int_{\mathbb{R}^{2}} e^{i\left\{x \cdot \xi-\left[\phi_{m}(t)-\phi_{m}(s)\right]|\xi|\right\}} \vp\left(\frac{|\xi|}{2^{j}}\right) a(t, s, \xi) \hat{F}(s, \xi) \mathrm{d} \xi.
\end{equation}
It follows from $1 \leq s \leq 2$ and \eqref{a-51} that
\begin{equation}\label{a-54}
\begin{aligned}
\big|\partial_{\xi}^{\beta} a(t, s, \xi)\big| \leq C \phi_{m}(t)^{-\frac{m}{2(m+2)}}s^{-\frac{\alpha}{q_{0}}}|\xi|^{-1-|\beta|}.
\end{aligned}
\end{equation}
Set
$$
T_{j} F(t, s, x)=: t^{\frac{\alpha}{q_{0}}} H_{j}(t, s, x) \quad \text { with } \lambda_{j}=2^{j},
$$
then it follows from stationary phase method that
$$
\left\|T_{j} F(t, s, \cdot)\right\|_{L^{2}\left(\mathbb{R}^{2}\right)} \leq C \lambda_{j}^{-1} t^{\frac{4\alpha-mq_0}{4q_{0}}}\|s^{-\frac{\alpha}{q_{0}}} F(s, \cdot)\|_{L^{2}\left(\mathbb{R}^{2}\right)}
$$
and
$$
\left\|T_{j} F(t, s, \cdot)\right\|_{L^{\infty}\left(\mathbb{R}^{2}\right)} \leq C \lambda_{j}^{\frac{1}{2}} t^{\frac{4\alpha-mq_0}{4q_{0}}}|t-s|^{-\frac{m+2}{4}}\|s^{-\frac{\alpha}{q_{0}}} F(s, \cdot)\|_{L^{1}\left(\mathbb{R}^{2}\right)}.
$$
By interpolation and direct computation, one has that for $j \geq 0$,
\begin{equation}\label{a-56}
\begin{aligned}
\| T_{j}F(t, s, \cdot) & \|_{L^{q_{0}}\left(\mathbb{R}^{2}\right)}\leq C \lambda^{\frac{\alpha-2m}{2(m+3+\alpha)}}|t-s|^{-\frac{2}{q_{0}}\left(1+\frac{m}{4}\right)}\|s^{-\frac{\alpha}{q_{0}}}
F(s, \cdot)\|_{L^{\frac{q_{0}}{q_{0}-1}}\left(\mathbb{R}^{2}\right)} .
\end{aligned}
\end{equation}
Define
$$
T_{+} F(t, s, x) \equiv  \sum_{j \geq 0} T_{j}F(t, s, x),
$$
then
\begin{equation}\label{a-57}
\left\|T_{+} F(t, s, \cdot)\right\|_{L^{q_{0}}\left(\mathbb{R}^{2}\right)}\leq C|t-s|^{-\frac{2}{q_{0}}\left(1+\frac{m}{4}\right)}\|s^{-\frac{\alpha}{q_{0}}} F(s, \cdot)\|_{L^{\frac{q_{0}}{q_{0}-1}}\left(\mathbb{R}^{2}\right)}.
\end{equation}
While for $j<0$, let
$$
T_{-} F(t, s, x) \equiv  \sum_{j<0}  T_{j}F(t, s, x).
$$
Then by  \eqref{a-51} and the condition $t\geq t-s \gtrsim t$, we obtain that for $1 \leq s \leq 2$,
$$
\begin{aligned}
\|T_{-} F(t, s, \cdot)\|_{L^{2}(\mathbb{R}^{n})} &\leq C(\int_{|\xi| \leq 1}| | \xi|^{-\frac{2}{m+2}} t^{\frac{\alpha}{q_{0}}}(1+\phi_{m}(t)|\xi|)^{-\frac{m}{2(m+2)}} \hat{F}(s, \xi)|^{2} \mathrm{~d} \xi)^{\frac{1}{2}}\\
&\leq C t^{\frac{\alpha}{q_{0}}}(\int_{0}^{1} r^{1-\frac{m+4}{m+2}} \phi_{m}(t)^{-\frac{m}{m+2}} \mathrm{d} r)^{\frac{1}{2}}\|s^{-\frac{\alpha}{q_{0}}} F(s, \cdot)\|_{L^{2}\left(\mathbb{R}^{2}\right)}\\
& \leq C|t-s|^{\frac{\alpha}{q_0}-\frac{m}{4}}\|s^{-\frac{\alpha}{q_{0}}}F(s, \cdot)\|_{L^2\left(\mathbb{R}^2\right)}.
\end{aligned}
$$
Similarly,
$$
\left\|T_{-} F(t, s, \cdot)\right\|_{L^{\infty}\left(\mathbb{R}^{2}\right)}
\leq C|t-s|^{\frac{\alpha}{q_{0}}-\frac{m}{4}- \frac{m+2}{4}}\|s^{-\frac{\alpha}{q_{0}}}F(s, \cdot)\|_{L^{1}\left(\mathbb{R}^{2}\right)}.
$$
By interpolation, one arrives at
\begin{equation}\label{a-58}
\left\|T_{-} F(t, s, \cdot)\right\|_{L^{q_{0}}\left(\mathbb{R}^{2}\right) }\leq C|t-s|^{-\frac{2}{q_{0}}\left(1+\frac{m}{4}\right)}
\|s^{-\frac{\alpha}{q_{0}}}F(s, \cdot)\|_{L^{\frac{q_{0}}{q_{0}-1}}\left(\mathbb{R}^{2}\right)}.
\end{equation}
Thus, combining \eqref{a-57}-\eqref{a-58} with Littlewood-Paley theory yields \eqref{a-52}.
\end{proof}

\begin{lemma}\label{A2}
\eqref{a-70} holds.
\end{lemma}
\begin{proof}
Note that for $z=\frac{2(m+3+ \alpha)}{(m+2)(\alpha+2)}+i \theta$ with $\theta \in \mathbb{R}$
and $-1<\alpha<0$,  there exists $\sigma>0$ such that
$$
\frac{2(m+3+ \alpha)}{(m+2)(\alpha+2)}-\sigma>\frac{m+3}{m+2}.
$$
Then for $|\xi| > 1$, one obtains
\begin{equation}\label{a-72-1}
\begin{aligned}
& \big|\theta e^{-\theta^{2}} \int_{\frac{1}{2} \phi_{m}(1) \leq|y| \leq \phi_{m}(2)} \int_{|\xi| \geq 1} e^{i\left[(x-y) \cdot \xi-\left(\phi_{m}(t)-|y|\right)|\xi| \mid\right.} t^{\frac{\alpha}{q_{0}}}\left(1+\phi_{m}(t)|\xi|\right)^{-\frac{m}{2(m+2)}} g(y) \frac{\mathrm{d} \xi}{|\xi|^{z}} \mathrm{~d} y\big| \\
& \leq C\left(\phi_{m}(t)-\phi_{m}(2)\right)^{-\frac{1}{2}} \phi_{m}(t)^{-\frac{m}{2(m+2)}}t^{\frac{\alpha}{q_{0}}} \int_{|\xi| \geq 1}|\xi|^{-\frac{1}{2}-\frac{m}{2(m+2)} }|\xi|^{-\frac{m+3}{m+2}-\sigma}\mathrm{d} \xi\|g\|_{L^{1}\left(\mathbb{R}^{2}\right)} \\
& \leq C \phi_{m}(t)^{-\frac{1}{2}-\frac{m}{2(m+2)} }t^{\frac{\alpha}{q_{0}}}\|g\|_{L^{1}\left(\mathbb{R}^{2}\right)} .
\end{aligned}
\end{equation}
When $|\xi| \leq 1$ and $\alpha>-1$, by setting $\tilde{\sigma}=2-\frac{2(m+3+\alpha)}{(m+2)(\alpha+2)}>0$, we have
\begin{equation}\label{a-73-1}
\begin{aligned}
& \big|\theta e^{-\theta^{2}} \int_{\mathbb{R}^{2}}\int_{|\xi| \leq 1} e^{i\left[(x-y) \cdot \xi-\left(\phi_{m}(t)-|y|\right)|\xi|\right.} t^{\frac{\alpha}{q_{0}}}\left(1+\phi_{m}(t)|\xi|\right)^{-\frac{m}{2(m+2)}} g(y) \frac{\mathrm{d} \xi}{|\xi|^{z}} \mathrm{~d} y\big| \\
& \leq C\|g\|_{L^{1}\left(\mathbb{R}^{2}\right)} \int_{|\xi| \leq 1} t^{\frac{\alpha}{q_{0}}}\left(1+\phi_{m}(t)|\xi|\right)^{-\frac{1}{2}-\frac{m}{2(m+2)}}|\xi|^{-\frac{2(m+3+\alpha)}{(m+2)(\alpha+2)}} \mathrm{d} \xi \\
& \leq C\|g\|_{L^{1}\left(\mathbb{R}^{2}\right)} \int_{0}^{1}t^{\frac{\alpha}{q_{0}}} \left(1+\phi_{m}(t) r\right)^{-\frac{m+1}{m+2}}r^{-\frac{2(m+3+ \alpha)}{(m+2)(\alpha+2)}} r \mathrm{d} r\\
&=C t^{\frac{\alpha}{q_{0}}}\|g\|_{L^{1}\left(\mathbb{R}^{2}\right)} \int_{0}^{1} \left(1+\phi_{m}(t) r\right)^{-\frac{m+1}{m+2}}r^{-1+\tilde{\sigma}}  \mathrm{d} r\\
&\leq C t^{\frac{\alpha}{q_{0}}}\|g\|_{L^{1}\left(\mathbb{R}^{2}\right)}\big( \tilde{\sigma}^{-1} \phi_m(t)^{-\frac{m+1}{m+2}}+\frac{m+1}{ \tilde{\sigma}(m+2)} \int_0^1 \left(1+\phi_m(t) r\right)^{-\frac{m+1}{m+2}} r^{-1+\tilde{\sigma}} \mathrm{d} r\big)\\
& \leq C \phi_{m}(t)^{-\f{m+1}{m+2} }t^{\frac{\alpha}{q_{0}}}\|g\|_{L^{1}\left(\mathbb{R}^{2}\right)}.
\end{aligned}
\end{equation}
This, together with \eqref{a-72-1}, derives \eqref{a-70}.
\end{proof}

\begin{lemma}\label{A3}
\eqref{a-71} holds.
\end{lemma}
\begin{proof}
Let $\rho \in C^{\infty}\left(\mathbb{R}^{2}\right)$ satisfy
$$
\rho(\xi)= \begin{cases}1, & |\xi| \geq 2, \\ 0, & |\xi| \leq 1.\end{cases}
$$
For $\zeta=1+\nu$, as in Section 4B3 of \cite{HWY4}, set
$$
\left(T_{z} g\right)(t, x)=\left(R_{z} g\right)(t, x)+\left(S_{z} g\right)(t, x),
$$
where
$$
\begin{aligned}
\big(R_{z} g\big)(t, x)=&\big(z-\frac{2(m+3+\alpha)}{(m+2)(\alpha+2)}\big) e^{z^{2}} \int_{\frac{1}{2} \phi_{m}(1) \leq|y| \leq \phi_{m}(2)} \int_{\mathbb{R}^{2}} e^{i\left((x-y) \cdot \xi-\left(\phi_{m}(t)-|y|\right)|\xi|\right)}\\
& \times t^{\frac{\alpha}{q_0}}\left(1+\phi_m(t)|\xi|\right)^{-\frac{m}{2(m+2)}}\left(1-\rho\left(\phi_m(t)^{1-\zeta} \delta^{\zeta} \xi\right)\right) g(y) \frac{\mathrm{d} \xi}{|\xi|^z} \mathrm{d} y,
\end{aligned}
$$
$$
\begin{aligned}
\left(S_{z} g\right)(t, x)=&\big(z-\frac{2(m+3+\alpha)}{(m+2)(\alpha+2)}\big) e^{z^{2}} \int_{\frac{1}{2} \phi_{m}(1) \leq|y| \leq \phi_{m}(2)} \int_{\mathbb{R}^{2}} e^{i\left((x-y) \cdot \xi-\left(\phi_{m}(t)-|y|\right)|\xi|\right)}\\
& \times t^{\frac{\alpha}{q_0}}\left(1+\phi_m(t)|\xi|\right)^{-\frac{m}{2(m+2)}}\rho\left(\phi_m(t)^{1-\zeta} \delta^{\zeta} \xi\right)g(y) \frac{\mathrm{d} \xi}{|\xi|^z} \mathrm{d} y.
\end{aligned}
$$
By Lemma A. 3 and Lemma A. 4 in \cite{HWY4} (the related results are independent of the space dimensions), one can obtain
\begin{equation}\label{a-76}
\left\|\left(R_z g\right)(t, \cdot)\right\|_{L^2\left(\mathbb{R}^2\right)} \leq C \phi_m(t)^{\frac{2 \alpha}{(m+2) q_0}-\frac{m}{2(m+2)}}\left(\phi_m(t)^{\zeta-1} \delta^{-\zeta}\right)^{\frac{1}{m+2}}\|g\|_{L^2\left(\mathbb{R}^2\right)}
\quad \text { for Re $z=0$ }
\end{equation}
and
\begin{equation}\label{a-77}
\begin{aligned}
& \left\|S_{z} g(t, \cdot)\right\|_{L^{2}\left(\left\{x: \delta \phi_{m}(1) \leq \phi_{m}(t)-|x| \leq 2 \delta \phi_{m}(1)\right\}\right)} \\
& \quad \leq C \phi_{m}(t)^{\frac{2 \alpha}{(m+2) q_{0}}-\frac{m}{2(m+2)}}\left(\phi_{m}(t)^{\zeta-1} \delta^{-\zeta}\right)^{\frac{1}{m+2}}\|g\|_{L^{2}\left(\mathbb{R}^{2}\right)} \quad \text { for Re  $z=0$.}
\end{aligned}
\end{equation}
Collecting \eqref{a-76} and \eqref{a-77} yields \eqref{a-71}.
\end{proof}

\begin{lemma}\label{A4}
\eqref{a-96} holds.
\end{lemma}
\begin{proof}
Note that
$$
\|t^{\frac{\alpha}{2}} v_{02}\|_{L^{2}\left(D_{t, x}^{T, \delta}\right)} \leq\|\int\| \check{T} F\|_{L^{2}\left(\left\{x: \delta\phi_m(1) \leq \phi_{m}(t)-|x| \leq 2 \delta\phi_m(1)\right\}\right)} d s\|_{L^{2}\left(\left\{t: \frac{T}{2} \leq t \leq T\right\}\right)},
$$
where
$$
\check{T} F=\iint_{|\xi| \geq 1} e^{i\left((x-y) \cdot \xi-\left(\phi_{m}(t)-\phi_{m}(s)\right)|\xi|\right)} \phi_{m}(t)^{-\frac{m}{2(m+2)}} t^{\frac{\alpha}{2}} \frac{\rho_1(\delta \xi)}{|\xi|} F(s, y) \mathrm{d} y \mathrm{d} \xi.
$$
It follows from Lemma A.5  in \cite{HWY4} (the result in Lemma A.5 is independent of the spaced dimensions) and direct computation that
$$
\begin{aligned}
& \phi_{m}(T)^{-\frac{1}{2}+\frac{m-\alpha}{m+2}-\frac{\nu}{2}} \delta^{-\frac{1}{2}+\frac{m-\alpha}{m+2}+\frac{\nu}{2}}\|t^{\frac{\alpha}{2}} v_{02}\|_{L^{2}\left(D_{t, x}^{T, \delta}\right)} \\
&\leq C \phi_m(T)^{-\frac{1}{2}+\frac{m-\alpha}{m+2}-\frac{\nu}{2}} \delta^{-\frac{1}{2}+\frac{m-\alpha}{m+2}+\frac{\nu}{2}} \delta^{\frac{1}{2}} T^{\frac{1}{2}} \phi_m(T)^{\frac{\alpha}{m+2}-\frac{m}{2(m+2)}}\\
&\qquad\qquad\qquad\times\big(\sum_{j=0}^{\infty}\big\|\iiint_{2^{j} \leq|\xi| \leq 2^{j+1}} e^{\left.i\left((x-y) \cdot \xi-\phi_{m}(s)\right)|\xi|\right)} \frac{\rho_1(\delta \xi)}{|\xi|^{\frac{1}{2}}}s^{-\f{\al}{2}} F(s, y) \mathrm{d} y \mathrm{d} \xi \mathrm{d} s\big\|_{L_{x}^{2}}\big)\\
& = C \delta^{\frac{m-\alpha}{m+2}}\left(\frac{\delta}{\phi_{m}(T)}\right)^{\frac{\nu}{2}}\big(\sum_{j=0}^{\infty}\big\|\iiint_{2^{j} \leq|\xi| \leq 2^{j+1}} e^{\left.i\left((x-y) \cdot \xi-\phi_{m}(s)\right)|\xi|\right)} \frac{\rho_1(\delta \xi)}{|\xi|^{\frac{1}{2}}}s^{-\f{\al}{2}} F(s, y) \mathrm{d} y \mathrm{d} \xi \mathrm{d} s\big\|_{L_{x}^{2}}\big).
\end{aligned}
$$
Analogous to the proofs of (5-11)-(5-14) in \cite{HWY4}, we can obtain
$$
\begin{aligned}
& \phi_{m}(T)^{-\frac{1}{2}+\frac{m-\alpha}{m+2}-\frac{\nu}{2}} \delta^{-\frac{1}{2}+\frac{m-\alpha}{m+2}+\frac{\nu}{2}}\|t^{\frac{\alpha}{2}} v_{02}\|_{L^{2}(D_{t, x}^{T, \delta})} \\
& \leq C \delta^{\frac{m-\alpha}{m+2}}\left(\frac{\delta}{\phi_{m}(T)}\right)^{\frac{\nu}{2}}(1+|\ln \delta|) \delta_{0}^{\frac{1}{2}}\|s^{-\frac{\alpha}{2}} F\|_{L^{2}([1, 2]\times \mathbb{R}^{2})} \\
& \leq C \delta_{0}^{\frac{1}{2}}\|s^{-\frac{\alpha}{2}} F\|_{L^{2}([1, 2]\times \mathbb{R}^{2})}.
\end{aligned}
$$
Thus \eqref{a-96} is proved.
\end{proof}

\vskip 0.2 true cm

{\bf Acknowledgements}. Yin Huicheng wishes to express his deep gratitude to Professor Ingo Witt
(University of G\"ottingen, Germany) for his constant interests in this problem and many fruitful discussions in the past.
In addition, the authors would like to thank Dr. He Daoyin  and Prof. Ruan Zhuoping for their invaluable discussions.

\vskip 0.2 true cm
{\bf \color{blue}{Conflict of Interest Statement:}}

\vskip 0.2 true cm

{\bf The authors declare that there is no conflict of interest in relation to this article.}

\vskip 0.2 true cm
{\bf \color{blue}{Data availability statement:}}

\vskip 0.2 true cm

{\bf  Data sharing is not applicable to this article as no data sets are generated
during the current study.}

\vskip 0.2 true cm



\begin{thebibliography}{99}


\bibitem{Chen-Mei} Chen Shaohua, Li Haitong, Li Jingyu, Mei Ming, Zhang Kaijun,
Global and blow-up solutions for compressible Euler equations with time-dependent damping,
J. Differential Equations \textbf{268} (2020), no. 9, 5035--5077.


\bibitem{DA}
 M. D'Abbicco, The threshold of effective damping for semilinear wave equations,
 Math. Methods Appl. Sci. \textbf{38} (6) (2015), 1032--1045.

\bibitem{DA-0}  M. D'Abbicco, Small data solutions for the Euler-Poisson-Darboux equation with a
power nonlinearity,
J. Differential Equations \textbf{286} (2021), 531-556.




\bibitem{Rei2} M. D'Abbicco, S. Lucente, NLWE with a special scale invariant damping in odd space dimension,
in: Discrete Contin. Dyn. Syst. 2015, Dynamical Systems, Differential Equations and Applications.
10th AIMS Conference. Suppl., pp. 312--319.


\bibitem{Rei1} M.~D'Abbicco, S.~Lucente, M.~Reissig, A
shift in the Strauss exponent for semilinear wave equations with a
not effective damping, J. Differential Equations \textbf{259} (2015),
5040-5073.

\bibitem{Erd1} A. Erdelyi, W. Magnus, F. Oberhettinger, F. G. Tricomi,
Higher transcendental functions,
Vol.1, New York, Toronto and London: McGraw-Hill Book Company, 1953.



\bibitem{Fuj} H.~Fujita, On the blowing up of solutions of the
Cauchy problem for $u_t=\Delta u+u^{1+\al}$,
J. Fac. Sci. Univ. Tokyo Sect. \textbf{13} (1966), 109--124.




\bibitem{Gls1} V.~Georgiev, H.~Lindblad, C.D.~Sogge,
Weighted Strichartz estimates and global existence for
semi-linear wave equations, Amer. J. Math. \textbf{119} (1997),
1291--1319.


\bibitem{HLWY} He Daoyin, Li Qianqian, Yin Huicheng,
Global small data weak solutions of 2-D semilinear wave equations
with  scale-invariant  damping, II, arXiv:2503.19438, Preprint (2025).


\bibitem{HLWY-1} He Daoyin, Li Qianqian, Yin Huicheng,
Global small data weak solutions of 2-D semilinear wave equations
with  scale-invariant  damping, IV, Preprint (2025).

\bibitem{HWY1}
He Daoyin,  Witt Ingo, Yin Huicheng, On the global solution problem for semilinear generalized Tricomi equations, I,
Calc. Var. Partial Differential Equations \textbf{56} (2017), no. 2, Paper No. 21, 24 pp.

\bibitem{HWY2}
He Daoyin, Witt Ingo, Yin Huicheng, On semilinear Tricomi equations with critical exponents or in two space dimensions,
J. Differential Equations \textbf{263} (2017), no. 12, 8102--8137.

\bibitem{HWY3}
He Daoyin, Witt Ingo, Yin Huicheng, On the Strauss index of semilinear Tricomi equation,
Commun. Pure Appl. Anal. \textbf{19} (2020), no. 10, 4817--4838.

\bibitem{HWY4} He Daoyin, Witt Ingo, Yin Huicheng, On the
  global solution problem of semilinear generalized Tricomi equations,
  II, Pacific J. Math. \textbf{314} (2021), no. 1, 29--80.




 \bibitem{Hou-1} Hou Fei, Witt Ingo, Yin Huicheng, Global existence or blowup of smooth solutions to 3-D
 compressible Euler equations with time-dependent damping,  Pacific Journal of Mathematics, Vol. \textbf{292} (2018), No. 2, 389--426.


 \bibitem{Hou-2} Hou Fei, Yin Huicheng, On the global existence and blowup of smooth solutions to the multi-dimensional
 compressible Euler equations with time-depending damping, Nonlinearity, Vol. \textbf{30} (2017), No. 6, 2485--2517.


\bibitem{IS}
M. Ikeda, M. Sobajima, Life-span of solutions to semilinear wave equation with time-dependent critical
damping for specially localized initial data, Math. Ann. \textbf{372} (3-4) (2018), 1017--1040.

 \bibitem{Imai} T. Imai, M. Kato, H. Takamura, K. Wakasa,
 The lifespan of solutions of semilinear wave equations with the scale-invariant damping in two space dimensions,
J. Differential Equations \textbf{269} (2020), no. 10, 8387--8424.



\bibitem{LTW}
 Lai Ning-An, H. Takamura, K. Wakasa, Blow-up for semilinear wave equations with the scale invariant damping and
 super-Fujita exponent, J. Differential Equations \textbf{263} (9) (2017), 5377--5394.

\bibitem{LZ} Lai Ning-An, Zhou Yi,
Global existence for semilinear wave equations with scaling invariant damping in 3-D,
Nonlinear Anal. \textbf{210} (2021), Paper No. 112392, 12 pp.


\bibitem{LY} Li Qianqian, Yin Huicheng,
Global small data weak solutions of 2-D semilinear wave equations
with  scale-invariant  damping, III, arXiv.6614053, Preprint (2025).


\bibitem{LS} H.~Lindblad, C.D.~Sogge, On existence and
  scattering with minimal regularity for semilinear wave equations,
  J. Funct. Anal. \textbf{130} (1995), 357--426.

\bibitem{Mei-2} Luan Liping, Mei Ming, Rubino Bruno, Zhu Peicheng,
Large-time behavior of solutions to Cauchy problem for bipolar Euler-Poisson system
with time-dependent damping in critical case,
Commun. Math. Sci. \textbf{19} (2021), no. 5, 1207--1231.



\bibitem{PR}
A. Palmieri, M. Reissig, A competition between Fujita and Strauss type exponents for blow-up of semi-linear wave
equations with scale-invariant damping and mass, J. Differential Equations \textbf{266} (2-3) (2019), 1176--1220.




\bibitem{Rua3} Ruan Zhuoping, Witt Ingo,  Yin Huicheng, On
the existence of low regularity solutions to semilinear generalized
Tricomi equations in mixed type domains, J. Differential Equations
\textbf{259} (2015), 7406--7462.

\bibitem{Rua4}  Ruan Zhuoping, Witt Ingo, Yin Huicheng, Minimal regularity
solutions of semilinear generalized Tricomi equations,  Pacific J. Math. \textbf{296} (2018), no. 1, 181--226.



\bibitem{SSW}
H. F. Smith, C. D. Sogge, C. Wang, Strichartz estimates for Dirichlet-wave equations in two dimensions with applications,
Trans. Amer. Math. Soc. \textbf{364} (2012), 3329--3347.




\bibitem{SK}
A. Stefanov, P. G. Kevrekidis, Asymptotic behaviour of small solutions for the discrete nonlinear Schr\"odinger and
Klein-Gordon equations, Nonlinearity \textbf{364} (2005), 1841--1857.


\bibitem{S}
E. M. Stein, Interpolation of linear operators, Trans. Amer. Math. Soc. \textbf{83} (1956), 482--492

\bibitem{Strauss} W.~Strauss, Nonlinear scattering theory at
  low energy, J. Funct. Anal. \textbf{41} (1981), 110--133.

\bibitem{Tani} K. Taniguchi, Y. Tozaki,  A hyperbolic equation with double characteristics which
has a solution with branching singularities,  Math. Japon \textbf{25}  (1980), 279--300.


\bibitem{TL1}
Tu Ziheng, Lin Jiayun, A note on the blowup of scale invariant damping wave equation with sub-Strauss exponent,
arXiv preprint arXiv:1709.00866, 2017.






\bibitem{TL2} Tu Ziheng, Lin Jiayun, Life-span of semilinear wave equations with scale-invariant damping: critical
Strauss exponent case,
Differential Integral Equations \textbf{32} (2019), no. 5-6, 249--264.



\bibitem{W1}
 Y. Wakasugi, Critical exponent for the semilinear wave equation with scale invariant damping, in: Fourier Analysis,
Birkh\"auser, Cham, 2014, pp. 375--390.

 \bibitem{Wirth-1}
J. Wirth,  Wave equations with time-dependent dissipation. I. Non-effective dissipation,
J. Differential Equations \textbf{222} (2006), no. 2, 487--514.


 \bibitem{Wirth-2}
J. Wirth, Wave equations with time-dependent dissipation. II. Effective dissipation,
J. Differential Equations \textbf{232} (2007), no. 1, 74--103.



\bibitem{Xu-Yin-1} Xu Gang, Yin Huicheng,
On global multidimensional supersonic flows with vacuum states at infinity,
Arch. Ration. Mech. Anal. \textbf{218} (2015), No. 3, 1189--1238.

\bibitem{Xu-Yin-2} Xu Gang, Yin Huicheng,
3-D global supersonic Euler flows in the infinite long divergent nozzles, SIAM J. Math. Anal. \textbf{53} (2021), no. 1, 133--180.

\bibitem{Yag2} K.Yagdjian, Global existence for the
  $n-$dimensional semilinear Tricomi-type equations, Comm. Partial
  Diff. Equations \textbf{31} (2006), 907--944.










\end{thebibliography}
\end{document}